\documentclass[reqno,11pt]{amsart}
\usepackage[colorlinks=true, linkcolor=blue, citecolor=blue]{hyperref}
    
\usepackage{amssymb}
\usepackage{amsmath, graphicx, rotating}
\usepackage{color}
\usepackage{soul}
\usepackage[dvipsnames]{xcolor}
           
\usepackage{ifthen}
\usepackage{xkeyval}
\usepackage{todonotes}  
\usepackage[T1]{fontenc}
\usepackage{lmodern}
\usepackage[english]{babel}

\usepackage{ upgreek }
\usepackage{stmaryrd}
\SetSymbolFont{stmry}{bold}{U}{stmry}{m}{n}
\usepackage{amsthm}
\usepackage{float}

\usepackage{mathtools}

\usepackage{ bbm }
\usepackage{ stmaryrd }
\usepackage{ mathrsfs }
\usepackage{ frcursive }
\usepackage{ comment }

\usepackage{pgf, tikz}
\usetikzlibrary{shapes}
\usepackage{varioref}
\usepackage{enumitem}

\definecolor{rouge}{rgb}{0.7,0.00,0.00}
\definecolor{vert}{rgb}{0.00,0.5,0.00}
\definecolor{bleu}{rgb}{0.00,0.00,0.8}

\usepackage[textheight=7.61in,textwidth=4.98in]{geometry} 
\newtheorem{theorem}{Theorem}[section]
\newtheorem*{theorem*}{Theorem}
\newtheorem{lemma}[theorem]{Lemma}

\newtheorem{proposition}[theorem]{Proposition}

\labelformat{hypothesis}{\textbf{M\kern-0.1mm#1}}

\newtheorem{condition}{Condition}
\newtheorem{conditionA}{A\kern-0.1mm}
\labelformat{conditionA}{\textbf{A\kern-0.1mm#1}}

\renewcommand\dots{\hbox to 1em{.\hss.\hss.}}

\theoremstyle{definition}

\numberwithin{equation}{section}

\def\bb#1{\mathbb{#1}}
\def\bf#1{\mathbf{#1}}
\def\bbm#1{\mathbbm{#1}}
\def\geq{\geqslant}
\def\leq{\leqslant}

\newcommand\ee{\varepsilon}
\DeclareMathOperator{\supp}{supp}

\DeclarePairedDelimiter\floor{\lfloor}{\rfloor}

\begin{document}

\title[Large deviations for the coefficients]
{Large deviation expansions for the coefficients of random walks on \\ the general linear group}


\author{Hui Xiao}
\author{Ion Grama}
\author{Quansheng Liu}


\curraddr[Xiao, H.]{Universit\'{e} de Bretagne-Sud, LMBA UMR CNRS 6205, Vannes, France}
\email{hui.xiao@univ-ubs.fr}
\curraddr[Grama, I.]{Universit\'{e} de Bretagne-Sud, LMBA UMR CNRS 6205, Vannes, France}
\email{ion.grama@univ-ubs.fr}
\curraddr[Liu, Q.]{Universit\'{e} de Bretagne-Sud, LMBA UMR CNRS 6205, Vannes, France}
\email{quansheng.liu@univ-ubs.fr}


\begin{abstract}
Let $(g_n)_{n\geq 1}$ be a sequence of independent and identically distributed 
elements of the general linear group $GL(d, \bb R)$. 
Consider the random walk $G_n : = g_n \ldots g_1$.
Under suitable conditions, we establish Bahadur-Rao-Petrov type large deviation expansion for 
the coefficients $\langle f, G_n v \rangle$, where $f \in (\bb R^d)^*$ and $v \in \bb R^d$. 
In particular, our result implies the large deviation principle with an explicit rate function, 
thus improving significantly the large deviation bounds established earlier. 
Moreover, we establish Bahadur-Rao-Petrov type large deviation expansion for 
the coefficients $\langle f, G_n v \rangle$ under the changed measure. 
Toward this end we prove the H\"{o}lder regularity of the stationary measure corresponding to 
the Markov chain $G_n v /|G_n v|$ under the changed measure, which is of independent interest. 
In addition, we also prove local limit theorems with large deviations for the coefficients of $G_n$. 
\end{abstract}

\date{\today}
\subjclass[2010]{Primary 60F10, 60B15, 37A30; Secondary 60B20, 60J05}
\keywords{Bahadur-Rao-Petrov large deviations; Products of random matrices;  
 Coefficients;  Regularity of stationary measure; Spectral gap}
\maketitle

\section{Introduction} 
\subsection{Background and objectives}

Let $d\geq 2$ be an integer. 
Assume that on the probability space $(\Omega,\mathcal{F},\mathbb{P})$ we are given  
a sequence of real random $d\times d$ matrices $(g_{n})_{n\geq 1}$
which are independent and identically distributed (i.i.d.) with common law $\mu$. 
A great deal of research has been devoted  to studying the random matrix product $G_n: = g_n \ldots g_1$.
Many fundamental results related to $G_n$, 
such as the strong law of large numbers, the central limit theorem, the law of iterated logarithm 
and large deviations have been established by 
Furstenberg and Kesten \cite{FK60}, Kingman \cite{Kin73},
Le Page \cite{LeP82}, Guivarc'h and Raugi \cite{GR85},  Bougerol and Lacroix \cite{BL85},  
Gol'dsheid and Margulis \cite{GM89}, Hennion \cite{Hen97}, Furman \cite{Fur02}, 
Guivarc'h and Le Page \cite{GL16},
Benoist and Quint \cite{BQ16, BQ17}, to name only a few. 
These limit theorems turn out to be very useful in various areas, 
such as the  spectral theory of random Schr\"{o}dinger operators \cite{BL85, CL90},
disordered systems and chaotic dynamics coming from statistical physics \cite{CPV93}, 
the multidimensional stochastic recursion \cite{Kes73,GL16}, 
the dynamics of group actions  \cite{BFLM11, BQ13},
and  the survival probabilities and conditioned limit theorems 
of branching processes in random environment  \cite{GLP17, LPP18, GLL19}.

Denote by $\langle f, G_n v \rangle$ the coefficients of the matrix $G_n$, 
where $f \in (\bb R^d)^*$ and $v \in \bb R^d$, and $\langle \cdot, \cdot \rangle$ is the corresponding dual bracket. 
There has been of growing interest in the study of the asymptotic behavior of $\langle f, G_n v \rangle$, 
since the seminal work of Furstenberg and Kesten \cite{FK60}, where
the following strong law of large numbers has been established for positive matrices: 
\begin{align*} 
\lim_{n \to \infty } \frac{1}{n} \log | \langle f, G_n v \rangle | = \lambda, \quad  \mbox{a.s.},
\end{align*}
with $\lambda$ a constant called the first Lyapunov exponent of the sequence $(g_{n})_{n\geq 1}$. 
In \cite{FK60} the central limit theorem has also been proved, 
thus giving an affirmative answer to Bellman's conjecture in \cite{Bel54}. 
In the case of invertible matrices, Guivarc'h and Raugi \cite{GR85} have established 
the strong law of large numbers and  the central limit theorem 
for the coefficients $\langle f, G_n v \rangle$,
where the proof turns out to be more involved than that in \cite{FK60}, and 
is based on the regularity of the stationary measure of 
the Markov chain $G_n x: = G_n v/|G_n v|$ with $x = \bb R v$ a starting point on the projective space $\bb P^{d-1}$.
Recently, Benoist and Quint \cite{BQ17} have proved the following large deviation bound: 
for any $q > \lambda$, there exists a constant $c>0$ such that 
\begin{align}\label{IntroLDBQ}
\mathbb{P} \big( \log | \langle f, G_n v \rangle | > nq \big) \leq e^{-cn}. 
\end{align}
But the precise decay rate on the large deviation probability in \eqref{IntroLDBQ} is not known. 
The goal of this paper is to establish an exact 
large deviation asymptotic for the coefficients $\langle f, G_n v \rangle$, 
called  Bahadur-Rao-Petrov type large deviations
following the groundwork by Bahadur-Rao \cite{BR60} and Petrov \cite{Pet65} 
for sums of i.i.d.\ real-valued random variables.  
Our result will imply the large deviation principle with an explicit rate function,
which  
 improves \eqref{IntroLDBQ}. 
Moreover, we shall also establish Bahadur-Rao-Petrov type upper tail large deviation asymptotics
for the couple $(G_n x, \log |\langle f, G_n v \rangle|)$ with target functions, 
which is of independent interest;  in particular it implies   a new result
on the local limit theorem with large deviations for coefficients $\langle f, G_n v \rangle$.
Similar results for lower tail large deviations are also obtained, whose 
 proof turns out to be more delicate.


\subsection{Brief overview of the main results}
Let $I_{\mu}= \{ s \geq 0: \mathbb{E}(\|g_1\|^{s})< \infty\} $, where $\| g \|$ is the operator norm of a matrix $g$.
For any $s\in I_\mu$, define $ \kappa(s)=\lim_{n\to\infty}\left(\mathbb{E}\| G_n \|^{s}\right)^{\frac{1}{n}}$.  
Set $\Lambda = \log\kappa$ and consider its Fenchel-Legendre transform
$\Lambda^*$, which satisfies $\Lambda^*(q) = s q - \Lambda(s) >0$
for $q = \Lambda'(s)$ and $s \in I_{\mu}^{\circ}$ (the interior of the interval $I_{\mu}$).
In the sequel $\langle \cdot, \cdot \rangle$ and $| \cdot |$ 
denote respectively the dual bracket and the Euclidean norm.
Denote by $\mathbb{P}^{d-1}: = \{ x = \mathbb R v:  v  \in \mathbb{R}^d \setminus \{0\} \}$ 
the projective space in $\bb R^d$; the projective space $(\mathbb P^{d-1})^*$ in $(\bb R^d)^*$ is defined similarly. 
For any $x=\mathbb Rv \in \mathbb P^{d-1}$ and $y=\mathbb Rf \in (\mathbb P^{d-1})^*$ we define
$\delta(y,x)=  \frac{| \langle f,v \rangle |}{|f| |v|}$.
For any $g\in GL(d, \bb R)$ and $x = \mathbb R v \in\mathbb P^{d-1}$, 
let $gx = \mathbb R  gv \in \mathbb P^{d-1}$, 
and denote by $gv\in \mathbb R^d$ the image of the automorphism $v\mapsto gv$ on $\mathbb R^d$. 
Consider the transfer operator $P_s$ defined by 
$P_s \varphi (x) = \mathbb{E} [ e^{ s \sigma(g_1, x)} \varphi(g_1 x) ]$, $x = \bb R v \in \mathbb{P}^{d-1}$, 
where $\sigma(g, x) = \log \frac{|gv|}{|v|}$, and $\varphi$ is a continuous function on $\mathbb{P}^{d-1}$;
the conjugate transfer operator $P_s^*$ is defined similarly: see \eqref{transfoper001}. 
The operators $P_{s}$ and $P_s^*$  have  continuous strictly positive eigenfunctions $r_s$ and $r_s^*$ 
on $\mathbb{P}^{d-1}$ which are unique up to a scaling constant, 
 and unique probability eigenmeasures $\nu_s$ and $\nu_s^*$, 
satisfying $P_s r_s = \kappa(s) r_s$, $P_s \nu_s = \kappa(s) \nu_s$,  
$P_s^* r_s^* = \kappa(s) r_s^*$ and $P_s^* \nu_s^* = \kappa(s) \nu_s^*$.
Denote $\sigma_s:= \sqrt{ \Lambda''(s) }>0$.  
For details see Section \ref{subsec a change of measure}. 

Our first objective is to establish a Bahadur-Rao type large deviation asymptotic 
for the coefficients $\langle f, G_n v \rangle$; 
we refer to Bahadur and Rao \cite{BR60} for the case of i.i.d.\ real-valued random variables.
More precisely, we prove that, for any $s \in I_{\mu}^{\circ}$,  $v \in \bb R^d $ and $f \in (\bb R^d)^*$
with $|v|=|f| = 1$, $q = \Lambda'(s)$,  $x = \bb R v$ and $y = \bb R f$,   as $n \to \infty$,  
\begin{align}\label{IntroEntryInver02}
 \mathbb{P} \Big( \log | \langle f, G_n v \rangle |  \geq nq \Big) = 
 \frac{r_{s}(x) r^*_{s}(y)}{ \varrho_s }
 \frac{ \exp \big( -n   \Lambda^*(q) \big) }{ s \sigma_s \sqrt{2\pi n} }  \big[ 1 + o(1) \big],  
\end{align}
where 
$\varrho_s = \nu_s (r_s) = \nu_s^* (r_s^*) >0$.  
The asymptotic \eqref{IntroEntryInver02} clearly implies the 
large deviation principle for $\langle f, G_n v \rangle$ with the rate function $\Lambda^*$, 
which obviously improves the large deviation bound \eqref{IntroLDBQ}.

In fact, we shall extend \eqref{IntroEntryInver02} 
to the couple $(G_n x, \log | \langle f, G_n v \rangle |)$ with target functions.
Precisely, for any $s \in I_{\mu}^{\circ}$,  any H\"{o}lder continuous function $\varphi$ on $\mathbb{P}^{d-1}$
and any measurable function  $\psi$ on $\mathbb{R}$ 
such that $u \mapsto e^{-s_1u} \psi(u)$ is directly Riemann integrable for some $s_1 \in (0,s)$, 
we prove that as $n \to \infty$, 
\begin{align} 
&  \mathbb{E} \Big[ \varphi(G_n x) \psi \big( \log |\langle f, G_n v \rangle| - nq \big) \Big] \label{IntroTargSPosi} \\
&  =    \frac{r_{s}(x)}{ \varrho_s } 
 \frac{ \exp (-n \Lambda^*(q)) }{ \sigma_{s}\sqrt{2\pi n}}  
 \left[ \int_{\mathbb{P}^{d-1}} \varphi(x) \delta(y, x)^s \nu_s(dx) 
   \int_{\mathbb{R}} e^{-su} \psi(u) du + o(1)  \right].  \nonumber
\end{align}

Our second objective is to establish a Bahadur-Rao type result for the lower large deviation probabilities
$\mathbb{P} \big( \log |\langle f, G_n v \rangle|  \leq nq \big)$, 
where $q = \Lambda'(s) < \lambda$ with  $s<0$ sufficiently close to $0$. 
Specifically, for $s<0$ small enough, we prove that, as $n \to \infty$,
\begin{align} \label{IntroEntrySNeg}
 \mathbb{P} \left( \log | \langle f, G_n v \rangle | \leq nq \right)  
 =  \frac{r_{s}(x) r^*_{s}(y)}{ \varrho_s }
\frac{ \exp \left( -n   \Lambda^*(q) \right) } { - s \sigma_{s}\sqrt{2\pi n}} \big[ 1 + o(1) \big], 
\end{align}
where $r_s$, $r_s^*$, $\varrho_s$, $\Lambda^*$ and $\sigma_s$ 
are defined in Section \ref{subsec a change of measure}, which are strictly positive,
similarly to the case $s>0$. 
The asymptotic \eqref{IntroEntrySNeg} is of course much sharper than 
the corresponding lower tail large deviation principle for $\langle f, G_n v \rangle$.
More generally, we extend the lower tail large deviation expansion \eqref{IntroEntrySNeg} to the couple 
$(G_n x, \log | \langle f, G_n v \rangle |)$ with target functions, 
in the same line as \eqref{IntroTargSPosi}. 

For a brief description of the main ideas of the approach see Section \ref{proof strategy}.

The assertions \eqref{IntroEntryInver02}, \eqref{IntroTargSPosi} and \eqref{IntroEntrySNeg} stated above
concern Bahadur-Rao type large deviation asymptotics. 
Actually we shall establish an extended version of these results with an additional vanishing perturbation on $q$, 
which in the literature is known as Bahadur-Rao-Petrov type large deviation expansion. 
Such type of extensions has important and interesting implications,
for instance, to local limit theorems with large deviations for the coefficients $\langle f, G_n v \rangle$: 
see  Theorem \ref{Theorem local LD002}.
Recently, Buraczewski, Collamore, Damek and Zienkiewicz \cite{BCDZ16}  
have established a law of large numbers, a central limit theorem and large deviation results
for perpetuities using the Bahadur-Rao-Petrov large deviation asymptotic for sums of i.i.d.\ real valued random variables. 
With the help of our large deviation results for products of random matrices
it is possible to extend these results to multivariate perpetuity sequences arising in financial mathematics.
Another potential application of our results is in the study of multitype branching processes 
and branching random walks governed by products of random matrices;
we refer to Mentemeier \cite{Men16}, Bui, Grama and Liu \cite{BGL2020a, BGL2020b} for details.

It is worth mentioning that using the approach developed in this paper,
it is possible to establish new limit theorems
for the Gromov product of random walks on hyperbolic groups;
we refer to Gou\"{e}zel \cite{Gou09, Gou14} on this topic. 
We also mention that our approach opens a way to 
study invariance principles for the coefficients $\langle f, G_n v \rangle$; 
recent progress in the study of invariance principles can be found  
in Cuny, Dedecker and Jan \cite{CDJ17}
and Cuny, Dedecker and Merlev\`ede \cite{CDM19}, 
where the vector norm $|G_n v|$ and the operator norm $\|G_n\|$
have been studied via the martingale approximation approach.


\section{Main results}\label{sec.prelim}

In this section we present our main results and the strategy of the proofs.

\subsection{Notation and conditions}\label{subsec.notations}
Denote by $c$, $C$ absolute constants whose values may change from line to line.
By $c_\alpha$, $C_{\alpha}$ we mean constants depending only on the parameter $\alpha.$
For any integrable function $\rho: \mathbb{R} \to \mathbb{C}$, denote its Fourier transform by
$\widehat{\rho} (t) = \int_{\mathbb{R}} e^{-ity} \rho(y) dy$, $t \in \mathbb{R}$.
For a measure $\nu$ and a function $\varphi$ we write $\nu(\varphi)=\int \varphi d\nu.$
Let $\mathbb N = \{1,2,\ldots\}$. By convention $\log 0 =-\infty$.

The space $\mathbb{R}^d$ is equipped  with the standard scalar product $\langle \cdot, \cdot\rangle$ 
and the associated norm $|\cdot|$. 
For any integer $d \geq 2$, 
denote by $\bb G: = GL(d,\mathbb R)$ the general linear group of invertible $d \times d$ matrices
with coefficients in $\mathbb R$. 
The projective space $\mathbb P^{d-1}$ of $\mathbb R^d$ 
is the set of elements $x = \bb R v$, where $v \in \bb R^d \setminus \{0\}$.
The projective space of $(\mathbb R^d)^*$ is denoted by $(\mathbb P^{d-1})^*$. 
We equip $\bb P^{d-1}$ with the angular distance $\mathbf d$ (see \cite{GL16}), i.e.,
for any $x, x' \in \mathbb{P}^{d-1}$ with $x \in \bb R v$ and $x' \in \bb R v'$, 
$\mathbf{d}(x,y)= (1 - \frac{|\langle v, v' \rangle|}{|v| |v'|})^{1/2}$.  

Let $\mathcal{C}(\bb P^{d-1})$ be the space of complex-valued continuous functions on $\bb P^{d-1}$. 
We write $\mathbf{1}$ for the identity function $1(x)$, $x \in \bb P^{d-1}$. 
Throughout this paper,  $\gamma>0$ is a fixed sufficiently small constant.  
For any $\varphi\in \mathcal{C}(\bb P^{d-1})$, set
\begin{align*}
\|\varphi\|_{\infty}:= \sup_{x\in \bb P^{d-1}}|\varphi(x)| \quad \mbox{and} \quad
\|\varphi\|_{\gamma}:= \|\varphi\|_{\infty}
+ \sup_{x \neq y} \frac{ |\varphi(x)-\varphi(y)| }{ \mathbf{d}(x,y)^{\gamma} }, 
\end{align*}
and consider the Banach space
$\mathcal{B}_{\gamma}:=\{\varphi\in \mathcal{C}(\bb P^{d-1}): \|\varphi\|_{\gamma}< +\infty\}.$

All over the paper $(g_{n})_{n\geq 1}$ is a sequence of i.i.d.\  elements
of the same probability law $\mu$ on $\bb G$.
Denote by  $\Gamma_{\mu}$ the smallest closed semigroup generated by the support of $\mu$.  
For any $g \in \bb G$, 
denote $\|g\| = \sup_{ v \in \bb R^d \setminus \{0\} } \frac{|g v|}{|v|}$. 
Let 
\begin{align*}
I_{\mu} = \big\{ s \geq 0: \mathbb{E}(\| g_1 \|^s) < + \infty \big\}, 
\end{align*}
and $I_{\mu}^{\circ}$ be its interior. 
In the sequel we always assume that there exists $s>0$  such that 
$\mathbb{E} ( \| g_1 \|^s )  < + \infty,$  
so that $I_{\mu}^{\circ}$ is non-empty open interval of $\mathbb{R}$.

For any $g \in \bb G$, 
set $\iota(g) = \inf_{ v \in \bb R^d \setminus \{0\} } \frac{|g v|}{|v|}$, 
and it holds that $\iota(g) = \| g^{-1} \|^{-1}$. 
We will need the following exponential moment condition:
\begin{conditionA}\label{Condi_Exp}
There exist $s\in I_\mu^\circ$  and $\beta \in(0,1)$ such that
$$
\int_{\bb G} \| g \|^{s + \beta} \iota(g)^{-\beta} \mu (dg) < + \infty.
$$
\end{conditionA}

Moreover, we shall use the following two-sided moment condition.  
Denote $N(g) = \max\{ \|g\|, \| g^{-1} \| \}$ for any $g \in \bb G$. 
\begin{conditionA}\label{Condi-TwoExp} 
There exists a constant $\eta > 0$ such that $\mathbb{E} \big( N(g_1)^{\eta} \big) < +\infty$. 
\end{conditionA}


A matrix $g \in \bb G$ is called \emph{proximal} if it has an algebraic simple dominant eigenvalue, 
namely, $g$ has an eigenvalue $\lambda_{g}$ satisfying $|\lambda_{g}| > |\lambda_{g}'|$
for all other eigenvalues $\lambda_{g}'$ of $g$.
It is easy to verify that $\lambda_{g} \in \mathbb{R}$. 
The eigenvector $v_g$ with unit norm $|v_g| = 1$, corresponding to the eigenvalue $\lambda_{g}$,
is called the dominant eigenvector. 
We will need the following strong irreducibility and proximality conditions: 

 \begin{conditionA}\label{Condi-IP}
{\rm (i)(Strong irreducibility)} 
No finite union of proper subspaces of $\mathbb{R}^d$ is $\Gamma_{\mu}$-invariant.

{\rm (ii)(Proximality)} $\Gamma_{\mu}$ contains at least one proximal matrix. 
\end{conditionA}

For any $g \in \bb G$ and $x = \bb R v \in \bb P^{d-1}$, 
 let $gx =  \bb R gv \in \mathbb P^{d-1}$ and  
\begin{align}\label{Def_MarkovChain01}
G_0 x: = x,   \quad 
G_n x: = \bb R G_n v,  \quad  n \geq 1.
\end{align}
Then $(G_n x)_{n \geq 0}$ forms a Markov chain on the projective space $\bb P^{d-1}$. 
Moreover, under condition \ref{Condi-IP}, 
$(G_n x)_{n \geq 0}$ has a unique stationary probability measure $\nu$ on $\bb P^{d-1}$ 
such that for any $\varphi \in \mathcal{C}(\bb P^{d-1})$,
\begin{align} \label{mu station meas}
 \int_{\bb P^{d-1}} \int_{\bb G} \varphi(gx) \mu(dg) \nu(dx)
 = \int_{\bb P^{d-1}} \varphi(x) \nu(dx). 
\end{align}
Furthermore, the support of $\nu$ is given by
\begin{align}\label{Def_supp_nu}
\supp \nu = \overline{\{ v_{g}\in 
\mathbb P^{d-1}:  g\in\Gamma_{\mu}, \ g \mbox{ is proximal} \}}. 
\end{align} 
For any $s\in (-s_0, 0)  \cup  I_{\mu}$ with small enough $s_0>0$,
define the transfer operator $P_s$ and the conjugate transfer operator $P_{s}^{*}$ 
as follows: for any $\varphi \in \mathcal{C}(\bb P^{d-1})$,
\begin{align}\label{transfoper001}
P_{s}\varphi(x) = \int_{\bb G}  e^{s \sigma (g,x)} \varphi( g x ) \mu(dg), 
\quad  x\in \bb P^{d-1}, 
\end{align}
where $\sigma(g, x) = \log \frac{|gv|}{|v|}$, 
and for any $\varphi \in \mathcal{C}((\bb P^{d-1})^*)$,
\begin{align}\label{transfoper002}
P_{s}^{*}\varphi(y) = \int_{\bb G}  e^{s \sigma (g^*, y)} \varphi(g^* y) \mu(dg), 
\quad  y \in (\bb P^{d-1})^*. 
\end{align}
where $g^*$ denotes the adjoint automorphism of the matrix $g$. 
Under suitable conditions, 
the transfer operator $P_s$ has a unique probability eigenmeasure $\nu_s$ on $\bb P^{d-1}$
corresponding to the eigenvalue $\kappa(s)$: 
$P_s \nu_s = \kappa(s)\nu_s.$
Similarly, the conjugate transfer operator $P_{s}^{*}$ 
has a unique probability eigenmeasure $\nu^*_s$
corresponding to the eigenvalue $\kappa(s)$: 
$P_{s}^{*} \nu^*_s = \kappa(s)\nu^*_s.$
For any $x = \bb R v \in \bb P^{d-1}$ and $y = \bb R f \in (\bb P^{d-1})^*$ 
with $v \in \bb R^d \setminus \{0\}$ and $f \in (\bb R^d)^* \setminus \{0\}$, 
denote $\delta(y, x) = \frac{|\langle f, v \rangle|}{|f||v|}$ and set
\begin{align*}
r_{s}(x) = \int_{(\bb P^{d-1})^*} \delta(y, x)^s \nu^*_{s}(dy),  \quad
r_{s}^*(y) = \int_{\bb P^{d-1}} \delta(y, x)^s \nu_{s}(dx).
\end{align*}
Then, $r_s$ is the unique, up to a scaling constant, 
strictly positive eigenfunction of $P_s$:
$P_s r_s = \kappa(s)r_s$;
similarly
$r^*_s$ is the unique, up to a scaling constant, 
strictly positive eigenfunction of $P_{s}^{*}$: $P_{s}^{*} r^*_s = \kappa(s)r^*_s$. 
It is easy to see that 
$$\nu_s(r_s) = \nu^*_s(r^*_s)= : \varrho_s.$$
The stationary measure $\pi_s$ is defined by
$ \pi_{s}(\varphi)=\frac{\nu_{s}(\varphi r_{s})}{\varrho_{s}}$, for any $\varphi\in \mathcal{C}(\bb P^{d-1})$. 
We refer to Section \ref{subsec a change of measure} for details.

Define $\Lambda = \log\kappa: (-s_0, 0)  \cup  I_{\mu} \to \mathbb R$, 
then the function $\Lambda$ is convex and analytic. 
Condition \ref{Condi-IP} implies 
 that $\sigma_s=\Lambda''(s)$ is strictly positive for any $s\in (-s_0, 0)  \cup  I_{\mu}$. 
Denote by $\Lambda^{\ast}$ the Fenchel-Legendre transform of $\Lambda$, 
then it holds that $\Lambda^*(q)=s q - \Lambda(s)>0$ 
if $q=\Lambda'(s)$ for $s\in (-s_0, 0) \cup I_{\mu}^{\circ}$.

 
\subsection{Precise large deviations for coefficients}\label{sec scalar prod}

The goal of this section is to state exact large deviation asymptotics 
for the coefficients $\langle f, G_n v \rangle$, where $f \in (\bb R^d)^*$ and $v \in \bb R^d$. 
To the best of our knowledge, the precise large deviations and even the large deviation principle 
for $\langle f, G_n v \rangle$ 
have not been studied by now in the literature.
Our first result is a large deviation asymptotic of the Bahadur-Rao type (see \cite{BR60}) 
for the upper tails of $\langle f, G_n v \rangle$. 
Recall the notation $x = \bb R v$ and $y = \bb R f$ 
for any $v \in \bb R^d \setminus \{0\}$ and $f \in (\bb R^d)^* \setminus \{0\}$.

\begin{theorem} \label{thrmBR001}
Assume conditions \ref{Condi_Exp} and \ref{Condi-IP}. 
Let $s \in I_{\mu}^{\circ}$ and $q = \Lambda'(s)$. 
Then, as $n \to \infty$, 
uniformly in $f \in (\bb R^d)^*$ and $v \in \bb R^d$ with $|f| = |v| = 1$,  
\begin{align} \label{SCALREZ01BR}
 \mathbb{P} \Big( \log | \langle f, G_n v \rangle | \geq nq \Big)  
 =  \frac{ r_{s}(x)  r^*_{s}(y)}{\varrho_s} 
\frac{ \exp \left( -n   \Lambda^*(q) \right) } {s \sigma_{s}\sqrt{2\pi n}} \big[ 1 + o(1) \big]. 
\end{align}
\end{theorem}

In particular, if we fix a basis $(e_i^*)_{1 \leq i \leq d}$ in $(\bb R^d)^*$
and a basis $(e_j)_{1 \leq j \leq d}$ in $\bb R^d$,
then taking $f=e_i^*$ and $v =e_j$ in \eqref{SCALREZ01BR}, we get the 
Bahadur-Rao type large deviation asymptotic for the $(i,j)$-th entry $G_n^{i,j}$ of the matrix product $G_n$. 
It is easy to verify that the large deviation asymptotic \eqref{SCALREZ01BR} implies a large deviation principle, 
as stated below: under the assumptions of Theorem \ref{thrmBR001}, 
we have, uniformly in $f \in (\bb R^d)^*$ and $v \in \bb R^d$ with $|f| = |v| = 1$, 
\begin{align}\label{LDP-001}
\lim_{n \to \infty}  \frac{1}{n}  \log 
\mathbb{P} \Big(  \log |\langle f, G_n v \rangle|  \geq  n q   \Big)
 = - \Lambda^*(q). 
\end{align}
In its turn, 
the asymptotic \eqref{LDP-001} 
improves significantly the bound 
\eqref{IntroLDBQ}. 

An important field of applications of large deviation asymptotics for the coefficients of type \eqref{SCALREZ01BR} 
is the study of asymptotic behaviors of multi-type branching processes in random environment.
For results in the case of single-type branching processes we refer to \cite{GLM-EJP-2017, GLM-SPA-2017} 
and for the relation between the coefficients of products of random matrices 
and the multi-type branching processes we refer to \cite{Cohn}. 



Our next result is an improvement of Theorem \ref{thrmBR001} by allowing a vanishing perturbation $l$ on $q=\Lambda'(s)$,
 in the spirit of the Petrov result \cite{Pet65}, called the Bahadur-Rao-Petrov type large deviation. 
 Large deviations with a perturbation $l$ have been used
for example in 
Buraczewski, Collamore, Damek and Zienkiewicz \cite{BCDZ16} for a recent application
to the asymptotic of the ruin time in some models of financial mathematics.
These  results are also useful to deduce local limit theorems with large deviations, 
see subsection \ref{Applic to LocalLD}.  

\begin{theorem} \label{Thm_BRP_Upper}
Assume conditions \ref{Condi_Exp} and \ref{Condi-IP}. 
Let $s \in I_{\mu}^{\circ}$ and $q = \Lambda'(s)$.
Let $(l_n)_{n \geq 1}$ be any positive sequence satisfying $\lim_{n \to \infty} l_n = 0$. 
 Then, 
 we have, as $n \to \infty$, uniformly in $|l| \leq l_n$, 
  $f \in (\bb R^d)^*$ and $v \in \bb R^d$ with $|f| = |v| = 1$, 
\begin{align*} 
 \mathbb{P} \Big( \log | \langle f, G_n v \rangle | \geq n(q+l) \Big)  
 =  \frac{ r_{s}(x)  r^*_{s}(y)}{\varrho_s} 
\frac{ \exp \left( -n   \Lambda^*(q+l) \right) } {s \sigma_{s}\sqrt{2\pi n}} \big[ 1 + o(1) \big].  
\end{align*}
More generally, for any measurable function $\psi$ on $\mathbb{R}$ 
such that $u \mapsto e^{-s'u}\psi(u)$  is directly Riemann integrable for some $s' \in (0,s)$, 
we have, as $n \to \infty$, uniformly in $|l| \leq l_n$, 
  $f \in (\bb R^d)^*$ and $v \in \bb R^d$ with $|f| = |v| = 1$, and $\varphi \in \mathcal{B}_{\gamma}$, 
\begin{align} 
&  \mathbb{E} \Big[ \varphi(G_n x) \psi \big( \log |\langle f, G_n v \rangle| - n(q+l) \big) \Big]
  \label{SCALREZ02}    \\
&   =  \frac{r_{s}(x)}{\varrho_s}     
 \frac{ \exp (-n \Lambda^*(q+l)) }{ \sigma_{s}\sqrt{2\pi n}}  
 \left[ \int_{\bb P^{d-1}} \varphi(x) \delta(y,x)^s \nu_s(dx)
       \int_{\mathbb{R}} e^{-su} \psi(u) du + o(1) \right].   \nonumber
\end{align}
\end{theorem}

A more general version of Theorem \ref{Thm_BRP_Upper} is given in Theorem \ref{Thm_BRP_Uni_s},
where it is shown that the above large deviation asymptotics hold uniformly in $s \in K_{\mu}$
with any compact set $K_{\mu} \subset I_{\mu}^{\circ}$. 

Consider the reversed random walk $M_n$ defined by $M_n = g_1 \ldots g_n$. 
Since the two probabilities $\mathbb{P} \big( \log | \langle f, G_n v \rangle | \geq n(q+l) \big)$  and
$\mathbb{P} \big( \log | \langle f, M_n v \rangle | \geq n(q+l) \big)$ are equal (as $G_n$ and $M_n$ have the same law),
for $M_n$ we have the same large deviation expansions as for $G_n$.

Now we are going to give exact asymptotics of the lower tail large deviation probabilities 
$\mathbb{P}( \log |\langle f, G_n v \rangle| \leq nq)$, where
$q=\Lambda'(s)<\lambda = \Lambda'(0)$ for $s<0$. 
These asymptotics cannot be deduced from Theorems \ref{thrmBR001} and \ref{Thm_BRP_Upper}; 
the proofs turn out to be more delicate 
and require to develop the corresponding spectral gap theory for the transfer operator $P_s$
and to establish the H\"older regularity for the stationary measure $\pi_s$ with $s<0$. 

\begin{theorem} \label{Thm-Posi-Neg-s}
Assume conditions \ref{Condi-TwoExp} and \ref{Condi-IP}.   
Then, there exists a constant $s_0>0$ 
such that for any $s \in (-s_0, 0)$ and $q=\Lambda'(s)$, 
as $n \to \infty$, 
uniformly in $f \in (\bb R^d)^*$ and $v \in \bb R^d$ with $|f| = |v| = 1$, 
\begin{align} \label{LD_Lower01}
 \mathbb{P} \Big( \log |\langle f, G_n v \rangle| \leq n q \Big)  
 =  \frac{ r_{s}(x)  r^*_{s}(y)}{\varrho_s}  
\frac{ \exp \left( -n   \Lambda^*(q) \right) } { -s \sigma_{s}\sqrt{2\pi n}} \big[ 1 + o(1) \big]. 
\end{align}
\end{theorem}

In particular, fixing a basis $(e_i^*)_{1 \leq i \leq d}$ in $(\bb R^d)^*$
and a basis $(e_j)_{1 \leq j \leq d}$ in $\bb R^d$,
with $f = e_i^*$ and $v = e_j$ in \eqref{LD_Lower01}, 
we obtain the Bahadur-Rao type lower tail large deviation asymptotic
 for the entries $G_n^{i,j}$.
From \eqref{LD_Lower01} we get a lower tail large deviation principle
under the assumptions of Theorem \ref{Thm-Posi-Neg-s}: 
uniformly in $f \in (\bb R^d)^*$ and $v \in \bb R^d$ with $|f| = |v| = 1$, 
\begin{align}\label{LDP-002}
\lim_{n \to \infty}  \frac{1}{n}  \log 
\mathbb{P} \Big(   \log |\langle f, G_n v \rangle|  \leq  n q  \Big)
 = - \Lambda^*(q). 
\end{align}
 The result \eqref{LDP-002} sharpens the following lower tail large deviation bound 
established by Benoist and Quint \cite[Theorem 14.21]{BQ17}: 
for $q < \lambda$, there exists a constant $c>0$ such that for 
all $f \in (\bb R^d)^*$ and $v \in \bb R^d$ with $|f| = |v| = 1$, 
\begin{align*} 
\mathbb{P} \Big( \log |\langle f, G_n v \rangle| < nq \Big) \leq e^{-cn}.
\end{align*} 

Now we give a Bahadur-Rao-Petrov version of the above theorem.

\begin{theorem} \label{Thm-Posi-Neg-sBRP}
Assume conditions \ref{Condi-TwoExp} and \ref{Condi-IP}. 
Let $(l_n)_{n \geq 1}$ be any positive sequence satisfying $\lim_{n \to \infty} l_n = 0$. 
Then, there exists a constant $s_0>0$ such that for any $s \in (-s_0, 0)$ and $q=\Lambda'(s)$,
we have, as $n \to \infty$, 
uniformly in $|l| \leq l_n$, $f \in (\bb R^d)^*$ and $v \in \bb R^d$ with $|f| = |v| = 1$,
\begin{align*} 
 \mathbb{P} \Big( \log |\langle f, G_n v \rangle| \leq n(q+l) \Big)  
 =  \frac{ r_{s}(x)  r^*_{s}(y)}{\varrho_s}  
   \frac{ \exp \left( -n   \Lambda^*(q+l) \right) } { -s \sigma_{s}\sqrt{2\pi n}} 
  \big[ 1 + o(1) \big]. 
\end{align*}
More generally, for any $\varphi \in \mathcal{B}_{\gamma}$ and any measurable function $\psi$ on $\mathbb{R}$ 
such that $u \mapsto e^{-s'u} \psi(u)$ is directly Riemann integrable for some $s'\in(-s_0,s)$, we have,
as $n \to \infty$, uniformly in $|l| \leq l_n$, 
  $f \in (\bb R^d)^*$ and $v \in \bb R^d$ with $|f| = |v| = 1$,
\begin{align*} 
&  \mathbb{E} \Big[ \varphi(G_n x) \psi \big( \log |\langle f, G_n v \rangle| - n(q+l) \big) \Big]
  = \frac{r_{s}(x)}{\varrho_s}
    \frac{ \exp (-n \Lambda^*(q+l)) }{ \sigma_{s}\sqrt{2\pi n}}      \nonumber\\
& \qquad\qquad\qquad  \times  \left[ \int_{ \bb P^{d-1} } \varphi(x) \delta(y,x)^s \nu_s(dx) 
        \int_{\mathbb{R}} e^{-su} \psi(u) du + o(1) \right]. 
\end{align*}
\end{theorem}



\subsection{Local limit theorems with large deviations for coefficients}  \label{Applic to LocalLD}
In this subsection we formulate the precise local limit theorems with large deviations
for the coefficients $\langle f, G_n v \rangle$. 
For sums of independent real-valued random variables, 
local limit theorems with large and moderate deviations can be found for instance in 
Gnedenko \cite{Gne48}, Sheep \cite{She64}, Stone \cite{Sto65},
Borovkov and Borovkov \cite{BB08}, Breuillard \cite{Bre2005}, Varju \cite{Var15}.
For products of random matrices, such types of local limit theorems for the vector norm $|G_n v|$
have been recently established  in \cite{BQ17, XGL19a, XGL19b}. 
Our following theorem extends the results in \cite{XGL19a, XGL19b} for the vector norm $|G_n v|$ 
to the case of the coefficients $\langle f, G_n v \rangle$.

\begin{theorem}\label{Theorem local LD002}
Let $(l_n)_{n \geq 1}$ be any positive sequence satisfying $\lim_{n \to \infty} l_n = 0$. 
Let $-\infty < a_1 < a_2 < \infty$ be real numbers.  
\begin{enumerate}
\item
Assume conditions \ref{Condi_Exp} and \ref{Condi-IP}.   
Let $s \in I_{\mu}^{\circ}$ and $q = \Lambda'(s)$.   
Then, as $n\to\infty,$  uniformly in $|l| \leq l_n$, 
  $f \in (\bb R^d)^*$ and $v \in \bb R^d$ with $|f| = |v| = 1$,
\begin{align}
\qquad \quad  & \mathbb{P} \Big( \log |\langle f, G_n v \rangle|  \in [a_1, a_2] + n(q+l) \Big)  \nonumber \\
 &   =  \big( e^{-s a_1} - e^{-s a_2} \big)
 \frac{ r_{s}(x)  r^*_{s}(y)}{\varrho_s} 
   \frac{ \exp \left( -n   \Lambda^*(q+l) \right) } {s \sigma_{s}\sqrt{2\pi n}}
\big[ 1 + o(1) \big].     \label{LLTLDa}
\end{align}

\item
Assume conditions \ref{Condi-TwoExp} and \ref{Condi-IP}. 
Then, there exists a constant $s_0>0$ 
 such that for any $s \in (-s_0, 0)$ and $q=\Lambda'(s)$, 
 as $n \to \infty$, uniformly in $|l| \leq l_n$, 
  $f \in (\bb R^d)^*$ and $v \in \bb R^d$ with $|f| = |v| = 1$,
\begin{align} 
\qquad\quad  & \mathbb{P} \Big( \log |\langle f, G_n v \rangle|  \in [a_1, a_2] + n(q+l) \Big)  \nonumber\\
 &  = \big( e^{-s a_2} - e^{-s a_1} \big)
  \frac{ r_{s}(x)  r^*_{s}(y)}{\varrho_s}  
   \frac{ \exp \left( -n   \Lambda^*(q+l) \right) } { -s \sigma_{s}\sqrt{2\pi n}} 
     \big[ 1 + o(1) \big].   \label{LLTLDb}
\end{align} 

\end{enumerate}
\end{theorem}

Taking $\varphi = \mathbf{1}$ and $\psi = \mathbbm 1_{[a_1, a_2]}$ 
with real numbers $a_1 < a_2$, 
it is easy to see that Theorem \ref{Thm_BRP_Upper} and Theorem \ref{Thm-Posi-Neg-sBRP}
respectively recovers the local limit theorem with large deviations \eqref{LLTLDa} and \eqref{LLTLDb}. 


\subsection{Precise large deviations for coefficients under the changed measure}
We now give Bahadur-Rao-Petrov type large deviations 
for the coefficients $\langle f, G_n v \rangle$ under the changed measure $\bb Q_s^x$, 
which are useful for example  in the study of branching processes and branching random walks. 

We first deal with the upper tail case.  The following result  is a more general version of Theorems \ref{thrmBR001} and \ref{Thm_BRP_Upper}.  
Denote $q_s = \Lambda'(s)$ and $q_t = \Lambda'(t)$ 
for any $s, t \in (-s_0, 0] \cup I^{\circ}_{\mu}$ with $s<t$. 

\begin{theorem}  \label{Thm_Coeff_BRLD_changedMea}
Assume conditions \ref{Condi_Exp}, \ref{Condi-TwoExp} and \ref{Condi-IP}.  
Let $s_{\infty} = \sup \{ s: s \in I_{\mu} \}$. 
Then, there exists a constant $s_0 > 0$
such that for any fixed $s \in (-s_0, s_{\infty})$
and any compact set $K_{\mu} \subset (s, s_{\infty})$, 
we have, as $n \to \infty$, 
uniformly in $t \in K_{\mu}$, $f \in (\bb R^d)^*$ and $v \in \bb R^d$ with $|f| = |v| = 1$,  
\begin{align} 
&  \mathbb Q_s^x  \Big( \log| \langle f, G_n v \rangle |  \geq n q_t  \Big) 
  =  \frac{ r_{t}(x) }{ r_{s}(x) }
   \frac{ \exp  \{  -n(\Lambda^*(q_t) - \Lambda^*(q_s) - s(q_t -q_s)) \} }
 {(t -s)\sigma_{t} \sqrt{2\pi n}} 
         \nonumber\\
&  \qquad\qquad\qquad\qquad\qquad\qquad  \times    
    \int_{ \bb P^{d-1} } \delta(y, x)^t \,  \frac{r_s(x)}{r_t(x)} \pi_t(dx)
     [ 1 + o(1)].  \label{LD_Upper_ChangeMea001}
\end{align}
More generally, there exists a constant $s_0 > 0$
such that for any fixed $s \in (-s_0, s_{\infty})$
and any compact set $K_{\mu} \subset (s, s_{\infty})$, 
for any measurable function $\psi$ on $\mathbb{R}$ 
such that $u \mapsto e^{-s'u}\psi(u)$  is directly Riemann integrable 
for any $s' \in K_{\mu}^{\epsilon} : = \{ s' \in \bb R: |s' - s| < \epsilon, s \in K_{\mu} \}$ 
with $\epsilon >0$ small enough, 
we have, as $n \to \infty$, uniformly in $t \in K_{\mu}$, 
  $f \in (\bb R^d)^*$ and $v \in \bb R^d$ with $|f| = |v| = 1$, 
\begin{align} 
&  \mathbb{E}_{\bb Q_s^x} 
\Big[ \varphi(G_n x) \psi \big( \log |\langle f, G_n v \rangle| - nq_t \big) \Big]  \nonumber\\
&  =  \frac{ r_{t}(x) }{ r_{s}(x) }
   \frac{ \exp  \{  -n(\Lambda^*(q_t) - \Lambda^*(q_s) - s(q_t -q_s)) \} }
 {\sigma_{t} \sqrt{2\pi n}}      \nonumber \\
&  \quad \times  
 \left[  \int_{ \bb P^{d-1} } \varphi(x) \delta(y, x)^t \,  \frac{r_s(x)}{r_t(x)} \pi_t(dx)
       \int_{\mathbb{R}} e^{- (t-s) u} \psi(u) du + o(1) \right].   \label{LD_Upper_ChangeMea002}
\end{align}
\end{theorem}

We next consider the lower tail case. 
The following result is an extension of   Theorems \ref{Thm-Posi-Neg-s} and  \ref{Thm-Posi-Neg-sBRP}. 
Denote $q_s = \Lambda'(s)$ and $q_t = \Lambda'(t)$ 
for any $s, t \in (-s_0, 0] \cup I^{\circ}_{\mu}$ with $s > t$. 

\begin{theorem}  \label{Thm_Coeff_BRLD_changedMea02}
Assume conditions \ref{Condi_Exp}, \ref{Condi-TwoExp} and \ref{Condi-IP}.  
Then, there exists a constant $s_0 > 0$
such that for any fixed $s \in (-s_0, 0] \cup I^{\circ}_{\mu}$
and any compact set $K_{\mu} \subset (-s_0, s)$, 
we have, as $n \to \infty$, 
uniformly in $t \in K_{\mu}$, $f \in (\bb R^d)^*$ and $v \in \bb R^d$ with $|f| = |v| = 1$,    
\begin{align*} 
&  \mathbb Q_s^x  \Big( \log| \langle f, G_n v \rangle |  \leq n q_t  \Big) 
  = \int_{ \bb P^{d-1} } \delta(y, x)^t \,  \frac{r_s(x)}{r_t(x)} \pi_t(dx)
         \nonumber\\
&  \qquad\quad  \times   \frac{ r_{t}(x) }{ r_{s}(x) }  
   \frac{ \exp  \{  -n(\Lambda^*(q_t) - \Lambda^*(q_s) - s(q_t -q_s)) \} }
 {(s - t)\sigma_{t} \sqrt{2\pi n}} [ 1 + o(1)].
\end{align*}
More generally, there exists a constant $s_0 > 0$
such that for any fixed $s \in (-s_0, 0] \cup I^{\circ}_{\mu}$
and any compact set $K_{\mu} \subset (-s_0, s)$,
for any measurable function $\psi$ on $\mathbb{R}$ 
such that $u \mapsto e^{-s'u}\psi(u)$  is directly Riemann integrable 
for any $s' \in K_{\mu}^{\epsilon} : = \{ s' \in \bb R: |s' - s| < \epsilon, s \in K_{\mu} \}$ 
with $\epsilon >0$ small enough, 
we have, as $n \to \infty$, uniformly in $t \in K_{\mu}$, 
  $f \in (\bb R^d)^*$ and $v \in \bb R^d$ with $|f| = |v| = 1$, 
\begin{align} 
&  \mathbb{E}_{\bb Q_s^x} 
\Big[ \varphi(G_n x) \psi \big( \log |\langle f, G_n v \rangle| - nq_t \big) \Big]  \nonumber\\
&  =  \frac{ r_{t}(x) }{ r_{s}(x) }
   \frac{ \exp  \{  -n(\Lambda^*(q_t) - \Lambda^*(q_s) - s(q_t -q_s)) \} }
 {\sigma_{t} \sqrt{2\pi n}}      \nonumber \\
&  \quad \times  
 \left[  \int_{ \bb P^{d-1} } \varphi(x) \delta(y, x)^t \,  \frac{r_s(x)}{r_t(x)} \pi_t(dx)
       \int_{\mathbb{R}} e^{ -(t-s) u} \psi(u) du + o(1) \right].   \label{LD_Lower_ChangeMea002}
\end{align}
\end{theorem}

The proof of  Theorem \ref{Thm_Coeff_BRLD_changedMea02} 
 relies essentially on the 
H\"{o}lder regularity of the stationary measure $\pi_s$,
which will be presented in Section \ref{sect-holder-reg}. 

 
\subsection{Proof strategy} \label{proof strategy}
The standard approach to obtain precise large deviations for i.i.d.\ real-valued random variables consists in performing a change of measure
and proving an Edgeworth expansion under the changed measure (see e.g. \cite{BR60, Pet65, DZ09}).
Applying this strategy to the coefficients $\langle f, G_n v \rangle$ 
of products of random matrices turns out to be way more difficult. 
We have to overcome three main difficulties: 
state an Edgeworth expansion for the couple  $(G_n x, \log |\langle f, G_n v \rangle|)$ 
with a target function $\varphi$ on the Markov chain $G_n x$ under the changed measure;
give a precise control of the difference between $\log |\langle f, G_n v \rangle|$ and $\log |G_n v|$;
establish the regularity of the eigenmeasure $\nu_s$.


For the first point, 
it turns out that the techniques which work for the quantity $\log |\langle f, G_n v \rangle|$ 
alone cannot be applied for the couple.
Dealing with a couple $(G_n x, \log |\langle f, G_n v \rangle|)$ with a target function on $G_n x$ 
needs considering a new kind of smoothing inequality 
on a complex contour, instead of the usual Esseen one on the real line.
We make use of the saddle point method to obtain precise asymptotics 
for the integrals of the corresponding Laplace transforms on the complex plane.
For this method we refer to a recent work of the authors \cite{XGL19b}
where the Edgeworth expansion with a target function on $G_n x$ 
for the norm cocycle $\log |G_n v|$ has been established.

Secondly, from the previous work on limit theorems such as the strong law of large numbers, 
the central limit theorem and the law of iterated logarithm 
for the coefficients $\langle f, G_n v \rangle$, 
see e.g. \cite{GR85, BL85, Hen97, BQ17}, 
we know that the difference $|\log |\langle f, G_n v \rangle| - \log |G_n v| |$
generally diverges to infinity as $n \to \infty$. 
It is controlled by the corresponding norming factors in these limit theorems.
However, such a control is not enough  
to obtain precise large deviation expansions for $\langle f, G_n v \rangle$, nor even for a large deviation 
principle with explicit rate function. A precise account of the contribution of the error term 
is given by the following decomposition: 
for any $x = \bb R v$ and $y = \bb R f$ with $|f| = |v| =1$, 
\begin{align} \label{basic decompos001}
\log | \langle f, G_n v \rangle | = \log |G_n v| + \log \delta(y, G_n x),  \quad n \geq 1, 
\end{align}
where $\delta(y,x) = \frac{|\langle f, v \rangle|}{|f||v|}$. 
The exact decomposition \eqref{basic decompos001} allows us to deduce the precise large deviation asymptotic
from the results for the couple $(G_n x, \log |G_n v|)$ 
with a target function on $G_n x$ established in \cite{XGL19a}.
The idea is as follows:
with $\mathbb{Q}_{s}^{x}$ the changed measure defined in Section \ref{subsec a change of measure}, 
we have  
\begin{align} \label{intro001}
 \frac{ e^{ n\Lambda^{*}( q) } }{r_{s}(x)}
\mathbb{P} \left( \log |\langle f, G_n v \rangle | \geq n q  \right)
=    \mathbb{E}_{\mathbb{Q}_{s}^{x}}
     \left[  \frac{ e^{ -s (\log |G_n v| - nq ) } }{ r_{s}(G_n x) }  
      \mathbbm{1}_{ \{ \log | \langle f, G_n v \rangle | - nq \geq 0 \} }  \right].
\end{align}
We only sketch 
how to cope with the upper bound of the right-hand side of \eqref{intro001}. 
Consider a partition $I_k: = (-\eta k, -\eta(k-1)]$, $k \geq 1$, of the interval $(-\infty,0]$, 
where $\eta > 0$ is a sufficiently small constant. 
Using \eqref{basic decompos001} we get the upper bound 
\begin{align*}
 \mathbbm{1}_{ \{ \log | \langle f, G_n v \rangle | - nq \geq 0 \} }   
  \leq \sum_{k = 1}^{\infty} \mathbbm{1}_{ \big\{ \log |G_n v | - nq - \eta (k-1) \geq 0 \big\} } 
     \mathbbm{1}_{ \big\{ \log \delta(y, G_n x) \in  I_k  \big\} },  
\end{align*}
which we substitute into \eqref{intro001}. 
Thus we are led to the estimation of the sum 
\begin{align}\label{Intro-UppSum01}
\sum_{k = 1}^{\infty}  e^{ - s \eta (k-1) }     
 \mathbb{E}_{\mathbb{Q}_{s}^{x}}
   \left[  \frac{ \psi_s (\log |G_n v| - nq - \eta (k-1) ) }{ r_{s}(G_n x) }  
     \mathbbm{1}_{ \{ \log \delta(y, G_n x) \in  I_k \} }  \right],
\end{align}
where $\psi_s (u) = e^{-su} \mathbbm{1}_{ \{ u \geq 0 \} }$, $u \in \mathbb{R}$. 
Let $R_{s,it}(\varphi)(x) = \mathbb{E}_{\mathbb{Q}_{s}^{x}} [e^{it (\sigma(g_1, x) - q )} \varphi(g_1 x)]$
be the perturbed transfer operator defined  
for any H\"{o}lder continuous function $\varphi$ on $\mathbb{P}^{d-1}$,
and $R^{n}_{s,it}$ be its $n$-th iteration. 
Then, by the Fourier inversion formula, the sum in \eqref{Intro-UppSum01} is bounded from above by 
\begin{align}\label{Intro-Inte-a}
\frac{ 1 }{2 \pi}  \sum_{k = 1}^{\infty} e^{ - s \eta (k-1) } 
  \int_{\mathbb{R}}  e^{-it \eta (k-1) }  
  R^{n}_{s,it}( r_s^{-1} \Phi_{s,k,\ee_2} )(x) \widehat{\Psi}_{s, \ee_1} (t)dt,
\end{align}
where we choose some appropriate smooth functions $\Phi_{s,k,\ee_2}$ and $\Psi_{s,\ee_1}$, 
for $\ee_1, \ee_2>0$,
which dominate $\mathbbm{1}_{\{ \log \delta(y, \cdot) \in I_k \} }$ and $\psi_s$, respectively.
Using spectral gap properties of $R_{s,it}$, 
it has been established recently in \cite{XGL19a}  
(see Propostion \ref{Prop Rn limit1}) 
that, for any $k \geq 1$, 
the term under the sign of the infinite sum in \eqref{Intro-Inte-a}, say $I_n(k)$, 
converges  as $n \to \infty$ to a limit, say 
$I(k)=\frac{ \sqrt{ 2 \pi} }{ s  \sigma_s \nu_s(r_s) }  e^{ - s \eta (k-1) } \nu_s( \Phi_{s,k,\ee_2}  )$.
The interchangeability of the limit as $n \to \infty$ 
and of the summation over $k$ in \eqref{Intro-Inte-a} is justified
by specifying the rate in the convergence of $I_n(k)$ to $I(k)$, as argued in \cite{XGL19a}. 
This implies that as $n \to \infty$ and $\ee_1 \to 0$, \eqref{Intro-Inte-a} converges to 
$\sum_{k = 1}^{\infty} I_k$. 
It remains to show that the last sum converges to $r_s^* (y)$, as $\eta \to 0$ and $\ee_2 \to 0$.  
For this we have to make use of 
the zero-one law of the eigenmeasure $\nu_s$ established recently in \cite{GQX20}: 
for any $y \in (\bb P^{d-1})^*$ and any $t \in (-\infty, 0)$, 
\begin{align} \label{Intro-Regu02}
\nu_s \left( \left\{x \in \mathbb{P}^{d-1}:  \log \delta(y, x) =  t \right\}  \right)
= 0 \  \mbox{or} \  1.
\end{align}
With $s=0$ it was used in \cite{GQX20} to prove a local limit theorem 
for the coefficients $\langle f, G_n v \rangle$. 

The proof of the lower large deviation asymptotic \eqref{IntroEntrySNeg} 
can be carried out in a similar way as that of upper large deviation asymptotic  \eqref{IntroEntryInver02}.
The novelty here consists in the use of the change of measure formula for $\mathbb{Q}_s^x$
when $s<0$ and of the spectral gap theory under the changed measure as stated in \cite{XGL19b} for $s<0$.
In addition we need the H\"{o}lder regularity of the eigenmeasure $\nu_s$ for $s <0$ sufficiently close to $0$. 

In some applications it is very useful to extend the large deviation results
\eqref{IntroEntryInver02}, \eqref{IntroTargSPosi} and \eqref{IntroEntrySNeg}
to the setting under the changed measure $\bb Q_s^x$; 
see Theorems \ref{Thm_Coeff_BRLD_changedMea} and \ref{Thm_Coeff_BRLD_changedMea02}.  
To obtain these results, an important step 
is to establish the H\"{o}lder regularity of the eigenmeasure $\nu_s$ 
when $s > 0$; see Proposition \ref{PropRegularity}. 
For this we adapt the arguments from \cite{GR85} and \cite{BL85} 
where \eqref{Intro-Regu02} was established for $s=0$.
For $s>0$ the arguments are much more delicate. 
One of the difficulties is that the sequence $(g_n)_{n \geq 1}$ 
becomes dependent under the changed measure $\mathbb{Q}_s^x$.
We need to extend the results in \cite{BL85} to this case.   
Of crucial importance are 
the simplicity of the dominant Lyapunov exponent for $G_n$ under the changed measure 
recently established in \cite{GL16} (see Lemma \ref{Lem_Lya_Meas}),
and  the key proximality property which states that $M_n m$
(here $M_n = g_1 \ldots g_n$) 
 converges weakly to the Dirac measure $\delta_{Z_s}$,
where $Z_s$ is a random variable whose law is the stationary measure $\pi_s$ of $G_n x$, for $s>0$
(see Lemma \ref{Lem_DiracMea}), 
and $m$ is the unique rotation invariant measure on $\mathbb P^{d-1}$.

\section{Spectral gap properties and H\"{o}lder regularity of the stationary measure} \label{sect-holder-reg}

In this section we  present some preliminaries on the spectral gap properties and 
 state  some new results on the regularity of the stationary  measure $\pi_s$. 
 The  spectral gap and regularity properties  will be used in the proofs of the main theorems. 
In particular,  the regularity properties of the stationary measure $\pi_s$, 
will play an important role in 
 the proof of  Theorem \ref{Thm_Coeff_BRLD_changedMea02}. 
As other applications of the regularity properties, we will obtain 
a law of large numbers and
a central limit theorem for coefficients under the changed measure.


\subsection{Spectral gap properties and a change of measure}\label{subsec a change of measure}

Recall that the transfer operator $P_{s}$ and the conjugate transfer operator
$P_{s}^{*}$ are defined by \eqref{transfoper001}. 
Below $P_s\nu_{s}$ stands for the measure on $\bb P^{d-1}$ such that $P_s\nu_{s}(\varphi)=\nu_{s}(P_s \varphi),$ 
for any continuous functions $\varphi$ on $\bb P^{d-1}$, and $P^*_s\nu^*_{s}$ is defined similarly.   
The spectral gap properties of $P_{s}$ and $P_{s}^{*}$ are summarized in the following proposition which was proved 
in \cite{GL16}.


\begin{proposition}\label{transfer operator}
Assume condition \ref{Condi-IP}. 
Then, for any $s\in I_{\mu}^{\circ}$, the following assertions hold: 
\begin{enumerate}

\item 
 the spectral radii of the operators $P_s$ and $P_s^*$ are both equal to $\kappa(s)$
and there exist a unique, up to a scaling constant,
strictly positive H\"{o}lder continuous  
function $r_{s}$
and a unique probability measure $\nu_{s}$ on $\bb P^{d-1}$ such that 
\begin{align*}
P_s r_s=\kappa(s)r_s, \quad P_s\nu_{s}=\kappa(s)\nu_{s}; 
\end{align*}

\item 
 there exist a unique strictly positive H\"{o}lder continuous function 
$r_{s}^{\ast}$ and 
a unique probability measure $\nu_{s}^{*}$ on $\bb P^{d-1}$ such that
\begin{align*}
P_{s}^{*}r_{s}^{*}=\kappa(s)r_{s}^{*}, \quad  P_{s}^{*}\nu_{s}^{*}=\kappa(s)\nu_{s}^{\ast};
\end{align*}
moreover, the function $\kappa: I_{\mu}^{\circ} \mapsto \mathbb R $ is analytic. 
\end{enumerate}
\end{proposition}

The case of $s<0$ is not covered by Proposition \ref{transfer operator}. 
We state below the corresponding result, which was proved in \cite{GQX20, XGL19a}.

\begin{proposition} \label{Prop-Trans-s-neg}
Assume conditions \ref{Condi-TwoExp} and \ref{Condi-IP}. 
Then there exists a constant $s_0 > 0$ such that for any $s \in (-s_0, 0)$, 
the assertions (1) and (2) of Proposition \ref{transfer operator} remain valid.
Moreover, the function $\kappa: (-s_0, 0) \mapsto \mathbb R $ is analytic.   
\end{proposition}

Now we give explicit formulae for the eigenfunctions $r_s$ and $r_s^*$.

\begin{lemma} \label{Lemma-expleinenfun-s-neg} 
\begin{enumerate}
\item
Assume condition \ref{Condi-IP}. 
Then, for $s\in I_{\mu}^{\circ}$, the eigenfunctions $r_s$ and $r_s^*$ are given as follows:
for any $x\in \bb P^{d-1}$ and $y \in (\bb P^{d-1})^*$, 
\begin{align} \label{expleigenfun001}
\quad\quad   r_{s}(x)= \int_{(\bb P^{d-1})^*} \delta(x,y)^s \nu^*_{s}(dy),
\quad   r_{s}^*(y)= \int_{\bb P^{d-1}} \delta(x,y)^s \nu_{s}(dx). 
\end{align}

\item
Assume conditions \ref{Condi-TwoExp} and \ref{Condi-IP}. 
Then there exists a constant $s_0 > 0$ such that for any $s \in (-s_0, 0)$, 
the expressions in \eqref{expleigenfun001} remain valid. 
\end{enumerate} 
\end{lemma}

The first assertion of Lemma \ref{Lemma-expleinenfun-s-neg} for $s>0$ was proved in \cite{GL16}.
The proof of the second one for $s<0$ is quite different from that in the case $s>0$ 
and was proved in \cite{GQX20}. 
It is based on the H\"{o}lder regularity of the eigenmeasures $\nu_s$ and $\nu_s^*$ 
which is the subject of the next section. 

By Propositions \ref{transfer operator} and \ref{Prop-Trans-s-neg}, 
the eigenvalue $\kappa(s)$ and the eigenfunction $r_s$ are both strictly positive.
This allows to perform a change of measure, as shown below. 
Under the corresponding assumptions of Propositions \ref{transfer operator} and \ref{Prop-Trans-s-neg}, 
for any $s \in (-s_0, 0) \cup I_{\mu}$, 
the family of probability kernels  
$q_{n}^{s}(x,g) = \frac{ e^{s \sigma(g, x)} }{\kappa^{n}(s)} \frac{r_{s}(g x)}{r_{s}(x)},$
$n\geq 1$, satisfies the cocycle property: 
for any $x \in \bb P^{d-1}$ and $g_1, g_2 \in \Gamma_{\mu}$, 
\begin{align} \label{cocycle01}
q_{n}^{s}(x,g_1)q_{m}^{s}(g_1 x, g_2)=q_{n+m}^{s}(x,g_2g_1).
\end{align}
Thus the probability measures 
$q_{n}^{s}(x,g_{n}\dots g_{1})\mu(dg_1)\dots\mu(dg_n)$
form a projective system on $\bb G^{\bb N^*}$. 
By the Kolmogorov extension theorem,
there exists a unique probability measure  $\mathbb Q_s^x$ on $\bb G^{\bb N^*}$.  
The corresponding expectation is denoted by $\mathbb{E}_{\mathbb Q_s^x}$.
Then the change of measure formula follows: 
for any measurable function $h$ on $(\bb P^{d-1} \times \mathbb R)^{n}$, 
\begin{align}\label{basic equ1}
& \frac{1}{ \kappa^{n}(s) r_{s}(x) }
 \mathbb{E} \Big[  r_{s}(G_n x) e^{s \sigma(G_n, x) }  
  h \Big( G_1 x, \sigma (G_1, x), \dots,  G_n x, \sigma (G_n, x)  \Big) \Big]   \nonumber\\
& =\mathbb{E}_{\mathbb{Q}_{s}^{x}} 
  \Big[ h \Big( G_1 x, \sigma (G_1, x), \dots,  G_n x, \sigma (G_n, x)  \Big) \Big].
\end{align}
Under the changed measure $\mathbb Q_s^x$, the process $(G_n x)_{n \geq 0}$ defined by \eqref{Def_MarkovChain01}
still constitutes a Markov chain on $\bb P^{d-1}$ with the transition operator given by 
\begin{align}\label{Def_Q_s_lll}
Q_{s}\varphi(x) = \frac{1}{\kappa(s)r_{s}(x)}P_s(\varphi r_{s})(x), 
\quad x \in \bb P^{d-1}. 
\end{align}
The Markov operator $Q_{s}$ has a unique stationary
probability measure $\pi_{s}$ satisfying that there exists constants $c, C >0$ such that 
for any $\varphi \in \mathcal{B}_{\gamma}$, 
\begin{align} \label{equcontin Q s limit}
\| Q_{s}^{n}\varphi - \pi_{s}(\varphi) \|_{\gamma}  \leq C e^{-cn} \|\varphi\|_{\gamma},  
\quad  \mbox{where}  \quad 
\pi_{s}(\varphi)=\frac{\nu_{s}(\varphi r_{s})}{\nu_{s}(r_{s})}.
\end{align}
For any $s \in (-s_0, 0) \cup I_{\mu}$ and $t \in \mathbb{R},$ 
define a family of perturbed operators $R_{s,it}$ as follows: 
for any $\varphi \in \mathcal{B}_{\gamma}$,
\begin{align}\label{operator Rsz}
R_{s, it}\varphi(x) 
= \mathbb{E}_{\mathbb{Q}_{s}^{x}} \left[ e^{ it ( \sigma(g_1, x) - q ) } \varphi(g_1 x) \right],
\quad   x \in \bb P^{d-1}.  
\end{align}
It follows from the cocycle property \eqref{cocycle01} that 
\begin{align*} 
R^{n}_{s, it}\varphi(x) 
= \mathbb{E}_{\mathbb{Q}_{s}^{x}} \left[ e^{ it( \sigma(G_n, x) - nq) } \varphi(G_n x) \right], 
\quad   x \in \bb P^{d-1}. 
\end{align*}
Under various restrictions on $s$, it was shown in \cite{BM16, XGL19a, XGL19b}
that the operator $R_{s, it}$ acts onto the Banach space $\mathcal{B}_{\gamma}$ 
and has a spectral gap. 


\subsection{H\"{o}lder regularity of the stationary measure}

In this section we present our results on the H\"{o}lder regularity of the stationary measure 
$\pi_s$ and of the eigenmeasure $\nu_s$.
The regularity of $\pi_s$ and $\nu_s$  is central to  establishing 
the precise large deviation asymptotics for the coefficients $\langle f, G_n v \rangle$
under the changed measure $\bb Q_s^x$
and is also of independent interest. 
Below we denote $B(y, r) = \{ x \in \bb P^{d-1}: \delta(y,x) \leq r \}$
for $y \in (\bb P^{d-1})^*$ and $r \geq 0$. 

\begin{proposition}\label{PropRegularity}
Assume conditions \ref{Condi_Exp} and \ref{Condi-IP}. 
Then, for any $s \in I_{\mu}^{\circ}$, there exists a constant $\alpha > 0$ such that 
\begin{align} \label{RegularityIne00}
\sup_{y \in (\bb P^{d-1})^* } \int_{ \bb P^{d-1} } \frac{1}{ \delta(y, x)^{\alpha}} \pi_s(dx) < +\infty.  
\end{align}
In particular, for any $s \in I_{\mu}^{\circ}$,
there exist constants $\alpha, C >0$ such that for any $r \geq 0$, 
\begin{align} \label{RegularityIne}
\sup_{y \in (\bb P^{d-1})^* }  
\pi_s \big( B(y, r)  \big)  \leq C r^{\alpha}. 
\end{align}
Moreover, the assertions \eqref{RegularityIne00} and \eqref{RegularityIne} remain valid
when the stationary measure $\pi_s$ is replaced by the eigenmeasure $\nu_s$. 
\end{proposition}

The proof of Proposition \ref{PropRegularity} is technically involved and 
is postponed to Section \ref{Sec:regpositive}.

By \eqref{RegularityIne} and the Frostman lemma, it follows that the Hausdorff
dimension of the stationary measure $\pi_s$ is at least $\alpha$.


For $s=0$ the H\"{o}lder regularity of the stationary measure $\nu$ ($\nu = \pi_0 = \nu_0$) is due to Guivarc'h \cite{Gui90}.
We also refer to \cite{BL85} for a detailed description of the method used in \cite{Gui90}
and to \cite{BFLM11, BQ17} for a different approach. 
Such regularity is of great importance in the study of products of random matrices.  
For example, it turns out to be crucial for establishing limit theorems for the coefficients $\langle f, G_n v \rangle$
and for the spectral radius $\rho(G_n)$ of $G_n$. 
However, similar result has not been established in the literature 
for the stationary measure $\pi_s$ when $s \in I_{\mu}^{\circ}$. 
The proof of the assertion \eqref{RegularityIne00} 
is based on the asymptotic properties of the components in the Cartan and Iwasawa decompositions 
of the reversed random matrix product $M_n = g_1 \ldots g_n$ 
and on the simplicity of the dominant Lyapunov exponent of $G_n$ under the changed measure $\mathbb{Q}_s^x$:
see Section \ref{Sec:regpositive}.

When $s$ is non-positive and sufficiently close to $0$, 
we also give the H\"{o}lder regularity of the stationary measure $\pi_s$. 

\begin{proposition}\label{PropRegu02}
Assume conditions \ref{Condi-TwoExp} and \ref{Condi-IP}.
Then, there exist constants $\alpha, s_0, C>0$ such that 
the statements \eqref{RegularityIne00} and \eqref{RegularityIne} hold for any $s \in (-s_0, 0]$. 
\end{proposition}

Proposition \ref{PropRegu02} has been recently established in \cite{GQX20} using the
H\"older regularity of the stationary measure $\nu$ and the analyticity of the eigenfunction $\kappa$.

We will establish the following assertion, which is a stronger version of Proposition \ref{PropRegularity}. 

\begin{proposition}\label{Prop_Regu_Strong01}
Assume conditions \ref{Condi_Exp} and \ref{Condi-IP}. 
Let $s \in I_{\mu}^{\circ}$.  
Then, for any $\ee > 0$, 
there exist constants $c: = c(s) > 0$ and $n_0: = n_0(s) \geq 1$
such that for all $n \geq k \geq n_0$, $x \in \mathbb P^{d-1}$ and $y \in (\mathbb P^{d-1})^*$, 
\begin{align*} 
\bb Q_s^x \Big( \delta (y, G_n x) \leq e^{- \ee k}  \Big) \leq e^{- c k}.
\end{align*}
\end{proposition}

Similarly, the following result is a stronger version of Proposition \ref{PropRegu02}. 

\begin{proposition}\label{Prop_Regu_Strong02}
Assume conditions \ref{Condi-TwoExp} and \ref{Condi-IP}.
Let $s \in (-s_0, s_0)$, where $s_0 > 0$ is small enough.
Then, for any $\ee > 0$, 
there exist constants $c: = c(s) > 0$ and $n_0: = n_0(s) \geq 1$
such that for all $n \geq k \geq n_0$, $x \in \mathbb P^{d-1}$ and $y \in (\mathbb P^{d-1})^*$, 
\begin{align*} 
\bb Q_s^x \Big( \delta (y, G_n x) \leq e^{- \ee k}  \Big) \leq e^{- c k}.
\end{align*}
\end{proposition}

It turns out that Propositions \ref{Prop_Regu_Strong01} and \ref{Prop_Regu_Strong02}
play an important role for establishing the Bahadur-Rao-Petrov type lower tail large deviations 
for the coefficients $\langle f, G_n v \rangle$ under the changed measure $\bb Q_s^x$,
see Theorem \ref{Thm_Coeff_BRLD_changedMea02}. 
Moreover, they are very useful to obtain the strong law of large numbers (SLLN) 
and the central limit theorem (CLT) for the coefficients $\langle f, G_n v \rangle$
under the changed measure $\bb Q_s^x$, see the next section. 

\subsection{Applications to SLLN and CLT for the coefficients}
In this section we formulate the SLLN and the CLT
for the coefficients $\langle f, G_n v \rangle$ 
under the changed measure $\mathbb{Q}_{s}^{x}$. 
These assertions are not used in the proofs of our large deviation results, but are of independent interest.
They are deduced from the SLLN and the CLT for the norm cocycle $\log |G_n v|$ using the H\"older regularity of 
stationary measure $\pi_s$ stated in Propositions \ref{PropRegularity} and \ref{PropRegu02}.

When $s \in I_{\mu}$, the SLLN for $\log |G_n v|$ was established in \cite{GL16}: 
under conditions  \ref{Condi-TwoExp} and \ref{Condi-IP}, 
for any $x = \bb R v \in \bb P^{d-1}$,  
\begin{align}\label{SLLN_Gnx}
\lim_{n \to \infty} \frac{1}{n} \log |G_n v| 
=  \Lambda'(s), 
\quad  \mathbb{Q}_s^x\mbox{-a.s.}, 
\end{align}
where $\Lambda'(s) = \frac{\kappa'(s)}{\kappa(s)}$ 
with the function $\kappa$ defined in Proposition \ref{transfer operator}. 
The CLT for $\log |G_n v|$ under the changed measure
$\mathbb{Q}_s^x$ was proved in \cite{BM16}:
for any $s \in I_{\mu}$ and $t \in \mathbb{R}$, 
it holds uniformly in $x = \bb R v \in \bb P^{d-1}$ with $|v| = 1$ that 
\begin{align}\label{CLT_Cocycle01}
\lim_{ n \to \infty }  \mathbb{Q}_s^x  
  \left(  \frac{ \log |G_n v| - n \Lambda'(s) }{ \sigma_s \sqrt{n} }  \leq t  \right)
 =  \Phi(t),  
\end{align}
where $\Phi$ is the standard normal distribution function on $\mathbb{R}$.

When $s \in (-s_0, 0)$ with small enough $s_0 > 0$, the SLLN and the CLT for $\log |G_n v|$ under the measure
$\mathbb{Q}_s^x$ have been recently established in \cite{XGL19b}. 

We now give the SLLN and the CLT for the coefficients $\langle f, G_n v \rangle$
under the measure $\mathbb{Q}_s^x$. 

\begin{proposition}\label{LLN_CLT_Entry}
\begin{itemize}
\item[\rm{(1)}]
Assume conditions \ref{Condi_Exp} and \ref{Condi-IP}. 
Then, for any $s \in I_{\mu}$, 
uniformly in $f \in (\bb R^d)^*$ and $v \in \bb R^d$ with $|f| = |v| = 1$,
\begin{align}\label{SLLN_Entry}
\lim_{ n \to \infty } \frac{1}{n} \log | \langle f, G_n v \rangle | = \Lambda'(s), 
\quad  \mathbb{Q}_s^x \mbox{-a.s..}
\end{align}
Moreover, for any $s \in I_{\mu}$ and $t \in \mathbb{R}$, we have, 
uniformly in $f \in (\bb R^d)^*$ and $v \in \bb R^d$ with $|f| = |v| = 1$, 
\begin{align}\label{CLT_Entry}
\lim_{ n \to \infty }  \mathbb{Q}_s^x  
  \left(  \frac{ \log | \langle f, G_n v \rangle | - n \Lambda'(s) }{ \sigma_s \sqrt{n} }  \leq t  \right)
 =  \Phi(t). 
\end{align}
\item[\rm{(2)}]
Assume conditions \ref{Condi-TwoExp} and \ref{Condi-IP}.
Then, there exists $s_0>0$ such that for any $s \in (-s_0, 0)$, 
the assertions \eqref{SLLN_Entry} and \eqref{CLT_Entry} hold. 
\end{itemize}
\end{proposition}

The proof of Proposition \ref{LLN_CLT_Entry} relies on Propositions \ref{PropRegularity} and \ref{PropRegu02}
and is postponed to Section \ref{Sec:regpositive}.

\section{Auxiliary results} \label{sec-spgappert}

In this section we state some preliminary results about the Taylor's expansion of $\Lambda^*$,  
a smoothing inequality, and some asymptotics of the perturbed operator $R_{s,it}$, 
which will be used to establish Bahadur-Rao-Petrov type large deviations.

The following lemma is proved in \cite{XGL19a} and gives Taylor's expansion of $\Lambda^*(q+l)$ 
with respect to the perturbation $l$. 
Recall that 
under conditions \ref{Condi_Exp}, \ref{Condi-TwoExp} and \ref{Condi-IP}, 
the moment generating function $\Lambda = \log \kappa$ is strictly convex and analytic on $(-s_0, 0) \cup I_{\mu}$;
see e.g. \cite{GL16, BM16, XGL19a}.  
Set $\gamma_{s,k} = \Lambda^{(k)} (s)$, $k\geq 1$. 
In particular, $\gamma_{s,2} = \Lambda'' (s) = \sigma_s^2$. 
Under the changed measure $\mathbb Q_s^x$,
define the Cram\'{e}r series $\zeta_s$ (see Petrov \cite{Pet75})
by
\begin{align*}
\zeta_s (t) = \frac{\gamma_{s,3} }{ 6 \gamma_{s,2}^{3/2} }  
  +  \frac{ \gamma_{s,4} \gamma_{s,2} - 3 \gamma_{s,3}^2 }{ 24 \gamma_{s,2}^3 } t
  +  \frac{\gamma_{s,5} \gamma_{s,2}^2 - 10 \gamma_{s,4} \gamma_{s,3} \gamma_{s,2} + 15 \gamma_{s,3}^3 }{ 120 \gamma_{s,2}^{9/2} } t^2
  +  \ldots, 
\end{align*}
which converges for small enough $|t|$. 

\begin{lemma} \label{lemmaCR001}
Assume either conditions \ref{Condi_Exp} and \ref{Condi-IP} when $s \in I_{\mu}^{\circ}$, 
or conditions \ref{Condi-TwoExp} and \ref{Condi-IP} when $s \in (-s_0,0)$ with small enough $s_0>0$. 
Let $q=\Lambda'(s)$. 
Then, there exists a constant $\eta >0$ such that for any $|l|\leq \eta,$
\begin{align*} 
\Lambda^*(q+l) = \Lambda^{*}(q) + sl +  h_s(l), 
\end{align*}
where  
$h_s$ is linked to the Cram\'{e}r series $\zeta_s$ by the identity
\begin{align*} 
h_s(l) = \frac{ l^2}{2 \sigma_s^2} - \frac{l^3}{\sigma_s^3} \zeta_s( \frac{l}{\sigma_s} ).
\end{align*}
\end{lemma}

In the sequel let us fix a non-negative density function $\rho$ on $\mathbb{R}$ 
with $\int_{\mathbb{R}} \rho(u) du = 1$, 
whose Fourier transform $\widehat{\rho}$ is supported on $[-1,1]$. 
Moreover, there exists a constant $C>0$ such that $\rho(u) \leq \frac{C}{1+u^4}$ for all $u \in \bb R$. 
For any $\ee>0$, define the scaled density function $\rho_{\ee}$ by
$\rho_{\ee}(u) = \frac{1}{\ee}\rho(\frac{u}{\ee})$, $u\in\mathbb R,$ 
whose Fourier transform $\widehat{\rho}_{\ee}$ is supported on $[-\ee^{-1},\ee^{-1}]$. 
For any non-negative integrable function $\psi$ on $\mathbb{R}$,
we introduce two modified functions related to $\psi$ 
as follows:  for any  $u \in \mathbb{R}$, 
set $\mathbb{B}_{\ee}(u) = \{u' \in\mathbb{R}: |u' - u| \leq \ee\}$  and  
\begin{align}\label{smoo001}
\psi^+_{\ee}(u) = \sup_{u' \in \mathbb{B}_{\ee}(u)} \psi(u') 
\quad  \text{and}  \quad 
\psi^-_{\ee}(u) = \inf_{u' \in \mathbb{B}_{\ee}(u)} \psi(u'). 
\end{align}
The following smoothing inequality gives two-sided bounds of $\psi$. 
\begin{lemma}  \label{estimate u convo}
Suppose that  $\psi$ is a non-negative integrable function on $\bb R$ and that 
$\psi^+_{\ee}$ and $\psi^-_{\ee}$ are measurable for any $\ee>0$. 
Then, for $0< \ee <1$, 
there exists a positive constant $C_{\rho}(\ee)$ with $C_{\rho}(\ee) \to 0$ as $\ee \to 0$,
such that for any $u\in \mathbb{R}$, 
\begin{align}
{\psi}^-_{\ee}\!\ast\!\rho_{\ee^2}(u) - 
\int_{|w|\geq \ee} {\psi}^-_{\ee}(u - w) \rho_{\ee^2}(w)dw
\leq \psi(u) \leq (1+ C_{\rho}(\ee))
{\psi}^+_{\ee}\!\ast\!\rho_{\ee^2}(u). \nonumber
\end{align}
\end{lemma}
The proof of the above lemma is similar to that of Lemma 5.2 in \cite{GLL17}, and will not be detailed here.

The next proposition gives precise asymptotics of the perturbed operator $R_{s,it}$,
which will be used to establish Bahadur-Rao-Petrov type large deviations 
for the coefficients $\langle f, G_n v \rangle$. 
Its proof is based on the spectral gap properties of the perturbed operator $R_{s,it}$.  

\begin{proposition} \label{Prop Rn limit1}  
Suppose that $\psi: \mathbb R \mapsto \mathbb C$
is bounded measurable function with compact support,
and that $\psi$ is differentiable in a small neighborhood of 0 in $\bb R$. 

\noindent(1)
Assume conditions \ref{Condi_Exp} and \ref{Condi-IP}. 
Then, for any compact set $K_{\mu} \subset I_{\mu}^{\circ}$, there exist constants 
$\delta = \delta(K) >0$, $c = c(K) >0$, $C = C(K) >0$ such that
for all $x\in \bb P^{d-1}$, $s \in K_{\mu}$, $|l| = O( \frac{1}{\sqrt{n}})$, $\varphi \in \mathcal{B}_{\gamma}$
and $n \geq 1$, 
\begin{align} \label{Thm1 lim R1}
 &  \left|  \sigma_s  \sqrt{n}  \,  e^{\frac{nl^2}{2 \sigma_s^2}}
\int_{\mathbb R} e^{-it l n} R^{n}_{s,it}(\varphi)(x) \psi (t) dt
- \sqrt{2\pi} \psi(0)\pi_{s}(\varphi)  \right|   \nonumber\\
& \leq  \frac{ C }{ \sqrt{n} } \| \varphi \|_\gamma 
  + \frac{C}{n} \|\varphi\|_{\gamma} \sup_{|t| \leq \delta} \big( |\psi(t)| + |\psi'(t)| \big)
  + Ce^{-cn} \|\varphi\|_{\gamma} \int_{\bb R} |\psi(t)| dt. 
\end{align}

\noindent(2)
Assume conditions \ref{Condi-TwoExp} and \ref{Condi-IP}. 
Then, there exist constants $s_0 > 0$, $\delta = \delta(s_0) >0$, $c = c(s_0) >0$, $C = C(s_0) >0$
such that for any compact set $K_{\mu} \subset (-s_0, 0)$, 
the inequality \eqref{Thm1 lim R1} holds uniformly in $x\in \bb P^{d-1}$, 
$s \in K_{\mu}$, $|l| = O( \frac{1}{\sqrt{n}} )$, $\varphi \in \mathcal{B}_{\gamma}$ and $n \geq 1$. 
\end{proposition}

The assertions (1) and (2) of Proposition \ref{Prop Rn limit1}
were respectively established in \cite{XGL19a} and \cite{XGL19b}.
The perturbation $l$ as well as the explicit rate of convergence in Proposition \ref{Prop Rn limit1} 
are important in the sequel. 
They play a crucial role to establish the Bahadur-Rao type large deviations 
for the coefficients $\langle f, G_n v \rangle$
in Theorems \ref{thrmBR001}, \ref{Thm_BRP_Upper}, \ref{Thm-Posi-Neg-s} and \ref{Thm-Posi-Neg-sBRP}.

\section{Proof of upper tail large deviations for coefficients} \label{sec proof scalarprod}

The aim of this section is to establish 
Theorems \ref{thrmBR001} and \ref{Thm_BRP_Upper}. 
Since Theorems \ref{thrmBR001} is a 
direct consequence of Theorem \ref{Thm_BRP_Upper},
it suffices to establish Theorem \ref{Thm_BRP_Upper}. 
We also establish a large deviation result under the changed measure. 

\subsection{Zero-one laws for the stationary measure}

We first present some zero-one laws for the stationary measure which will be used in the proof of 
Theorem \ref{Thm_BRP_Upper}. 


\begin{lemma}  \label{Lem_Gui_LeP}
Assume condition \ref{Condi-IP}.   
Then, for any $s \in I_{\mu}^{\circ}$
 and any proper projective subspace $Y \subsetneq \mathbb P^{d-1}$, 
it holds that $\pi_s(Y)=0$.
\end{lemma}

\begin{lemma}\label{Lem_zero_law_s_Neg}
Assume conditions \ref{Condi-TwoExp} and \ref{Condi-IP}. 
Then, there exists a constant $s_0 >0$ such that for any 
$s \in (-s_0, 0)$ and any proper projective subspace $Y \subsetneq \mathbb P^{d-1}$, 
it holds that $\pi_s( Y)=0$. 
\end{lemma}

Lemma \ref{Lem_Gui_LeP} was established by Guivarc'h and Le Page \cite{GL16} using the
strategy of Furstenberg \cite{Fur63}.  
Lemma \ref{Lem_zero_law_s_Neg} was proved in \cite{GQX20} based on the H\"{o}lder regularity of the 
stationary measure $\nu$.
Note that the results in \cite{GL16} and \cite{GQX20} are stated for the eigenmeasure $\nu_s$,
but they also hold for the stationary measure $\pi_s$ since the measures $\pi_s$ and $\nu_s$ are equivalent.

We shall also need the following zero-one law of the stationary measure $\pi_s$ 
recently established  in \cite{GQX20}.

\begin{lemma}\label{Lem_0-1_law_s}
Assume condition \ref{Condi-IP}. 
Then, for any $s \in I_{\mu}^{\circ}$ and any algebraic subset $Y$ of $\mathbb P^{d-1}$, 
it holds that either $\pi_s( Y)=0$ or $\pi_s( Y)=1$. 
In particular, for any $y \in (\bb P^{d-1})^*$ and any $t \in (-\infty, 0)$, 
\begin{align}  \label{0-1_law_s_Posi_Ex}
\pi_s \left( \left\{x \in \mathbb{P}^{d-1}:  \log \delta(y, x) =  t \right\}  \right)
= 0 \  \mbox{or} \  1.
\end{align}
\end{lemma}

\begin{lemma}\label{Lem_0-1_law_s_Neg}
Assume conditions \ref{Condi-TwoExp} and \ref{Condi-IP}. 
Then, there exists a constant $s_0 >0$ such that for any 
$s \in (-s_0, 0)$ and any algebraic subset $Y$ of $\mathbb P^{d-1}$, 
it holds that either $\pi_s( Y)=0$ or $\pi_s( Y)=1$. 
In particular, for any $y \in (\bb P^{d-1})^*$ and any $t \in (-\infty, 0)$, 
\begin{align}  \label{0-1_law_s_Nega_Ex}
\pi_s  \left(  \left\{x \in \mathbb{P}^{d-1}:  \log \delta(y, x) =  t  \right\}  \right)
= 0 \  \mbox{or} \  1.
\end{align}
\end{lemma}

The assertions \eqref{0-1_law_s_Posi_Ex} and \eqref{0-1_law_s_Nega_Ex} 
are sufficient for us to establish the Bahadur-Rao type large deviation asymptotics 
for the coefficient $\langle f, G_n v \rangle$ (cf. Theorems \ref{thrmBR001} and \ref{Thm-Posi-Neg-s}). 
However, in order to obtain the Petrov type extensions (cf. Theorems \ref{Thm_BRP_Upper} and \ref{Thm-Posi-Neg-sBRP}), 
we need the following slightly stronger statements than \eqref{0-1_law_s_Posi_Ex} and \eqref{0-1_law_s_Nega_Ex}.

\begin{lemma}\label{Lem_0-1_law_s_Posi_Uni}
Assume condition \ref{Condi-IP}. 
Then, for any $y \in (\bb P^{d-1})^*$ and any $t \in (-\infty, 0)$, if
\begin{align}  \label{0-1_law_s_Posi_uni}
\pi_s \left( \left\{x \in \mathbb{P}^{d-1}:  \log \delta(y, x) =  t \right\}  \right)
= 0 
\end{align}
holds for some $s \in I_{\mu}^{\circ}$, then \eqref{0-1_law_s_Posi_uni} holds for all $s \in I_{\mu}^{\circ}$. 
\end{lemma}

\begin{lemma}\label{Lem_0-1_law_s_Neg_Uni}
Assume conditions \ref{Condi-TwoExp} and \ref{Condi-IP}. 
Then, for any $y \in (\bb P^{d-1})^*$ and any $t \in (-\infty, 0)$, if 
\begin{align}  \label{0-1_law_s_Nega_uni}
\pi_s  \left(  \left\{x \in \mathbb{P}^{d-1}:  \log \delta(y, x) =  t  \right\}  \right)
= 0
\end{align}
holds for some $s \in (-s_0, 0)$ with $s_0 >0$ small enough, 
then \eqref{0-1_law_s_Nega_uni} holds for all $s \in (-s_0, 0]$.
\end{lemma}

\begin{proof}[Proof of Lemmas \ref{Lem_0-1_law_s_Posi_Uni} and \ref{Lem_0-1_law_s_Neg_Uni}]
We first prove Lemma \ref{Lem_0-1_law_s_Posi_Uni}. 
For any $y \in (\bb P^{d-1})^*$ and any $t \in (-\infty, 0)$,
denote $Y_{y, t} = \{x \in \mathbb{P}^{d-1}:  \log \delta(y, x) =  t \}$.
Suppose that there exist $s_1, s_2 \in I_{\mu}^{\circ}$ with $s_1 \neq s_2$ such that
$\pi_{s_1}(Y_{y, t}) = 0$ and $\pi_{s_2}(Y_{y, t}) \neq 0$.
Then by Lemma \ref{Lem_0-1_law_s} we have $\pi_{s_2}(Y_{y, t}) = 1$.
Since $Y_{y, t}$ is a closed set in $\bb P^{d-1}$,
by the definition of the support of the measure we get that $\supp \pi_{s_2} \subset Y_{y, t}$.
Since it is proved in \cite{GL16} that $\supp \pi_{s_1} = \supp \pi_{s_2}$ 
(both coincide with $\supp \nu$ defined by \eqref{Def_supp_nu}),
it follows that $\supp \pi_{s_1} \subset Y_{y, t}$ and hence $\pi_{s_1} (Y_{y, t}) = 1$.
This contradicts to the assumption $\pi_{s_1} (Y_{y, t}) = 0$. 
Therefore, if \eqref{0-1_law_s_Posi_uni} holds for some $s \in I_{\mu}^{\circ}$,
then it holds for all $s \in I_{\mu}^{\circ}$. 

The proof of Lemma \ref{Lem_0-1_law_s_Neg_Uni} is similar by using the fact that
$\supp \pi_s = \supp \nu$ for any $s \in (-s_0, 0)$, which is proved in \cite{GQX20}.  
\end{proof}

\subsection{Proof of Theorem \ref{Thm_BRP_Upper}
}
Now we are equipped to establish Theorem \ref{Thm_BRP_Upper}. 
This theorem is a direct consequence of the following more general result. 
Recall that $s$ and $q$ are related by $q = \Lambda'(s)$. 

\begin{theorem} \label{Thm_BRP_Uni_s}
Assume conditions \ref{Condi_Exp} and \ref{Condi-IP}. 
Let $K_{\mu} \subset I_{\mu}^{\circ}$ be any compact set in $\bb R$. 
Then, 
we have, 
as $n \to \infty$, uniformly in $s \in K_{\mu}$, 
$f \in (\bb R^d)^*$ and $v \in \bb R^d$ with $|f| = |v| = 1$, 
\begin{align}  \label{SCALREZ02_Uni_s_0000}  
 \mathbb{P} \Big( \log | \langle f, G_n v \rangle | \geq nq \Big)  
 =  \frac{ r_{s}(x)  r^*_{s}(y)}{\varrho_s} 
\frac{ \exp \left( -n   \Lambda^*(q) \right) } {s \sigma_{s}\sqrt{2\pi n}} \big[ 1 + o(1) \big].  
\end{align}
More generally, for any measurable function $\psi$ on $\mathbb{R}$ 
such that $u \mapsto e^{-s'u}\psi(u)$  is directly Riemann integrable 
for any $s' \in K_{\mu}^{\epsilon} : = \{ s' \in \bb R: |s' - s| < \epsilon, s \in K_{\mu} \}$ 
with $\epsilon >0$ small enough, 
we have, as $n \to \infty$, uniformly in $s \in K_{\mu}$, 
  $f \in (\bb R^d)^*$ and $v \in \bb R^d$ with $|f| = |v| = 1$, 
\begin{align} 
&  \mathbb{E} \Big[ \varphi(G_n x) \psi \big( \log |\langle f, G_n v \rangle| - nq \big) \Big]
  \label{SCALREZ02_Uni_s}    \\
&   =  \frac{r_{s}(x)}{\varrho_s}     
 \frac{ \exp (-n \Lambda^*(q)) }{ \sigma_{s}\sqrt{2\pi n}}  
 \left[ \int_{\bb P^{d-1}} \varphi(x) \delta(y,x)^s \nu_s(dx)
       \int_{\mathbb{R}} e^{-su} \psi(u) du + o(1) \right].   \nonumber
\end{align}
\end{theorem}


\begin{proof}


It suffices to prove 
assertion  \eqref{SCALREZ02_Uni_s},  
since
   \eqref{SCALREZ02_Uni_s_0000}   
follows from 
  \eqref{SCALREZ02_Uni_s}  
by choosing $\varphi = \mathbf{1}$ 
and $\psi (u) = \mathbbm{1}_{ \{ u\geq 0 \} },$ $u \in \mathbb R.$ 
With no loss of generality, we assume that the target functions $\varphi$ and $\psi$
are non-negative. 
For brevity, denote $\psi_s(u)=e^{-su}\psi(u)$ for $s \in 
 I_{\mu}^{\circ}$,  
and
\begin{align*}
\psi^+_{s,\ee} (u) = \sup_{u'\in\mathbb{B}_{\ee}(u)} \psi_s(u'),
\quad
\psi^-_{s,\ee}(u) = \inf_{u'\in\mathbb{B}_{\ee}(u)} \psi_s(u').
\end{align*}
Introduce the following condition:
for any $\ee>0,$ the functions
 $u \mapsto \psi^+_{s,\ee}(u)$ 
 and  $u \mapsto \psi^-_{s,\ee}(u)$ 
 are measurable and
\begin{align}\label{condition g}
\lim_{\ee \to 0^{+}} \int_{\mathbb{R}} \psi^+_{s,\ee}(u) du
= \lim_{\ee \to 0^{+}} \int_{\mathbb{R}} \psi^-_{s,\ee}(u) du 
=\int_{\mathbb{R}} e^{-su} \psi(u) du < +\infty.
\end{align}

To prove 
\eqref{SCALREZ02_Uni_s},  
we can assume additionally
that the function $\psi$ satisfies the condition \eqref{condition g}, 
In fact,  using the approximation techniques similar to that in \cite{XGL19a}, 
we can prove that if  
\eqref{SCALREZ02_Uni_s}  
 holds under \eqref{condition g}, then it also holds 
under  the directly Riemann integrability condition introduced in the theorem.
So in the following we assume  \eqref{condition g}.

Note that we have $f \in (\bb R^d)^*$ and $v \in \bb R^d$ with $|f| = |v| = 1$,
and $y = \bb R f \in (\bb P^{d-1})^*$ and $x = \bb R v \in \bb P^{d-1}$. 
Hence $\log |\langle f, G_n v \rangle| = \log |G_n v| + \log \delta(y, G_n x),$ 
and that $\log |\langle f, G_n v \rangle| = - \infty$ if and only if $\log \delta(y, G_n x) = -\infty$. 
Taking into account that $\psi(-\infty) = 0$, we can replace
 the logarithm of the coefficient $\log |\langle f, G_n v \rangle|$ 
by the sum $\log |G_n v| + \log \delta(y, G_n x)$ as follows: 
\begin{align*}
 A: & = \sigma_s \sqrt{2\pi n}  \frac{ e^{n \Lambda^*(q)} }{r_s(x)} 
\mathbb{E} \Big[ \varphi(G_n x)\psi( \log |\langle f, G_n v \rangle| - n q ) \Big] \nonumber\\
&  =   \sigma_s \sqrt{2\pi n}   
\frac{ e^{n \Lambda^*(q)} }{r_s(x)} \mathbb{E} \Big[ \varphi(G_n x) 
  \psi( \log |G_n x| + \log \delta(y, G_n x) - nq ) \Big]. 
\end{align*}
For short, we denote for any $y = \bb R f \in (\bb P^{d-1})^*$ and $x \in \bb R v \in \bb P^{d-1}$, 
\begin{align*}
T_n^v: = \log |G_n v| - nq,  \qquad  Y_n^{x,y}: = \log \delta(y, G_n x).  
\end{align*}
Recall that $q = \Lambda'(s)$. 
Taking into account that  $e^{n\Lambda^{*}(q)} = e^{nsq} \kappa^{-n}(s)$
and using the change of measure formula \eqref{basic equ1}, 
we get
\begin{align} \label{ScaProLimAn 01}
A =  \sigma_s \sqrt{2\pi n}  
  \mathbb{E}_{\mathbb{Q}_{s}^{x}} \Big[ (\varphi r_{s}^{-1})(G_n x)  e^{-s T_n^v}
\psi \big( T_{n}^v + Y_n^{x,y} \big)  \Big].
\end{align}
For any fixed small constant $0< \eta <1$, denote $I_k: = (-\eta k, -\eta(k-1)]$,  $k \geq 1$. 
Let $M_n:= \floor{ C_1 \log n }$, where $C_1>0$ is a sufficiently large constant
and $\floor{a}$ denotes the integer part of $a \in \bb R$. 
Then from \eqref{ScaProLimAn 01} we have the following decomposition: 
\begin{align}\label{PosiScalA_ccc}
A = A_1 + A_2, 
\end{align}
where  
\begin{align*}
& A_1 : =  \sigma_s \sqrt{2\pi n}  \mathbb{E}_{\mathbb{Q}_{s}^{x}} 
\left[ (\varphi r_{s}^{-1})(G_n x)  e^{-s T_n^v}
\psi \big( T_{n}^v + Y_n^{x,y} \big) \mathbbm{1}_{\{Y_n^{x, y} \leq -\eta M_n \}} \right],  
    \nonumber\\
& A_2 : =  \sigma_s \sqrt{2\pi n}  \sum_{k =1}^{M_n}  
\mathbb{E}_{\mathbb{Q}_{s}^{x}} 
\left[ (\varphi r_{s}^{-1})(G_n x)  e^{-s T_n^v}
\psi \big( T_{n}^v + Y_n^{x,y} \big) \mathbbm{1}_{\{Y_n^{x,y} \in I_k \}} \right]. 
\end{align*}
We now give a bound for the first term $A_1$.
Since the function $u \mapsto e^{-s' u} \psi(u)$ is directly Riemann integrable on $\mathbb{R}$ for 
any $s' \in K_{\epsilon} : = \{ s' \in \bb R: |s' - s| < \epsilon, s \in K \}$ with $\epsilon >0$ small enough, 
one can verify that the function $u \mapsto e^{-s u} \psi(u)$ is bounded on $\mathbb{R}$, uniformly in $s \in K_{\mu}$, 
and hence there exists a constant $C >0$ such that for all $s \in K_{\mu}$, 
\begin{align*}
e^{-s T_n^v} \psi( T_{n}^v + Y_n^{x,y} )  \mathbbm{1}_{\{Y_n^{x, y} \leq -\eta M_n \}} 
\leq  C  e^{ s Y_n^{x,y} }  \mathbbm{1}_{\{Y_n^{x, y} \leq -\eta M_n \}} 
\leq  C e^{- s \eta M_n}. 
\end{align*}
Since the function $\varphi r_s^{-1}$ is uniformly bounded on $\bb P^{d-1}$, uniformly in $s \in K_{\mu}$, 
we get the following upper bound for $A_1$: as $n \to \infty$, uniformly in $s \in K_{\mu}$, 
  $f \in (\bb R^d)^*$ and $v \in \bb R^d$ with $|f| = |v| = 1$, 
\begin{align} \label{ScalInverAn 2}
A_1 \leq  C  \sqrt{n} \,  e^{ - s \eta M_n}  \leq C n^{- (s \eta C_1 - \frac{1}{2})} \to 0. 
\end{align}
The remaining part of the proof is devoted to 
establishing upper and lower bounds for the second term $A_2$ defined by \eqref{PosiScalA_ccc}. 

\textit{Upper bound for $A_2$.} 
On the event $\{ Y_n^{x, y} \in I_k \}$,  
we have $Y_n^{x, y} + \eta(k-1) \in (0, \eta]$. With the notation 
$\psi^+_{\eta} (u) = \sup_{u' \in \mathbb{B}_{\eta}(u)} \psi(u')$, we get
\begin{align*}
\psi \big( T_{n}^v  - nl + Y_n^{x,y} \big) 
\leq  \psi_{\eta}^+ \big( T_{n}^v  -nl - \eta(k-1) \big). 
\end{align*}
It follows that
\begin{align*}
&  A_2  \leq  \sigma_s \sqrt{2\pi n}   \sum_{k =1}^{M_n} 
\mathbb{E}_{\mathbb{Q}_{s}^{x}} \left[ (\varphi r_{s}^{-1})(G_n x) e^{-s T_n^v}
\psi_{\eta}^+( T_{n}^v  - \eta(k-1))  \mathbbm{1}_{\{ Y_n^{x, y} \in I_k \}} \right].
\end{align*}
We choose a small constant $\ee > \eta$ and set
\begin{align}\label{Def_Psi_aaa}
\Psi_{s,\eta} (u) = e^{-su} \psi_{\eta}^+(u), 
\quad 
\Psi^+_{s, \eta, \ee}(u) = \sup_{u'\in\mathbb{B}_{\ee}(u)} \Psi_{s,\eta} (u'),  
\quad 
u \in \bb R. 
\end{align}
Since the function $\Psi^+_{s, \eta, \ee}$ is non-negative and integrable on the real line, 
using Lemma \ref{estimate u convo}, we get
\begin{align} \label{ScaProLimAn Bn 01}
& A_2   \leq   ( 1+ C_{\rho}(\ee) )
\sigma_s \sqrt{2\pi n}  \sum_{k =1}^{\infty} \mathbbm{1}_{ \{ k \leq M_n \} }  e^{-s\eta (k-1)}  \nonumber\\
& \qquad   \times  \mathbb{E}_{\mathbb{Q}_{s}^{x}}
\left[(\varphi r_{s}^{-1})(G_n x) \mathbbm{1}_{ \{ Y_n^{x,y} \in I_k \} } 
({\Psi}^+_{s, \eta, \ee}\!\ast\!\rho_{\ee^2})
(T_{n}^v - \eta(k-1))\right], 
\end{align} 
where $C_{\rho}(\ee) >0$ is a constant converging to $0$ as $\ee \to 0$. 
For fixed small constant $\ee_1 >0$,  
introduce the density function $\bar{\rho}_{\ee_1}$ defined as follows: 
$\bar{\rho}_{\ee_1}(u) = \frac{1}{\ee_1}(1 - \frac{|u|}{\ee_1}) $
for $u \in [-\ee_1, \ee_1]$, and $\bar{\rho}_{\ee_1}(u) = 0$ otherwise.
For any $k \geq 1$, 
setting $\chi_k(u) := \mathbbm{1}_{\{u \in I_k \}}$
and $\chi_{k, \ee_1}^+(u) = \sup_{u' \in \mathbb{B}_{\ee_1}(u)} \chi_k(u')$,
one can verify that the following smoothing inequality holds: 
\begin{align} \label{Pf_LD_SmoothIneHolder01}
\chi_k(u) \leq 
(\chi_{k, \ee_1}^+ * \bar{\rho}_{\ee_1})(u)
\leq \chi_{k, 2\ee_1}^+(u), \quad  u \in \mathbb{R}.
\end{align}
For short, we denote 
$\tilde\chi_k(u):= (\chi_{k, \ee_1}^+ * \bar{\rho}_{\ee_1})(u)$, $u \in \mathbb{R}$,
 and 
 \begin{align}\label{Pf_Baradur_LD_varph_11}
\varphi_{s,k,\ee_1}^y(x) = (\varphi r_s^{-1})(x) \tilde\chi_k(\log \delta(y, x)), 
\quad  x \in \bb{P}^{d-1}. 
\end{align}
In view of \eqref{ScaProLimAn Bn 01}, using the smoothing inequality \eqref{Pf_LD_SmoothIneHolder01} leads to 
\begin{align}  \label{ScalarBn a}
  A_2  & \leq   (1+ C_{\rho}(\ee))
\sigma_s \sqrt{2\pi n}  \sum_{k =1}^{\infty} \mathbbm{1}_{ \{ k \leq M_n \} } 
   e^{-s\eta (k-1)}    \nonumber \\
&  \quad \times  \mathbb{E}_{\mathbb{Q}_{s}^{x}}
  \left[ \varphi_{s,k,\ee_1}^y(G_n x)
 ({\Psi}^+_{s, \eta, \ee}\!\ast\!\rho_{\ee^2})
 (T_{n}^v - \eta(k-1))\right]    \nonumber \\
&  =: A_2^+.
\end{align}
Let $\widehat{{\Psi}}^+_{s, \eta,\ee}$ be the Fourier transform of ${\Psi}^+_{s,\eta, \ee}$.
By the Fourier inversion formula,
\begin{align*}
{\Psi}^+_{s, \eta, \ee}\!\ast\!\rho_{\ee^{2}}(u)
=\frac{1}{2\pi}\int_{\mathbb{R}}e^{itu}
\widehat {\Psi}^+_{s, \eta, \ee}(t) \widehat\rho_{\ee^{2}}(t)dt,  
\quad  u \in \mathbb{R}. 
\end{align*}
Substituting $y=T_{n}^v - nl - \eta (k-1)$, taking expectation with respect to $\mathbb{E}_{\mathbb{Q}_{s}^{x}},$
and using Fubini's theorem, we obtain 
\begin{align} \label{ScalarFourier a}
&  \mathbb{E}_{\mathbb{Q}_{s}^{x}}
\left[ \varphi_{s,k,\ee_1}^y(G_n x)
({\Psi}^+_{s, \eta, \ee}\!\ast\!\rho_{\ee^2})
(T_{n}^v - \eta(k-1))\right]  \nonumber\\
&  =   \frac{1}{2\pi} \int_{\mathbb{R}} e^{-it \eta(k-1)} 
R^{n}_{s,it} \big( \varphi_{s,k,\ee_1}^y  \big)(x)
\widehat {\Psi}^+_{s, \eta, \ee}(t) \widehat\rho_{\ee^{2}}(t) dt,
\end{align}
where
\begin{align}
R^{n}_{s,it} \big( \varphi_{s,k,\ee_1}^y  \big)(x)
=\mathbb{E}_{\mathbb{Q}_{s}^{x}}\left[e^{it T_{n}^v} \varphi_{s,k,\ee_1}^y(G_n x)
   \right].  \nonumber
\end{align}
Substituting \eqref{ScalarFourier a} into \eqref{ScalarBn a}, 
we get
\begin{align} \label{Pf_LDScalA2_2}
& A_2^+  =  (1+ C_{\rho}(\ee))
 \sigma_s \sqrt{\frac{n}{2\pi}} \,  
 \sum_{k =1}^{\infty} \mathbbm{1}_{ \{ k \leq M_n \} }  e^{-s\eta (k-1)}   \nonumber\\
& \qquad \quad \times   \int_{\mathbb{R}} e^{-it \eta(k-1)} 
R^{n}_{s,it} \big( \varphi_{s,k,\ee_1}^y  \big)(x)
\widehat {\Psi}^+_{s, \eta, \ee}(t) \widehat\rho_{\ee^{2}}(t) dt.
\end{align}
We shall use Proposition \ref{Prop Rn limit1} to handle the integral in \eqref{Pf_LDScalA2_2} for each fixed $k \geq 1$.
Let us first check the conditions stated in Proposition \ref{Prop Rn limit1}. 
Since  the function $\tilde \chi_k$ is H\"{o}lder continuous on the real line, 
one can check that $\varphi_{s,k,\ee_1}^y$ defined by \eqref{Pf_Baradur_LD_varph_11} is also H\"{o}lder continuous
on the projective space $\bb P^{d-1}$. 
Using the fact that the function 
$u \mapsto e^{-s'u}\psi(u)$  is directly Riemann integrable on $\mathbb{R}$ for 
any $s' \in K_{\epsilon} : = \{ s' \in \bb R: |s' - s| < \epsilon, s \in K \}$ with $\epsilon >0$ small enough, 
one can also verify that the function $\widehat {\Psi}^+_{s, \eta, \ee} \widehat\rho_{\ee^{2}}$ 
is compactly supported in $\mathbb{R}$. 
Moreover, for any $s \in K_{\mu}$, the function $\widehat {\Psi}^+_{s, \eta, \ee} \widehat\rho_{\ee^{2}}$
is differentiable in a small neighborhood of $0$ on the real line. 
Hence, applying Proposition \ref{Prop Rn limit1}
with $\varphi = \varphi_{s,k,\ee_1}^y$,
$\psi = \widehat {\Psi}^+_{s, \eta, \ee} \widehat\rho_{\ee^{2}}$
and with $l = l_{n,k}: =  \frac{\eta(k-1)}{n}$, 
we obtain that for sufficiently large $n$, uniformly in $1 \leq k \leq M_n$, 
$s \in K_{\mu}$, 
$f \in (\bb R^d)^*$ and $v \in \bb R^d$ with $|f| = |v| = 1$,
\begin{align}  \label{ScalarKeyProp 1}
I_k^+ : & =  \Big|  \sigma_s  \sqrt{n}e^{n h_s(l_{n,k})} 
\int_{\mathbb R} e^{-it n l_{n,k}} 
R^{n}_{s,it} \big( \varphi_{s,k,\ee_1}^y \big)(x)
\widehat {\Psi}^+_{s, \eta, \ee}(t) \widehat\rho_{\ee^{2}}(t) dt
 -  B^{+}(k)  \Big|  \nonumber\\
&  \leq  \frac{C}{\sqrt{n}}  \| \varphi_{s,k,\ee_1}^y \|_{\gamma}, 
\end{align}
where 
\begin{align*}
B^{+}(k) : = \sqrt{2\pi} \widehat {\Psi}^+_{s, \eta, \ee}(0) \widehat\rho_{\ee^{2}}(0)
\pi_{s} \big( \varphi_{s,k,\ee_1}^y \big). 
\end{align*}
Taking into account that $1 \leq k \leq M_n = \floor{ C_1 \log n }$, 
by Lemma \ref{lemmaCR001},  we get that 
$|e^{ - n h_s(l_{n,k})} - 1| \leq \frac{C \log n}{\sqrt{n}}$,  
uniformly in $1 \leq k \leq M_n$ and $s \in K_{\mu}$. 
Using \eqref{ScalarKeyProp 1} and the fact that $B^{+}(k)$ is dominated by 
$\| \varphi_{s,k,\ee_1}^y \|_{\gamma}$, 
we can replace $e^{n h_s(l_{n,k})}$ by $1$, yielding that 
uniformly in $s \in K_{\mu}$, $f \in (\bb R^d)^*$ and $v \in \bb R^d$ with $|f| = |v| = 1$, 
\begin{align}  \label{ScalarKeyProp 2}
&  \Big| \sigma_s  \sqrt{n}   
\int_{\mathbb R} e^{-it n l_{n,k} }  R^{n}_{s,it} \big( \varphi_{s,k,\ee_1}^y \big)(x)
\widehat {\Psi}^+_{s, \eta, \ee}(t) \widehat\rho_{\ee^{2}}(t) dt - B^{+}(k) \Big|  \nonumber\\
& \leq  I_k^+  e^{ - n h_s(l_{n,k})}  + |e^{ - n h_s(l_{n,k})} - 1| B^{+}(k)   \nonumber\\
&  \leq  \frac{C}{\sqrt{n}}  \| \varphi_{s,k,\ee_1}^y \|_{\gamma}     
  +  \frac{C \log n}{\sqrt{n}}  \| \varphi_{s,k,\ee_1}^y  \|_{\gamma}    \nonumber\\
& \leq  \frac{C \log n}{\sqrt{n}}  \| \varphi_{s,k,\ee_1}^y  \|_{\gamma}. 
\end{align}
By calculations, one can get that $\gamma$-H\"{o}lder norm 
$\| \varphi_{s,k,\ee_1}^y \|_{\gamma}$
is bounded by 
$
\frac{ e^{ \eta \gamma k} }{ ( 1 - e^{-2\ee_1} )^{\gamma} }. 
$
Taking sufficiently small $\gamma>0$, we obtain that
 the series $ \frac{\log n}{\sqrt{n}} \sum_{k = 1}^{\infty} e^{-s\eta (k-1)}   \| \varphi_{s,k,\ee_1}^y  \|_{\gamma}$
is convergent, and moreover, its limit is $0$ as $n \to \infty$. 
Consequently, we are allowed to interchange the limit as $n \to \infty$
and the infinite summation over $k$ in \eqref{Pf_LDScalA2_2}. 
Therefore, from \eqref{Pf_LDScalA2_2}, \eqref{ScalarKeyProp 1} and \eqref{ScalarKeyProp 2} 
we deduce that,  uniformly in $s \in K_{\mu}$, 
$f \in (\bb R^d)^*$ and $v \in \bb R^d$ with $|f| = |v| = 1$, 
\begin{align} \label{ScalarBn abc}
\limsup_{n \to \infty}  A_2^+
\leq   (1+ C_{\rho}(\ee)) \widehat {\Psi}^+_{s, \eta, \ee}(0) \widehat\rho_{\ee^{2}}(0)
\sum_{k =1}^{\infty} e^{-s\eta (k-1)} 
\pi_{s} \big( \varphi_{s,k,\ee_1}^y \big). 
\end{align}
In order to calculate the sum in \eqref{ScalarBn abc}, 
we shall make use of the zero-one law of the stationary measure $\pi_s$. 
Note that 
$\widehat\rho_{\ee^{2}}(0) =1$. 
Using \eqref{Pf_LD_SmoothIneHolder01}, we have $\tilde \chi_k \leq \chi_{k, 2\ee_1}^+$. 
Therefore, we obtain 
\begin{align} \label{LimsuBn a}
  \limsup_{n \to \infty}  A_2^+ 
\leq (1+ C_{\rho}(\ee)) 
\widehat {\Psi}^+_{s, \eta, \ee}(0) 
\sum_{k =1}^{\infty} e^{-s\eta (k-1)} 
\pi_{s} \big( \tilde \varphi_{s,k,\ee_1}^y \big),  
\end{align}
where
\begin{align} \label{Pf_LDLimsuBn b}
\tilde \varphi_{s,k,\ee_1}^y (x)
= (\varphi r_s^{-1})(x) \mathbbm{1}_{\{ \log \delta(y, \cdot) \in I_k \}}(x) 
  + (\varphi r_s^{-1})(x)  \mathbbm{1}_{\{ \log \delta(y, \cdot) \in I_{k,\ee_1} \}}(x),
\end{align}
and  $I_{k,\ee_1} = \big(-\eta k -2 \ee_1, -\eta k \big]
\cup \big(-\eta (k-1), -\eta (k-1) + 2 \ee_1 \big]$. 
For the first term on the right hand-side of \eqref{Pf_LDLimsuBn b}, 
we claim that uniformly in $s \in K_{\mu}$, 
\begin{align} \label{ScalPosiMainpart}
& \lim_{\eta \to 0}
\sum_{k =1}^{\infty} e^{-s\eta (k-1)} 
\pi_{s} \Big( (\varphi r_{s}^{-1}) 
\mathbbm{1}_{\{ \log \delta(y, \cdot) \in I_k \}}   \Big)   \nonumber\\
& = \int_{ \bb P^{d-1} } \delta(y, x)^s \varphi(x) r_s^{-1}(x) \pi_s(dx).
\end{align}
Indeed, recalling that $I_k = (-\eta k, -\eta(k-1)]$, we have 
\begin{align*}
& \sum_{k =1}^{\infty} e^{-s\eta (k-1)} 
\pi_{s} \Big( (\varphi r_{s}^{-1}) 
\mathbbm{1}_{\{ \log \delta(y, \cdot) \in I_k \}}   \Big)  \nonumber\\
& \geq  \sum_{k =1}^{\infty} 
\pi_{s} \Big( (\varphi r_{s}^{-1})  \delta(y, \cdot)^s
\mathbbm{1}_{\{ \log \delta(y, \cdot) \in I_k \}}   \Big)   
 = \int_{ \bb P^{d-1} } \delta(y, x)^s \varphi(x) r_s^{-1}(x) \pi_s(dx).
\end{align*}
On the other hand, we have, as $\eta \to 0$, uniformly in $s \in K_{\mu}$, 
\begin{align*}
& \sum_{k =1}^{\infty} e^{-s\eta (k-1)} 
\pi_{s} \Big( (\varphi r_{s}^{-1}) 
\mathbbm{1}_{\{ \log \delta(y, \cdot) \in I_k \}}   \Big)  \nonumber\\
& \leq  e^{s \eta}  \sum_{k =1}^{\infty} 
\pi_{s} \Big( (\varphi r_{s}^{-1})  \delta(y, \cdot)^s
\mathbbm{1}_{\{ \log \delta(y, \cdot) \in I_k \}}   \Big) 
\to  \int_{ \bb P^{d-1} } \delta(y, x)^s \varphi(x) r_s^{-1}(x) \pi_s(dx).  
\end{align*}
Hence \eqref{ScalPosiMainpart} holds.

 To deal with the second term on the right-hand side of \eqref{Pf_LDLimsuBn b}, 
we need to apply Lemma \ref{Lem_Gui_LeP} and the zero-one law of the stationary measure $\pi_s$ 
stated in Lemma \ref{Lem_0-1_law_s}. 
Specifically, taking into account that the function $\varphi r_{s}^{-1}$ is uniformly bounded on
the projective space $\bb P^{d-1}$, using the Lebesgue dominated convergence theorem we get
that there exists a constant $C_1 >0$ such that for all $y \in (\bb P^{d-1})^*$, 
\begin{align}\label{Pf_Upp_A1_ff}
& \lim_{\ee_1 \to 0}
\sum_{k =1}^{\infty} e^{-s\eta (k-1)} 
\pi_{s} \Big( (\varphi r_{s}^{-1}) 
\mathbbm{1}_{\{ \log \delta(y, \cdot) \in I_{k,\ee_1} \}}   \Big)   \nonumber\\
& \leq C_1  \lim_{\ee_1 \to 0}
\sum_{k =1}^{\infty} e^{-s\eta (k-1)} 
\pi_{s} \Big( x \in \bb P^{d-1}: \log \delta(y, x) \in I_{k,\ee_1}  \Big)  \nonumber\\
& = C_1  \sum_{k =1}^{\infty} e^{-s\eta (k-1)} 
\pi_{s} \Big( x \in \bb P^{d-1}: \log \delta(y, x) = - \eta k  \Big)   \nonumber\\
& \quad  + C_1  \sum_{k =1}^{\infty} e^{-s\eta (k-1)} 
\pi_{s} \Big( x \in \bb P^{d-1}: \log \delta(y, x) = - \eta (k-1)  \Big)   \nonumber\\
& = 2 C_1  \sum_{k =1}^{\infty} e^{-s\eta (k-1)} 
\pi_{s} \Big( x \in \bb P^{d-1}: \log \delta(y, x) = - \eta k  \Big), 
\end{align}
where the last equality holds due to Lemma \ref{Lem_Gui_LeP}. 
Now we are going to apply Lemma \ref{Lem_0-1_law_s_Posi_Uni} to prove that 
there exists a constant $0< \eta < 1$ such that 
\begin{align}\label{sum_0_1_law}
\sum_{k =1}^{\infty} e^{-s\eta (k-1)} 
\pi_{s} \Big( x \in \bb P^{d-1}: \log \delta(y, x) = - \eta k  \Big) = 0. 
\end{align}
Indeed, by Lemma \ref{Lem_0-1_law_s_Posi_Uni}, we get that,  
for any  $y \in (\bb P^{d-1})^*$ and 
any set $Y_{y,t} = \{ x \in \bb P^{d-1}:  \log \delta(y,x) = t \}$ with $t \in (-\infty, 0)$, 
it holds that either $\pi_s (Y_{y,t}) = 0$ or $\pi_s (Y_{y,t}) = 1$ for all $s \in K_{\mu}$. 
If $\pi_s (Y_{y,t}) = 0$ for all $y \in (\bb P^{d-1})^*$ and $t \in (-\infty, 0)$, then clearly \eqref{sum_0_1_law} holds. 
If $\pi_s (Y_{y_0, t_0}) = 1$ for some $y_0 \in (\bb P^{d-1})^*$ and $t_0 \in (-\infty, 0)$,
then we can always choose $0< \eta <1$ in such a way that $- \eta k \neq t_0$ for all $k \geq 1$,
so that we also obtain that \eqref{sum_0_1_law} holds for all $s \in K_{\mu}$. 
Hence, in view of \eqref{LimsuBn a}, 
combining \eqref{ScalPosiMainpart} and \eqref{sum_0_1_law} we obtain that uniformly in $s \in K_{\mu}$, 
\begin{align}\label{Pf_LD_pi_s_hh}
\lim_{\eta \to 0} 
\lim_{\ee_1 \to 0} \sum_{k =1}^{\infty} e^{-s\eta (k-1)} 
\pi_{s} \big( \tilde \varphi_{s,k,\ee_1}^y \big)  
= \int_{ \bb P^{d-1} } \delta(y, x)^s \varphi(x) r_s^{-1}(x) \pi_s(dx). 
\end{align}
Since the target function $\psi$ satisfies the  condition \eqref{condition g}, 
 from \eqref{Def_Psi_aaa} we get
\begin{align} \label{Pf_LD_psi_limit}
\lim_{\ee \to 0}  \lim_{\eta \to 0}  \widehat {\Psi}^+_{s, \eta, \ee}(0) 
=  \int_{\mathbb{R}}  e^{-su} \psi(u) du. 
\end{align}
Consequently, recalling that $A_2 \leq  A_2^+$ and $C_{\rho}(\ee) \to 0$ as $\ee \to 0$,
we obtain the desired upper bound for $A_2$: uniformly in $s \in K_{\mu}$, 
$f \in (\bb R^d)^*$ and $v \in \bb R^d$ with $|f| = |v| = 1$, 
\begin{align} \label{ScaProLimBn Upper 01}
\lim_{\ee \to 0}  \lim_{\eta \to 0} 
\lim_{\ee_1 \to 0}  \limsup_{n \to \infty}  A_2 
\leq  \int_{\mathbb{R}}  e^{-su} \psi(u) du 
\int_{ \bb P^{d-1} } \delta(y, x)^s \varphi(x) r_s^{-1}(x) \pi_s(dx). 
\end{align}


\textit{Lower bound for $A_2$.} 
We are going to establish the lower bound for $A_2$ given by \eqref{PosiScalA_ccc}. 
Recall that $Y_n^{x,y} = \log \delta(y, G_n x)$. 
On the event $\{ Y_n^{x,y} \in I_k \}$ 
we have $Y_n^{x,y} + \eta(k-1) \in (0, \eta]$. With the notation 
$\psi^-_{\eta} (u) = \inf_{u'\in\mathbb{B}_{\eta}(u)} \psi(u)$, we get
\begin{align*}
\psi( T_{n}^v + Y_n^{x,y} )  \geq  \psi_{\eta}^- ( T_{n}^v - \eta k ). 
\end{align*}
In view of \eqref{PosiScalA_ccc}, using Fatou's lemma, it follows that
\begin{align}
 \liminf_{ n \to \infty}  A_2 
 & \geq    \sum_{k =1}^{\infty} 
    \liminf_{ n \to \infty}   
   \sigma_s \sqrt{2\pi n} \,   \mathbbm{1}_{ \{ k \leq M_n \} }   \label{Pf_Low_A2_dd}\\
 & \quad \times    
\mathbb{E}_{\mathbb{Q}_{s}^{x}} \Big[ (\varphi r_{s}^{-1})(G_n x) e^{-s T_n^v}
\psi^-_{\eta}( T_{n}^v - \eta k )   \mathbbm{1}_{\{Y_n^{x,y} \in I_k \}} \Big].  \nonumber
\end{align}
We choose a small constant $\ee > \eta$ and set
\begin{align}\label{Def_Psi_ccc}
\Psi_{s,\eta} (u) = e^{-su} \psi_{\eta}^-(u), 
\quad 
\Psi^-_{s, \eta, \ee}(u) = \inf_{u'\in\mathbb{B}_{\ee}(u)} \Psi_{s,\eta} (u'),  
\quad 
u \in \bb R. 
\end{align}
Noting that the function $\Psi^-_{s, \eta, \ee}$ is non-negative and integrable on the real line, 
by Lemma \ref{estimate u convo}, from \eqref{Pf_Low_A2_dd} we get the following lower bound: 
\begin{align} \label{Scal Low An 1}
\liminf_{ n \to \infty}  A_2  
\geq  \sum_{k =1}^{\infty} \liminf_{ n \to \infty} A_3 - 
\sum_{k =1}^{\infty} \limsup_{ n \to \infty}  A_4, 
\end{align} 
where, with the notation $a_{n,k} = \sigma_s \sqrt{2\pi n}  \,  e^{-s\eta k} \mathbbm{1}_{ \{ k \leq M_n \} }$, 
\begin{align*} 
& A_3   =  a_{n,k}       
  \mathbb{E}_{\mathbb{Q}_{s}^{x}}
\Big[ (\varphi r_{s}^{-1})(G_n x) \mathbbm{1}_{\{Y_n^{x,y} \in I_k \}} 
({\Psi}^-_{s, \eta, \ee}\!\ast\!\rho_{\ee^2})
(T_{n}^v - \eta k ) \Big],  \nonumber\\
& A_4   =   a_{n,k}   \int_{|u|\geq \ee}   \mathbb{E}_{\mathbb{Q}_{s}^{x}}
\Big[ (\varphi r_{s}^{-1})(G_n x) \mathbbm{1}_{\{Y_n^{x,y} \in I_k \}} 
 {\Psi}^-_{s, \eta, \ee}(T_{n}^v - \eta k - u) \Big] \rho_{\ee^2}(u) du. 
\end{align*} 
We are going to give a lower bound for $A_3$. 
For brevity, we denote $\chi_k(u)= \mathbbm{1}_{\{u \in I_k \}}$ 
and $\chi_{k, \ee_1}^-(u) = \inf_{u'\in\mathbb{B}_{\ee_1}(u)} \chi_k(u')$, $u \in \bb R$, 
where $\ee_1 >0$ is a fixed small constant.   
Similarly to \eqref{Pf_LD_SmoothIneHolder01}, 
one can get the following smoothing inequality: 
\begin{align}  \label{SmoothIne Holder 02}
\chi_{k, 2\ee_1}^-(u)
\leq  (\chi_{k, \ee_1}^- * \bar{\rho}_{\ee_1})(u)
\leq  \chi_k(u), \quad  u \in \mathbb{R},
\end{align}
where the density function $\bar{\rho}$ is the same as that in \eqref{Pf_LD_SmoothIneHolder01}. 
In a similar way as in \eqref{Pf_Baradur_LD_varph_11}, we denote
$\tilde\chi_k^-(u) := (\chi_{k, \ee_1}^- * \bar{\rho}_{\ee_1})(u)$, $u \in \bb R$, and   
 \begin{align}\label{Pf_Baradur_LD_varph_22}
\phi_{s,k,\ee_1}^y(x) = (\varphi r_s^{-1})(x) \tilde \chi_k^-(\log \delta(y, x)), 
\quad  x \in \bb{P}^{d-1}. 
\end{align}
For the first term $A_3$ in \eqref{Scal Low An 1}, 
using the inequality \eqref{SmoothIne Holder 02} leads to 
\begin{align}  \label{ScalarBn Low a}
 A_3 \geq  a_{n,k}   
  \mathbb{E}_{\mathbb{Q}_{s}^{x}}
\left[ \phi_{s,k,\ee_1}^y(G_n x)  ( {\Psi}^-_{s,\delta,\ee}\!\ast\!\rho_{\ee^2} )
(T_{n}^v - \eta k) \right]. 
\end{align}
Denote by 
$\widehat{{\Psi}}^-_{s,\eta,\ee}$ the Fourier transform of ${\Psi}^-_{s,\eta,\ee}$.
Applying the Fourier inversion formula to ${\Psi}^-_{s,\eta,\ee}\!\ast\!\rho_{\ee^{2}}$,
and using  Fubini's theorem, we get
\begin{align} \label{ScalarFourier Low a}
&  \mathbb{E}_{\mathbb{Q}_{s}^{x}}
\left[ \phi_{s,k,\ee_1}^y(G_n x)
({\Psi}^-_{s, \eta, \ee}\!\ast\!\rho_{\ee^2})
(T_{n}^v - \eta k)\right]  \nonumber\\
& =   \frac{1}{2\pi}  \int_{\mathbb{R}} e^{-it \eta k} 
R^{n}_{s,it} \big( \phi_{s,k,\ee_1}^y \big)(x)
\widehat {\Psi}^-_{s,\eta, \ee}(t) \widehat\rho_{\ee^{2}}(t) dt,
\end{align}
where
\begin{align*}
R^{n}_{s,it} \big( \phi_{s,k,\ee_1}^y \big)(x)
= \mathbb{E}_{\mathbb{Q}_{s}^{x}}\left[e^{it T_{n}^v} 
\phi_{s,k,\ee_1}^y(G_n x)
   \right],  \quad  x \in \bb P^{d-1}. 
\end{align*}
Substituting  \eqref{ScalarFourier Low a} into \eqref{ScalarBn Low a}, we obtain
\begin{align} \label{Pf_Low_Bn-}
A_3  \geq  \frac{a_{n,k}}{2 \pi} 
 \int_{\mathbb{R}} e^{-it \eta k} 
R^{n}_{s,it} \big( \phi_{s,k,\ee_1}^y \big)(x) 
\widehat {\Psi}^-_{s,\delta,\ee}(t) \widehat\rho_{\ee^{2}}(t) dt.
\end{align}
We shall use Proposition \ref{Prop Rn limit1} to give a precise asymptotic
for the above integral. 
Let us first verify the conditions of Proposition \ref{Prop Rn limit1}. 
Since the function $\tilde \chi_k^-$ is H\"{o}lder continuous for any fixed $k \geq 1$,
one can check that $\phi_{s,k,\ee_1}^y$ is H\"{o}lder continuous
on the projective space $\bb P^{d-1}$. 
Since the function 
$u \mapsto e^{-s'u}\psi(u)$  is directly Riemann integrable on $\mathbb{R}$ for 
any $s' \in K_{\epsilon} : = \{ s' \in \bb R: |s' - s| < \epsilon, s \in K \}$ with $\epsilon >0$ small enough, 
one can also verify that the function $\widehat {\Psi}^-_{s, \eta, \ee} \widehat\rho_{\ee^{2}}$ 
has compact support in $\mathbb{R}$, 
and that $\widehat {\Psi}^-_{s, \eta, \ee} \widehat\rho_{\ee^{2}}$
is differentiable in a small neighborhood of $0$ on the real line, for all $s \in K_{\mu}$. 
Thus, using Proposition \ref{Prop Rn limit1}
with $\varphi = \phi_{s,k,\ee_1}^y$,
$\psi = \widehat {\Psi}^-_{s,\eta,\ee} \widehat\rho_{\ee^{2}}$
and  $l = l_{n,k}': = \frac{\eta k }{n}$, 
we obtain that for sufficiently large $n$, there exists a constant $C>0$ such that for all 
$1 \leq k \leq M_n = \floor{C_1 \log n}$,  $s \in K_{\mu}$, 
$f \in (\bb R^d)^*$ and $v \in \bb R^d$ with $|f| = |v| = 1$,
\begin{align}  \label{ScalarKeyPropLower1}
I_k^-: & =  \Big|  \sigma_s \sqrt{n}     e^{n h_s(l_{n,k}')}   
    \int_{\mathbb R} e^{-it n l_{n,k}' } 
 R^{n}_{s,it} \big( \phi_{s,k,\ee_1}^y \big)(x)
 \widehat {\Psi}^-_{s,\eta,\ee}(t) \widehat\rho_{\ee^{2}}(t) dt
    -  B^{-}(k) \Big|  \nonumber\\
& \leq   \frac{C}{\sqrt{n}}  \| \phi_{s,k,\ee_1}^y \|_{\gamma},  
\end{align}
where 
\begin{align*}
B^{-}(k) : = \sqrt{2\pi} \widehat {\Psi}^-_{s, \eta, \ee}(0) \widehat\rho_{\ee^{2}}(0)
  \pi_{s} \big( \phi_{s,k,\ee_1}^y \big). 
\end{align*}
Since $1 \leq k \leq M_n= \floor{C_1 \log n}$, 
using Lemma \ref{lemmaCR001} there exists a constant $C>0$ such that for all $1 \leq k \leq M_n$,
$s \in K_{\mu}$ and $n \geq 1$, it holds that 
$|e^{ - n h_s(l_{n,k}') } - 1| \leq \frac{C \log n}{\sqrt{n}}$. 
In a similar way as in the proof of \eqref{ScalarKeyProp 2}, 
we can replace $e^{n h_s( l_{n,k}' )}$ by $1$ to obtain that, 
uniformly in $s \in K_{\mu}$, 
$f \in (\bb R^d)^*$ and $v \in \bb R^d$ with $|f| = |v| = 1$, 
\begin{align*} 
&  \Big| \sigma_s \sqrt{n} 
\int_{\mathbb R} e^{-it n l_{n,k}'} 
R^{n}_{s,it} \big( \phi_{s,k,\ee_1}^y \big)(x)
\widehat {\Psi}^{-}_{s, \eta, \ee}(t) \widehat\rho_{\ee^{2}}(t) dt - B^{-}(k)  \Big|  \nonumber\\
& \leq  I_k^-  e^{ - n h_s(l_{n,k}')}  + |e^{ - n h_s(l_{n,k}')} - 1| B^{-}(k)   \nonumber\\
&  \leq  \frac{C}{\sqrt{n}}  \| \phi_{s,k,\ee_1}^y \|_{\gamma}     
  +  \frac{C \log n}{\sqrt{n}}  \| \phi_{s,k,\ee_1}^y  \|_{\gamma}    \nonumber\\
& \leq  \frac{C \log n}{\sqrt{n}}  \| \phi_{s,k,\ee_1}^y  \|_{\gamma}. 
\end{align*}
Since the $\gamma$-H\"{o}lder norm 
$\| \phi_{s,k,\ee_1}^y \|_{\gamma}$
is bounded by 
$\frac{ e^{ \eta \gamma k} }{ ( e^{2\ee_1} - 1 )^{\gamma} }$, 
taking sufficiently small $\gamma>0$, we obtain that
 the series $ \frac{\log n}{\sqrt{n}} \sum_{k = 1}^{\infty} e^{-s\eta (k-1)}   \| \phi_{s,k,\ee_1}^y  \|_{\gamma}$
converges to $0$ as $n \to \infty$. 
As a result, by virtue of \eqref{ScalarKeyPropLower1}, we obtain that,  
uniformly in $s \in K_{\mu}$, 
$f \in (\bb R^d)^*$ and $v \in \bb R^d$ with $|f| = |v| = 1$, 
\begin{align*}
\sum_{k =1}^{\infty} \liminf_{ n \to \infty}  A_3 
\geq    \widehat {\Psi}^-_{s,\eta,\ee}(0) \widehat\rho_{\ee^{2}}(0)
\sum_{k =1}^{\infty} e^{-s\eta k}  \pi_{s} \big( \phi_{s,k,\ee_1}^y \big). 
\end{align*}
Note that $\widehat\rho_{\ee^{2}}(0) =1$. 
Using \eqref{SmoothIne Holder 02}, we have that $\tilde \chi_k \geq \chi_{k, 2\ee_1}^-$. 
Consequently, 
we obtain the lower bound for the first term on the right hand side of \eqref{Scal Low An 1}: 
uniformly in $s \in K_{\mu}$, 
$f \in (\bb R^d)^*$ and $v \in \bb R^d$ with $|f| = |v| = 1$, 
\begin{align} \label{LimsuBn Low a}
 \sum_{k =1}^{\infty} \liminf_{ n \to \infty}  A_3
\geq   \widehat {\Psi}^-_{s,\eta,\ee}(0)
\sum_{k =1}^{ \infty } e^{-s\eta k} 
\pi_{s} \big( \tilde \phi_{s,k,\ee_1}^y \big).  
\end{align}
where
\begin{align} \label{Pf_LDLimsuA_3_nn}
\tilde \phi_{s,k,\ee_1}^y (x)
= (\varphi r_s^{-1})(x) \mathbbm{1}_{\{ \log \delta(y, \cdot) \in I_k \}}(x) 
  - (\varphi r_s^{-1})(x)  \mathbbm{1}_{\{ \log \delta(y, \cdot) \in \tilde I_{k,\ee_1} \}}(x),
\end{align}
and  $\tilde I_{k,\ee_1} = \big(-\eta k, -\eta k + 2 \ee_1 \big]
\cup \big( -\eta (k-1) + 2 \ee_1, -\eta (k-1) \big]$. 
For the first term on the right hand-side of \eqref{Pf_LDLimsuA_3_nn},
since $I_k = (-\eta k, -\eta(k-1)]$, in a similar way as in the proof of \eqref{ScalPosiMainpart}, 
it holds uniformly in $s \in K_{\mu}$ that 
\begin{align} \label{ScalPosiMainpart_bb}
& \lim_{\eta \to 0}
\sum_{k =1}^{\infty} e^{-s\eta k} 
\pi_{s} \Big( (\varphi r_{s}^{-1}) 
\mathbbm{1}_{\{ \log \delta(y, \cdot) \in I_k \}}   \Big)   \nonumber\\
& = \int_{ \bb P^{d-1} } \delta(y, x)^s \varphi(x) r_s^{-1}(x) \pi_s(dx).
\end{align}
 To handle the second term on the right-hand side of \eqref{Pf_LDLimsuA_3_nn}, 
we make use of Lemma \ref{Lem_Gui_LeP} and the zero-one law of the stationary measure $\pi_s$ 
shown in Lemma \ref{Lem_0-1_law_s_Posi_Uni}. 
Specifically, similarly to the proof of \eqref{Pf_Upp_A1_ff}, 
since the function $\varphi r_{s}^{-1}$ is bounded on $\bb P^{d-1}$, uniformly in $s \in K_{\mu}$, 
using the Lebesgue dominated convergence theorem we get
that there exists a constant $C_1 >0$ such that for all $y \in (\bb P^{d-1})^*$ and $s \in K_{\mu}$, 
\begin{align}\label{Pf_Upp_A3_uu}
& \lim_{\ee_1 \to 0}
\sum_{k =1}^{\infty} e^{-s\eta k} 
\pi_{s} \Big( (\varphi r_{s}^{-1}) 
\mathbbm{1}_{\{ \log \delta(y, \cdot) \in \tilde I_{k,\ee_1} \}}   \Big)   \nonumber\\
& \leq C_1  \lim_{\ee_1 \to 0}
\sum_{k =1}^{\infty} e^{-s\eta k} 
\pi_{s} \Big( x \in \bb P^{d-1}: \log \delta(y, x) \in \tilde I_{k,\ee_1}  \Big)  \nonumber\\
& = C_1  \sum_{k =1}^{\infty} e^{-s\eta k} 
\pi_{s} \Big( x \in \bb P^{d-1}: \log \delta(y, x) = - \eta k  \Big)   \nonumber\\
& \quad  + C_1  \sum_{k =1}^{\infty} e^{-s\eta k} 
\pi_{s} \Big( x \in \bb P^{d-1}: \log \delta(y, x) = - \eta (k-1)  \Big)   \nonumber\\
& = 2 C_1  \sum_{k =1}^{\infty}  e^{-s\eta k}  
\pi_{s} \Big( x \in \bb P^{d-1}: \log \delta(y, x) = - \eta k  \Big), 
\end{align}
where in the last equality we used Lemma \ref{Lem_Gui_LeP}. 
In the same way as in the proof of \eqref{sum_0_1_law}, 
applying Lemma \ref{Lem_0-1_law_s_Posi_Uni} we can obtain that 
there exists a constant $0< \eta < 1$ such that for all $s \in K_{\mu}$, 
\begin{align}\label{sum_0_1_law_A_3}
\sum_{k =1}^{\infty} e^{-s\eta k} 
\pi_{s} \Big( x \in \bb P^{d-1}: \log \delta(y, x) = - \eta k  \Big) = 0. 
\end{align}
Since the target function $\psi$ satisfies the  condition \eqref{condition g}, 
 from \eqref{Def_Psi_aaa} we get that uniformly in $s \in K_{\mu}$, 
\begin{align}\label{Pf_LD_psi_dd}
\lim_{\ee \to 0}  \lim_{\eta \to 0}  \widehat {\Psi}^-_{s, \eta, \ee}(0) 
=  \int_{\mathbb{R}}  e^{-su} \psi(u) du. 
\end{align}
Consequently, in view of \eqref{LimsuBn Low a}, 
combining \eqref{ScalPosiMainpart_bb}, \eqref{Pf_Upp_A3_uu}, \eqref{sum_0_1_law_A_3} and \eqref{Pf_LD_psi_dd},
we get the desired lower bound for $A_3$: uniformly in $s \in K_{\mu}$, 
\begin{align}\label{Scal Low Bn}
\lim_{\ee \to 0}  \lim_{\eta \to 0} 
\lim_{\ee_1 \to 0}
\sum_{k =1}^{\infty} \liminf_{ n \to \infty}  A_3 
\geq  \int_{\mathbb{R}}  e^{-su} \psi(u) du 
\int_{ \bb P^{d-1} } \delta(y, x)^s \varphi(x) r_s^{-1}(x) \pi_s(dx).
\end{align}

Now we proceed to establish an upper bound for the term $A_4$ in \eqref{Scal Low An 1}. 
Note that $\Psi^-_{s,\eta,\ee} \leq \Psi_{s,\eta}$, where 
$\Psi_{s,\eta} (u) = e^{-sy} \psi_{\eta}^+(u)$, $u \in \mathbb{R}$. 
Then it follows from Lemma \ref{estimate u convo} that
$\Psi^-_{s,\eta,\ee} 
\leq (1+ C_{\rho}(\ee)) \widehat{\Psi}^+_{s,\eta,\ee} \widehat{\rho}_{\ee^2}$,
where ${\Psi}^+_{s, \eta, \ee}(u) = \sup_{u'\in\mathbb{B}_{\ee}(u)} \Psi_{s,\eta} (u')$, 
$u \in \mathbb{R}$. 
Moreover, using \eqref{Pf_LD_SmoothIneHolder01}, we get 
$\mathbbm{1}_{ \{ Y_n^{x,y} \in I_k \} } \leq \tilde\chi_k(u)  
= (\chi_{k, \ee_1}^+ * \bar{\rho}_{\ee_1})(u)$.
Consequently, similarly to the proof of  \eqref{Pf_LDScalA2_2}, 
we can get the upper bound for $A_4$: uniformly in $s \in K_{\mu}$, 
\begin{align} \label{Pf_Upper_Un-}
&  A_4 
\leq  (1+ C_{\rho}(\ee))  \frac{a_{n,k}}{2 \pi}    \nonumber\\
    &  \quad \times   \int_{ |u|\geq \ee } \left[  
       \int_{\mathbb{R}} e^{-it (\eta k + u )} 
      R^{n}_{s,it} \big( \varphi_{s,k,\ee_1}^y  \big)(x)
      \widehat {\Psi}^+_{s,\eta,\ee}(t) \widehat\rho_{\ee^{2}}(t) dt \right]
         \rho_{\ee^2}(u)   du.
\end{align}
In order to handle the above integral, we first use the Lebesgue dominated convergence theorem
to interchange the limit $n \to \infty$ and the integral $\int_{|u| \geq \ee}$, 
and then we apply Proposition \ref{Prop Rn limit1}. 
An important issue is to find a dominating function, which can be done as follows. 
We split the integral $\int_{|u| \geq \ee}$ on the right hand side of \eqref{Pf_Upper_Un-} 
into two parts: $\int_{\ee \leq |u| \leq \sqrt{n}}$ and $\int_{ |u| >  \sqrt{n}}$.
For the first part,  by elementary calculations it holds that 
$e^{- nh_s( \eta \frac{k}{n} + \frac{u}{n})} \to 1$, 
uniformly in $s \in K_{\mu}$, $1 \leq k \leq M_n$ and $|u| \leq \sqrt{n}$ as $n \to \infty$.
Hence, using Proposition \ref{Prop Rn limit1}, the function on the right hand side of 
 \eqref{Pf_Upper_Un-} under the integral $\int_{\ee \leq |u| \leq \sqrt{n}}$ is dominated by 
 $C \rho_{\ee^2}$, which is integrable on $\mathbb{R}$. 
For the second part $\int_{ |u| >  \sqrt{n}}$, 
since the density function $\rho$ has 
polynomial decay, i.e. $\rho_{\ee^2}(u) \leq \frac{C}{1 + u^4}$, $|u| >  \sqrt{n}$,
we get that $\sqrt{n} \rho_{\ee^2}(u) \leq \frac{C}{1 + |u|^3}$, 
which is clearly integrable on $\mathbb{R}$.  
Therefore, we can pass the limit as $n\to \infty$ under the integration $\int_{|u|\geq \ee}$
and then we use Proposition \ref{Prop Rn limit1} to obtain the desired upper bound for $A_4$:
uniformly in $s \in K_{\mu}$, 
\begin{align*}
 \sum_{k =1}^{\infty} \limsup_{ n \to \infty}  A_4
&  \leq   (1+ C_{\rho}(\ee))  \sum_{k=1}^{ \infty }  e^{-s \eta k}
\pi_{s} \big( \varphi_{s,k,\ee_1}^y  \big)  \nonumber\\
&  \quad  \times    \widehat \Psi^+_{s,\eta,\ee}(0) \widehat\rho_{\ee^{2}}(0)
\int_{|u| \geq \ee} \rho_{\ee^2}(u)du. 
\end{align*}
In the same way as in the proof of  \eqref{Pf_LD_pi_s_hh}, 
by Lemmas \ref{Lem_Gui_LeP} and \ref{Lem_0-1_law_s_Posi_Uni}, we can get that uniformly in $s \in K_{\mu}$, 
\begin{align*}
\lim_{\eta \to 0} 
\lim_{\ee_1 \to 0} \sum_{k =1}^{\infty} e^{-s\eta k} 
\pi_{s} \big( \tilde \varphi_{s,k,\ee_1}^y \big)  
= \int_{ \bb P^{d-1} } \delta(y, x)^s \varphi(x) r_s^{-1}(x) \pi_s(dx). 
\end{align*}
Using \eqref{Pf_LD_psi_limit} and noting that $\widehat\rho_{\ee^{2}}(0) = 1$, 
it follows that uniformly in $s \in K_{\mu}$, 
\begin{align*}
\lim_{\eta \to 0} \lim_{\ee_1 \to 0} 
\sum_{k =1}^{\infty} \limsup_{ n \to \infty}  A_4
& \leq (1+ C_{\rho}(\ee))  
  \int_{ \bb P^{d-1} } \delta(y, x)^s \varphi(x) r_s^{-1}(x) \pi_s(dx)   \nonumber\\
& \quad \times \int_{\mathbb{R}}  e^{-su} \psi(u) du
   \int_{|u| \geq \ee} \rho_{\ee^2}(u)du.  
\end{align*}
Since $C_{\rho}(\ee) \to 0$ and $\int_{|u| \geq \ee} \rho_{\ee^2}(u)du  \to  0$ as $\ee \to 0$,
this implies
\begin{align*}
\lim_{\ee \to 0} 
\lim_{\eta \to 0} \lim_{\ee_1 \to 0} 
\sum_{k =1}^{\infty} \limsup_{ n \to \infty}  A_4 = 0. 
\end{align*}
Combining this with \eqref{Scal Low An 1} and \eqref{Scal Low Bn}, 
we get the desired lower bound for $A_2$: 
uniformly in $s \in K_{\mu}$, 
$f \in (\bb R^d)^*$ and $v \in \bb R^d$ with $|f| = |v| = 1$, 
\begin{align*}
\lim_{\ee \to 0}  \lim_{\eta \to 0} 
\lim_{\ee_1 \to 0}  \liminf_{n \to \infty}  A_2 
\geq  \int_{\mathbb{R}}  e^{-su} \psi(u) du 
\int_{ \bb P^{d-1} } \delta(y, x)^s \varphi(x) r_s^{-1}(x) \pi_s(dx). 
\end{align*}
This, together with \eqref{PosiScalA_ccc}, \eqref{ScalInverAn 2}  and \eqref{ScaProLimBn Upper 01},
proves the desired asymptotic 
\eqref{SCALREZ02_Uni_s}.  
This concludes the proof of Theorem \ref{Thm_BRP_Uni_s} as well as Theorem \ref{Thm_BRP_Upper}.   
\end{proof}

\subsection{Proof of Theorem \ref{Thm_Coeff_BRLD_changedMea}}

\begin{proof}[Proof of Theorem \ref{Thm_Coeff_BRLD_changedMea}]
It suffices to prove \eqref{LD_Upper_ChangeMea002} 
since \eqref{LD_Upper_ChangeMea001} follows from \eqref{LD_Upper_ChangeMea002}
by taking $\varphi = \bf 1$ and $\psi(u) = \bbm{1}_{\{u \geq 0\}}$. 
Using the change of measure formula \eqref{basic equ1} twice, we get
\begin{align}\label{Pf_BRP_Changed_ff}
&  \mathbb{E}_{\bb Q_s^x} 
  \Big[ \varphi(G_n x) \psi \big( \log |\langle f, G_n v \rangle| - nq_t \big) \Big] \nonumber\\
& =    \frac{1}{\kappa^n(s) r_s(x)}
\mathbb{E}
\left[ (\varphi r_s)(G_n x) e^{s \sigma (G_n, x)} 
  \psi \big( \log |\langle f, G_n v \rangle| - nq_t \big)  \right]  \nonumber\\
&  =    \frac{ \kappa^n(t) r_t(x) }{\kappa^n(s) r_s(x)}  
 \mathbb{E}_{\mathbb Q_t^x} 
 \left[  (\varphi r_s r_t^{-1})(G_n x)
     e^{-(t-s) \sigma(G_n, x)} 
     \psi \big( \log |\langle f, G_n v \rangle| - nq_t \big)  \right]   \nonumber\\
& =   \frac{ \kappa^n(t) r_t(x) }{\kappa^n(s) r_s(x)}  e^{ -(t-s) n q_t}  \times  \nonumber\\
& \quad     \mathbb{E}_{\mathbb Q_t^x} 
 \left[  (\varphi r_s r_t^{-1})(G_n x)
     e^{-(t-s) (\sigma(G_n, x) - n q_t)} 
     \psi \big( \log |\langle f, G_n v \rangle| - nq_t \big)  \right], 
\end{align}
where $\mathbb Q_t^x$ is the changed measure defined 
in the same way as $\mathbb Q_s^x$ with $s$ replaced by $t$. 
Following the proof of Theorem \ref{Thm_BRP_Upper}, one can verify that,  
 as $n \to \infty$, 
uniformly in $t \in K_{\mu}$, $f \in (\bb R^d)^*$ and $v \in \bb R^d$ with $|f| = |v| = 1$,  
\begin{align*}
&  \mathbb{E}_{\mathbb Q_t^x} 
 \left[  (\varphi r_s r_t^{-1})(G_n x)
     e^{-(t-s) (\sigma(G_n, x) - n q_t)} 
     \psi \big( \log |\langle f, G_n v \rangle| - nq_t \big)  \right]    \nonumber\\
&   = \frac{1}{\sigma_t \sqrt{2 \pi n}} 
  \left[  \int_{ \bb P^{d-1} } \varphi(x) \delta(y, x)^t \,  \frac{r_s(x)}{r_t(x)} \pi_t(dx)
       \int_{\mathbb{R}} e^{- (t-s) u} \psi(u) du + o(1) \right]. 
\end{align*}
The result follows by taking into account that 
$\Lambda^*(q_s) = sq_s - \Lambda(s)$, $\Lambda^*(q_t) = tq_t - \Lambda(t)$, 
$\Lambda(s) = \log \kappa(s)$ and $\Lambda(t) = \log \kappa(t)$.
\end{proof}

\section{Proof of lower tail large deviations for coefficients} \label{sec_Pf_Low_LD}

The goal of this section is to establish 
Theorems \ref{Thm-Posi-Neg-s} and \ref{Thm-Posi-Neg-sBRP} on Bahadur-Rao-Petrov type lower tail large deviations. 
In contrast to the proof of Theorems \ref{thrmBR001} and \ref{Thm_BRP_Upper},
it turns out that the proof of Theorems \ref{Thm-Posi-Neg-s} and \ref{Thm-Posi-Neg-sBRP} is more delicate. 

It suffices to prove Theorem \ref{Thm-Posi-Neg-sBRP}
since Theorem \ref{Thm-Posi-Neg-s} is a 
direct consequence of Theorem \ref{Thm-Posi-Neg-sBRP} 
by taking $l = 0$, $\varphi = \bf 1$ and $\psi(u) = \bbm{1}_{ \{u \leq 0\} }$, $u \in \bb R$. 

\subsection{Proof of Theorem \ref{Thm-Posi-Neg-sBRP}}

We shall need the H\"{o}lder regularity of the stationary measure $\pi_s$
(for sufficiently small $s$)
recently established in \cite{GQX20}.

\begin{lemma}\label{Lem_Regu_pi_s}
Assume conditions \ref{Condi-TwoExp} and \ref{Condi-IP}. 
Then, for any $\ee >0$, there exist constants $s_0 >0$, $k_0 \in \bb N$ and $c, C >0$ such that
for all $s \in (-s_0, s_0)$, $n \geq k \geq k_0$, $y \in (\bb P^{d-1})^*$ and $x \in \bb P^{d-1}$,
\begin{align}\label{Regu_pi_s}
\bb Q_s^x \Big( \log \delta(y, G_n x) \leq -\ee k  \Big) \leq C e^{- ck}. 
\end{align}
\end{lemma} 

Note that \eqref{Regu_pi_s} is stronger than the following assertion 
of the H\"{o}lder regularity of the stationary measure $\pi_s$: 
there exist constants $s_0, \alpha > 0$ such that
\begin{align*} 
\sup_{ s\in (-s_0, s_0) } \sup_{y \in (\bb{P}^{d-1})^* } 
  \int_{\bb{P}^{d-1} } \frac{1}{ \delta(y, x)^{\alpha} } \pi_s(dx) < +\infty. 
\end{align*}

As an application of Lemma \ref{Lem_Regu_pi_s}, 
we show the following result about the high-order negative moment of the $\delta(y, G_n x)$
under the changed measure $\bb Q_s^x$,
which will play important role in the proof of Theorem \ref{Thm-Posi-Neg-sBRP}.  

\begin{lemma}\label{Lem_Inte_Regu_a}
Assume conditions \ref{Condi-TwoExp} and \ref{Condi-IP}. 
Let $p>0$ be any fixed constant.
Then, there exists a constant $s_0 > 0$ such that
\begin{align*}
\sup_{n \geq 1} \sup_{ s\in (-s_0, s_0) } \sup_{y \in (\bb{P}^{d-1})^* }  
\sup_{x \in \bb P^{d-1} } 
\bb E_{\bb Q_s^x} \left( \frac{1}{ \delta(y, G_n x)^{p|s|} } \right) < + \infty. 
\end{align*}
\end{lemma}

\begin{proof}
By Lemma \ref{Lem_Regu_pi_s}, 
for any $\ee >0$, there exist constants $s_0 >0$, $k_0 \in \bb N$ and $c, C >0$ such that
for all $s \in (-s_0, s_0)$, $n \geq k \geq k_0$ and $y \in (\bb P^{d-1})^*$, $x \in \bb P^{d-1}$,
\begin{align}\label{Regu_Q_sx_bb}
\bb Q_s^x \Big( \delta(y, G_n x) \leq e^{-\ee k}  \Big) \leq C e^{- ck}. 
\end{align}
For any $y \in (\bb P^{d-1})^*$ and $k \geq k_0$, we denote
\begin{align*}
B_{n,k} = \left\{ x \in \bb P^{d-1}: e^{- \ee (k+1)} \leq \delta(y, G_n x) \leq e^{- \ee k} \right\}. 
\end{align*}
By \eqref{Regu_Q_sx_bb}, it follows that there exist constants $c, C >0$ such that for all $s \in (-s_0, s_0)$, 
\begin{align*}
& \bb E_{\bb Q_s^x} \left( \frac{1}{ \delta(y, G_n x)^{p|s|} }  \right)  \\
& =  \bb E_{\bb Q_s^x} \left( \frac{1}{ \delta(y, G_n x)^{p|s|} } 
    \bbm{1}_{\{ \delta(y, G_n x) > e^{- \ee k_0} \} } \right)
    + \sum_{k = k_0}^{\infty} \bb E_{\bb Q_s^x} \left( \frac{1}{ \delta(y, G_n x)^{p|s|} }
    \bbm{1}_{B_{n,k}} \right)    \\
& \leq  e^{\ee k_0 p|s|} + C \sum_{k = k_0}^{\infty} e^{\ee (k+1) p|s|} e^{-ck},
\end{align*}
which is finite by taking $s_0>0$ small enough. This proves Lemma \ref{Lem_Inte_Regu_a}. 
\end{proof}

Now we are in a position to establish Theorem \ref{Thm-Posi-Neg-sBRP}. 
In the same spirit as in Theorem \ref{Thm_BRP_Uni_s}, 
we are able to prove the following result which is stronger than Theorem \ref{Thm-Posi-Neg-sBRP}.

\begin{theorem} \label{Thm_BRP_Neg_Uni}
Assume conditions \ref{Condi-TwoExp} and \ref{Condi-IP}. 
Then, there exists a constant $s_0>0$ such that for any compact set $K_{\mu} \subset (-s_0, 0)$,
we have, as $n \to \infty$, 
uniformly in $s \in K_{\mu}$, $f \in (\bb R^d)^*$ and $v \in \bb R^d$ with $|f| = |v| = 1$,
\begin{align*} 
 \mathbb{P} \Big( \log |\langle f, G_n v \rangle| \leq n q \Big)  
 =  \frac{ r_{s}(x)  r^*_{s}(y)}{\varrho_s}  
   \frac{ \exp \left( -n   \Lambda^*(q) \right) } { -s \sigma_{s}\sqrt{2\pi n}} 
  \big[ 1 + o(1) \big]. 
\end{align*}
More generally, for any $\varphi \in \mathcal{B}_{\gamma}$ and any measurable function $\psi$ on $\mathbb{R}$ 
such that $u \mapsto e^{-s'u} \psi(u)$ is directly Riemann integrable for all  
$s' \in K_{\mu}^{\epsilon} : = \{ s' \in \bb R: |s' - s| < \epsilon, s \in K_{\mu} \}$ with $\epsilon >0$ small enough, 
we have,
as $n \to \infty$, uniformly in $s \in K_{\mu}$, 
  $f \in (\bb R^d)^*$ and $v \in \bb R^d$ with $|f| = |v| = 1$,
\begin{align*} 
&  \mathbb{E} \Big[ \varphi(G_n x) \psi \big( \log |\langle f, G_n v \rangle| - n q \big) \Big]   \nonumber\\
&  = \frac{r_{s}(x)}{\varrho_s}
    \frac{ \exp (-n \Lambda^*(q)) }{ \sigma_{s}\sqrt{2\pi n}}      
    \left[ \int_{ \bb P^{d-1} } \varphi(x) \delta(y,x)^s \nu_s(dx) 
        \int_{\mathbb{R}} e^{-su} \psi(u) du + o(1) \right]. 
\end{align*}
\end{theorem}

\begin{proof}[Proof of Theorems \ref{Thm-Posi-Neg-sBRP} and \ref{Thm_BRP_Neg_Uni}]
We only need to prove Theorem \ref{Thm_BRP_Neg_Uni} 
since Theorem \ref{Thm-Posi-Neg-sBRP} is a direct consequence of Theorem \ref{Thm_BRP_Neg_Uni}.

It suffices to prove the second assertion of Theorem \ref{Thm_BRP_Neg_Uni}, since
the first one 
follows from the second by choosing $\varphi = \mathbf{1}$ 
and $\psi (u) = \mathbbm{1}_{ \{ u \leq 0 \} },$ $u \in \mathbb R.$ 
As in the proof of Theorem \ref{Thm_BRP_Upper}, 
we assume that the target functions $\varphi$ and $\psi$
are non-negative, and that 
the function $\psi$ satisfies the condition \eqref{condition g}.

Since $f \in (\bb R^d)^*$ and $v \in \bb R^d$ with $|f| = |v| = 1$,
and $y = \bb R f$ and $x \in \bb R v$,  
we have $\log |\langle f, G_n v \rangle| = \log |G_n v| + \log \delta(y, G_n x).$ 
Hence we can replace the logarithm of the coefficient $\log |\langle f, G_n v \rangle|$ 
by the sum $\log |G_n v| + \log \delta(y, G_n x)$ as follows: 
\begin{align*}
J : & = \sigma_s \sqrt{2\pi n}  \frac{ e^{n \Lambda^*(q)} }{r_s(x)} 
\mathbb{E} \Big[ \varphi(G_n x)\psi( \log |\langle f, G_n v \rangle| - n q ) \Big] \nonumber\\
&  =   \sigma_s \sqrt{2\pi n}   
\frac{ e^{n \Lambda^*(q)} }{r_s(x)} \mathbb{E} \Big[ \varphi(G_n x) 
  \psi( \log |G_n x| + \log \delta(y, G_n x) - n q ) \Big]. 
\end{align*}
As in the proof of Theorem \ref{Thm_BRP_Upper}, 
we denote for any $y = \bb R f \in (\bb P^{d-1})^*$ and $x \in \bb R v \in \bb P^{d-1}$, 
\begin{align*}
T_n^v: = \log |G_n v| - nq,  \qquad  Y_n^{x,y}: = \log \delta(y, G_n x).  
\end{align*}
Taking into account that $q = \Lambda'(s)$ and $e^{n\Lambda^{*}(q)}=e^{nsq} \kappa^{-n}(s)$, 
and using the change of measure formula \eqref{basic equ1}, 
we get
\begin{align} \label{ScaProLimAnLow01}
J =  \sigma_s \sqrt{2\pi n}  
  \mathbb{E}_{\mathbb{Q}_{s}^{x}} \Big[ (\varphi r_{s}^{-1})(G_n x)  e^{-s T_n^v}
\psi \big( T_{n}^v + Y_n^{x,y} \big)  \Big].
\end{align}
For any fixed small constant $0< \eta <1$, denote $I_k: = (-\eta k, -\eta(k-1)]$,  $k \geq 1$. 
Let $M_n:= \floor{ C_1 \log n }$, where $C_1>0$ is a sufficiently large constant
and $\floor{a}$ denotes the integer part of $a \in \bb R$. 
Then from \eqref{ScaProLimAnLow01} we have the following decomposition: 
\begin{align}\label{PosiScalALow_decom}
J = J_1 + J_2, 
\end{align}
where  
\begin{align*}
& J_1 : =  \sigma_s \sqrt{2\pi n}  \mathbb{E}_{\mathbb{Q}_{s}^{x}} 
\left[ (\varphi r_{s}^{-1})(G_n x)  e^{-s T_n^v}
\psi \big( T_{n}^v + Y_n^{x,y} \big) \mathbbm{1}_{\{Y_n^{x, y} \leq -\eta M_n \}} \right],  
    \nonumber\\
& J_2 : =  \sigma_s \sqrt{2\pi n}  \sum_{k =1}^{M_n}  
\mathbb{E}_{\mathbb{Q}_{s}^{x}} 
\left[ (\varphi r_{s}^{-1})(G_n x)  e^{-s T_n^v}
\psi \big( T_{n}^v + Y_n^{x,y}  \big) \mathbbm{1}_{\{Y_n^{x,y} \in I_k \}} \right]. 
\end{align*}

\textit{Upper bound of $J_1$.}
Since the function $u \mapsto e^{-s' u} \psi(u)$ is directly Riemann integrable on $\mathbb{R}$ for some $s' \in (0, s)$,
one can verify that the function $u \mapsto e^{-s u} \psi(u)$ is bounded on $\mathbb{R}$
and hence there exists a constant $C >0$ such that for all $s \in (-s_0, 0]$, 
\begin{align*}
e^{-s T_n^v} \psi( T_{n}^v + Y_n^{x,y} )  \leq  C e^{ s Y_n^{x,y} }. 
\end{align*}
Hence, by the H\"{o}lder inequality, Lemma \ref{Lem_Regu_pi_s} and Lemma \ref{Lem_Inte_Regu_a},
we obtain that as $n \to \infty$, uniformly in $s \in (-s_0, 0]$,
\begin{align*}
J_1 & \leq  C \sqrt{n}  \, \left\{ \bb E_{\bb Q_s^x} \left( \frac{1}{ \delta(y, G_n x)^{-2s} } \right)
 \bb Q_s^x \Big( \log \delta(y, G_n x) \leq -\eta \floor{C_1 \log n} \Big)  \right\}^{1/2}  \nonumber\\
& \leq  C \sqrt{n}  \,  e^{-c_{\eta} \floor{C_1 \log n} }  \to 0.  
\end{align*}


\textit{Upper bound of $J_2$.}
Following the proof of \eqref{Pf_LDScalA2_2}, one has
\begin{align*} 
& J_2  \leq  (1+ C_{\rho}(\ee))
 \sigma_s \sqrt{\frac{n}{2\pi}} \,  
 \sum_{k =1}^{\infty} \mathbbm{1}_{ \{ k \leq M_n \} }  e^{-s\eta (k-1)}   \nonumber\\
& \qquad \quad \times   \int_{\mathbb{R}} e^{-it \eta(k-1)} 
R^{n}_{s,it} \big( \varphi_{s,k,\ee_1}^y  \big)(x)
\widehat {\Psi}^+_{s, \eta, \ee}(t) \widehat\rho_{\ee^{2}}(t) dt, 
\end{align*}
where $\varphi_{s,k,\ee_1}^y$ and $\Psi^+_{s, \eta, \ee}$ 
are respectively defined by \eqref{Pf_Baradur_LD_varph_11} and \eqref{Def_Psi_aaa}. 
Since $|s|$ and $\gamma>0$ are sufficiently small, by elementary calculations we get that 
we obtain that
 the series $ \frac{\log n}{\sqrt{n}} \sum_{k = 1}^{M_n} e^{-s\eta (k-1)}   \| \varphi_{s,k,\ee_1}^y  \|_{\gamma}$
converges to $0$ as $n \to \infty$. 
Hence, we can apply Proposition \ref{Prop Rn limit1} (2) instead of Proposition \ref{Prop Rn limit1} (1),
and follow the proof of \eqref{ScalarKeyProp 1}, \eqref{ScalarKeyProp 2} and \eqref{ScalarBn abc} 
to obtain that uniformly in $s \in K_{\mu}$, 
$f \in (\bb R^d)^*$ and $v \in \bb R^d$ with $|f| = |v| = 1$, 
\begin{align} \label{ScalarBn_Low_J2}
\limsup_{n \to \infty}  J_2 
\leq   (1+ C_{\rho}(\ee)) \widehat {\Psi}^+_{s, \eta, \ee}(0) \widehat\rho_{\ee^{2}}(0)
\sum_{k =1}^{\infty} e^{-s\eta (k-1)} 
\pi_{s} \big( \varphi_{s,k,\ee_1}^y \big). 
\end{align}
Then, we can proceed in a similar way as in the proof of  
\eqref{LimsuBn a}, \eqref{ScalPosiMainpart}, \eqref{Pf_Upp_A1_ff} and \eqref{sum_0_1_law}. 
One of the main differences is that in \eqref{Pf_Upp_A1_ff} we need to 
use the H\"{o}lder regularity of the stationary measure $\pi_s$ stated in Lemma \ref{Lem_Regu_pi_s} 
to justify the applicability of the Lebesgue dominated convergence theorem, 
when we interchange the limit $\ee_1 \to 0$ and the sum over $k$ in \eqref{Pf_Upp_A1_ff}. 
Another difference is that in \eqref{sum_0_1_law} it is necessary to 
use the zero-one law for the stationary measure $\pi_s$ shown in 
Lemma \ref{Lem_0-1_law_s_Neg_Uni} instead of Lemma \ref{Lem_0-1_law_s_Posi_Uni}. 
Consequently, one can obtain the desired upper bound of $J_2$ 
which is similar to \eqref{ScaProLimBn Upper 01}:
uniformly in $s \in K_{\mu}$, 
$f \in (\bb R^d)^*$ and $v \in \bb R^d$ with $|f| = |v| = 1$, 
\begin{align*}
\lim_{\ee \to 0}  \lim_{\eta \to 0} 
\lim_{\ee_1 \to 0}  \limsup_{n \to \infty}  J_2 
\leq  \int_{\mathbb{R}}  e^{-su} \psi(u) du 
\int_{ \bb P^{d-1} } \delta(y, x)^s \varphi(x) r_s^{-1}(x) \pi_s(dx). 
\end{align*}

The lower bound of $J_2$ can be carried out in a similar way and hence we omit the details. 
\end{proof}

By Theorem \ref{Thm_BRP_Upper} and \ref{Thm-Posi-Neg-sBRP}, 
we now give a proof of Theorem \ref{Theorem local LD002}
on the local limit theorem with large deviations for coefficients $\langle f, G_n v \rangle$. 

\begin{proof}[Proof of Theorem \ref{Theorem local LD002}]
The asymptotic \eqref{LLTLDa} follows from Theorem \ref{Thm_BRP_Upper}
by taking $\varphi = \bf 1$ and $\psi(u) = \bbm{1}_{ \{u \in [a_1, a_2]\} } (u)$, $u \in \bb R$. 
In the same way, the asymptotic \eqref{LLTLDb} 
is a direct consequence of Theorem \ref{Thm-Posi-Neg-sBRP}. 
\end{proof}

\section{Proof of the H\"{o}lder regularity of the stationary measure} \label{Sec:regpositive}

In this section we prove Proposition \ref{PropRegularity} 
on the H\"{o}lder regularity of the stationary measure $\pi_s$ 
 for any $s \in I_{\mu}^{\circ}$.
This result is of independent interest and 
 plays a crucial role for establishing the precise large deviation asymptotics
for the coefficients $\langle f, G_n v \rangle$ under the changed measure $\bb Q_s^x$,  
see Theorem \ref{Thm_BRP_Upper}.  

The study of  the regularity of the stationary measure $\nu$ defined by \eqref{mu station meas},
attracted a great deal of attention, 
see e.g. \cite{Aou11, BQ13, BQ16, BQ17, BL85, BFLM11, CPV93, DKW19, Gui90, GR85}.
As far as we know, there are three different approaches to establish
the regularity of $\nu$. 
The first one is originally due to Guivarc'h \cite{Gui90}, 
see also \cite{BL85}. 
The approach in \cite{Gui90} consists in investigating the asymptotic behaviors of the components 
in the Cartan and Iwasawa decompositions of the random matrix product $M_n = g_1 \ldots g_n$. 
The second one is developed in \cite{BFLM11} for the study of the regularity of 
the stationary measure on the torus $\bb T^d = \bb R^d / \bb Z^d$,
and has been applied to the setting of products of random matrices in \cite{BQ16, BQ17},
where the large deviation bounds for the Iwasawa cocycle and for the Cartan projection play a crucial role. 
The third one, which is recently developed in \cite{DKW19} for the special linear group $SL(2, \mathbb{C})$
consisting of complex $2 \times 2$ matrices with determinant one,  
is based on the theory of super-potentials introduced in \cite{DS09}.   
All of the results mentioned above are concerned with the regularity of the stationary measure $\nu$.
However, the regularity of the eigenmeasure $\nu_s$ or of  the stationary measure $\pi_s$ for $s$ different from $0$
was not known before in the literature.

In order to prove Proposition \ref{PropRegularity}, 
we first extend some convergence results concerning 
the Cartan and Iwasawa decompositions of the matrix product $M_n$ 
established earlier in \cite{BL85} under the measure $\mathbb P$,
to the framework of the changed measure $\mathbb{Q}_s$. 

Similarly to \eqref{Def_Q_s_lll}, for any $s \in I_{\mu}$, 
we define the conjugate Markov operator $Q_s^*$
as follows: for any $\varphi \in \mathcal{C}((\bb P^{d-1})^*)$, 
\begin{align*}
Q_s^* \varphi (y) = \frac{1}{\kappa(s) r_s^*(y) } P_s^* (\varphi r_s^*) (y), \quad  y \in (\bb{P}^{d-1})^*. 
\end{align*}
Then $Q_s^*$ has a unique stationary measure $\pi_s^*$ given by
$\pi_s^*(\varphi) = \frac{\nu_s^* (\varphi r_s^*)}{\nu_s^* (r_s^*)}$
for any $\varphi \in \mathcal{C}((\bb P^{d-1})^*)$. 


\subsection{Asymptotics for the Cartan decomposition} \label{Sec_Cartan}
Recall that $G_n = g_n \ldots g_1$. 
We are going to investigate asymptotic behaviors of the components 
of the Cartan decomposition of the transposed matrix product 
\begin{align*}
G_n^* = g_1^* g_2^* \ldots g_n^*,  \quad  n \geq 1, 
\end{align*}
where $g^*$ is the adjoint automorphism of the matrix $g$. 
Let $K = SO(d,\mathbb{R})$ be the orthogonal group,
and $A^+$ be the set of diagonal matrices whose diagonal entries starting from the upper left corner 
are strictly positive and decreasing. 
With these notation, the well known Cartan decomposition states that $GL(d, \mathbb{R}) = K A^+ K$.
The Cartan decomposition of $G_n^*$ is written as 
$G_n^* = k_n a_n k_n'$, where $k_n, k_n' \in K$
and $a_n \in A^+$ with its diagonal elements (singular values) 
satisfying $a_n^{1,1} \geq a_n^{2,2} \geq \ldots \geq a_n^{d,d} > 0$. 
Note that the diagonal matrix $a_n$ is uniquely determined, 
but the orthogonal matrices $k_n$ and $k_n'$ are not unique. 
We choose one such decomposition of $G_n^*$. 
Denote by $e_1^*, \ldots, e_d^*$ the dual basis of $(\bb R^d)^*$. 
The vector $k_n e_1^* \in (\mathbb{P}^{d-1})^*$ is called the \emph{density point} of $G_n^*$.
It plays an important role in the study of products of random matrices: see \cite{BFLM11, BQ17}. 
The following result shows that the 
density point converges almost surely to the 
random variable $Z_s^*$ of the law $\pi_s^*$ under the changed measure 
$\mathbb{Q}_s: = \int_{\mathbb{P}^{d-1}} \mathbb{Q}_s^x \pi_s(dx)$. 
Note that by definition the measure $\mathbb{Q}_s$ is shift-invariant and ergodic 
since $\pi_s$ is the unique stationary measure of the Markov operator $Q_s$. 
Recall that $\delta(y,x) = \frac{|\langle f, v \rangle|}{|f||v|}$ 
for any $y = \bb R f \in (\bb P^{d-1})^*$ and $x = \bb R v \in \bb P^{d-1}$.  

\begin{lemma}\label{LemConverThmQs}
Let $s \in I_{\mu}^{\circ}$. 
Under condition \ref{Condi-IP}, with the above notation,  
we have
\begin{align}\label{ConverThmQs}
\lim_{n \to \infty} \frac{ a_n^{2,2} }{ a_n^{1,1} } = 0,   \   \mathbb{Q}_s\mbox{-}a.s.
  \quad \mbox{and}  \quad 
\lim_{n \to \infty}  k_n  e_1 = Z_s^*,  \   \mathbb{Q}_s\mbox{-}a.s.,   
\end{align}
and for any $x = \bb R v \in \mathbb{P}^{d-1}$ with $v \in \bb R^d \setminus \{0\}$, 
\begin{align} \label{ConverThmQs02}
 \lim_{ n \to \infty} \frac{ | G_n v | }{ \| G_n \| |v| } = \delta(Z_s^*, x),
 \   \mathbb{Q}_s\mbox{-}a.s.,  
\end{align}
where the law of the random variable $Z_s^*$ (on $(\mathbb{P}^{d-1})^*$) is the stationary measure $\pi_s^*$. 
Moreover, the assertions \eqref{ConverThmQs} and \eqref{ConverThmQs02} also hold true with the measure $\mathbb{Q}_s$
replaced by $\mathbb{Q}_s^x$, for any starting point $x \in \mathbb{P}^{d-1}$. 
\end{lemma}

Before proceeding to proving Lemma \ref{LemConverThmQs},  
let us first recall the following two results which were established in \cite{GL16}. 
In the sequel, 
let $m^*$ be the unique rotation invariant probability measure on the projective space $(\mathbb{P}^{d-1})^*$. 
For any matrix $g \in GL(d, \mathbb{R})$, denote by $g^* m^*$ the probability measure on $(\mathbb{P}^{d-1})^*$
such that for any measurable function $\varphi$ on $(\mathbb{P}^{d-1})^*$,
\begin{align*}
\int_{(\mathbb{P}^{d-1})^* } \varphi(y) (g^* m^*) (dy) = \int_{(\mathbb{P}^{d-1})^* } \varphi( g^* y ) m^*(dy). 
\end{align*}

\begin{lemma}\label{Lem_DiracMea}
Assume condition \ref{Condi-IP}. Let $s \in I_{\mu}^{\circ}$.  Then, 
the probability measure $G_n^* m^*$ converges weakly to the Dirac measure $\delta_{Z_s^*}$, $\mathbb{Q}_s$-a.s.,
where the law of the random variable $Z_s^*$ under the measure $\mathbb{Q}_s$ is given by $\pi_s^*$. 
\end{lemma}

\begin{proof}
This result has been recently established in \cite[Theorem 3.2]{GL16}. 
\end{proof}

The following result is proved in \cite[Lemma 3.5]{GL16}. 

\begin{lemma}\label{Lem_AbsoConti}
Assume condition \ref{Condi-IP}. Let $s \in I_{\mu}^{\circ}$. Then, 
there exists a constant $c_s>0$ such that for any $x \in \mathbb{P}^{d-1}$, 
it holds that $\mathbb{Q}_s^x \leq c_s \mathbb{Q}_s$. 
\end{lemma}

The assertion of Lemma \ref{Lem_AbsoConti} implies that 
the measure $\mathbb{Q}_s^x$ is absolutely continuous with respect to $\mathbb{Q}_s$. 
Using Lemmas \ref{Lem_DiracMea} and \ref{Lem_AbsoConti}, we are now in a position to prove Lemma \ref{LemConverThmQs}.

\begin{proof}[Proof of Lemma \ref{LemConverThmQs}]
By the Cartan decomposition of $G_n^*$, we have $G_n^* = k_n a_n k_n'$,
where $k_n, k_n' \in K$ and $a_n \in A^+$. 
By Lemma \ref{Lem_DiracMea},  
the probability measure $G_n^*  m^*$ converges weakly to the Dirac measure $\delta_{Z_s^*}$, $\mathbb{Q}_s$-a.s..
Since $m^*$ is a rotation invariant measure on $(\mathbb{P}^{d-1})^*$, 
it follows that $(k_n  a_n) m^*$ converges weakly to the random variable $Z_s^*$, $\mathbb{Q}_s$-a.s..
Taking into account that 
$a_n$ is a diagonal random matrix with decreasing diagonal entries,
we deduce that, as $n \to \infty$, we have 
$a_n  m^* \to \delta_{e_1^*}$,
$a_n^{2,2} / a_n^{1,1}  \to 0$ and $k_n  e_1^* \to Z_s^*$, $\mathbb{Q}_s$-a.s..
This concludes the proof of the assertion \eqref{ConverThmQs}. 
To show \eqref{ConverThmQs02}, using again the decomposition $G_n^* = k_n a_n k_n'$, 
it follows that for any $x = \bb R v \in \mathbb{P}^{d-1}$, 
\begin{align*} 
\frac{|G_n v|^2}{|v|^2}  = \frac{\langle  a_n k_n^* v,  a_n k_n^* v  \rangle}{|v|^2} 
=  \sum_{j=1}^d ( a_n^{j,j} )^2  \frac{| \langle k_n^* v, e_j^* \rangle |^2}{|v|^2}
=  \sum_{j=1}^d ( a_n^{j,j} )^2  \delta(k_n e_j^*, x)^2. 
\end{align*}
This, together with the fact that $\| G_n \| = a_n^{1,1}$, implies \eqref{ConverThmQs02}. 
Taking into account Lemma \ref{Lem_AbsoConti}, we see that 
the assertions \eqref{ConverThmQs} and \eqref{ConverThmQs02} remain valid with the measure $\mathbb{Q}_s$
replaced by $\mathbb{Q}_s^x$.
\end{proof}

\subsection{Asymptotics for the Iwasawa decomposition}   \label{Sec_Iwasawa}
In this subsection we study the asymptotics of the components 
in the Iwasawa decomposition of $G_n^*$ under the changed measure $\mathbb{Q}_s^x$. 
Denote by $L$ the group of lower triangular matrices with $1$ in the diagonal elements, 
by $A$ the group of diagonal matrices with strictly positive entries in the diagonal elements,
and as before by $K$ the group of orthogonal matrices.
The Iwasawa decomposition states that $GL(d, \mathbb{R}) = LAK$
and such decomposition is unique.
Hence, for the product $G_n^*$, there exist unique $L(G_n^*) \in L$,
$A(G_n^*) \in A$  and  $K(G_n^*) \in K$ such that $G_n^* = L(G_n^*) A(G_n^*) K(G_n^*)$. 
The following result shows that $L(G_n^*) e_1^*$ converges almost surely
under the measures $\mathbb{Q}_s$ and $\mathbb{Q}_s^x.$ 

\begin{lemma}\label{LemIwasaLim}
Let $s \in I_{\mu}^{\circ}$. 
Under condition \ref{Condi-IP}, for any $x \in \mathbb{P}^{d-1}$,
\begin{align*}
\lim_{n \to \infty} L(G_n^*) e_1^* 
= \frac{ Z_s^* }{ \langle Z_s^*, e_1 \rangle},   
\quad  \mathbb{Q}_s\mbox{-a.s.}   \  \mbox{and}  \quad  \mathbb{Q}_s^x\mbox{-a.s.}.
\end{align*}
where $Z_s^*$ is a random variable given by Lemma \ref{LemConverThmQs}. 
\end{lemma}

\begin{proof}
In view of Lemma \ref{Lem_AbsoConti}, it suffices to prove the assertion under the measure $\mathbb{Q}_s$. 
Using the Iwasawa decomposition $G_n^* = L(G_n^*) A(G_n^*) K(G_n^*)$
and noticing that $K(G_n^*)$ is an orthogonal matrix, it follows that 
\begin{align}\label{Pf_Reg_Eq1}
\frac{G_n^*  G_n e_1^* }{ | G_n e_1^* |^2 } 
= \frac{ L(G_n^*) A(G_n^*)^2  L(G_n^*)^* e_1^* }
   { | A(G_n^*) L(G_n^*)^* e_1^* |^2}
= L(G_n^*) e_1^*,  
\end{align}
where the second equality holds due to the fact that 
$A(G_n^*)^2 L(G_n^*)^* e_1^* =  | A(G_n^*) L(G_n^*) e_1^*|^2 e_1^*$.
By the Cartan decomposition of $G_n^*$ we have $G_n^* = k_n a_n k_n'$,
where $k_n, k_n'$ are two orthogonal matrices. 
Hence, 
for any $v \in (\mathbb{R}^d)^*$, 
\begin{align}\label{Pf_Reg_Equaa}
\langle G_n^* G_n e_1^*, v  \rangle  
& =   \langle  (a_n)^2 k_n^* e_1^*,  k_n^* v  \rangle   \nonumber\\
& =  ( a_n^{1,1} )^2 \langle k_n^* e_1^*, e_1 \rangle \langle e_1^*, k_n^* v \rangle 
       + O( a_n^{1,1} a_n^{2,2} )   \nonumber\\
& =  ( a_n^{1,1} )^2  \langle k_n^* e_1^*, e_1 \rangle \langle k_n e_1^*,  v  \rangle + O( a_n^{1,1} a_n^{2,2} ). 
\end{align}
Consequently, by \eqref{Pf_Reg_Eq1} and \eqref{Pf_Reg_Equaa} we obtain that $\mathbb{Q}_s$-a.s., 
\begin{align*}
\lim_{n \to \infty} \langle L(G_n^*) e_1^*, v \rangle 
= \lim_{n \to \infty} \frac{\langle G_n^* G_n e_1^*, v  \rangle}{\langle G_n^* G_n e_1^*, e_1^*  \rangle}
= \lim_{n \to \infty} \frac{\langle  k_n e_1^*, v  \rangle}{\langle  k_n e_1^*, e_1^*  \rangle}
= \frac{\langle Z_s^*, v  \rangle}{\langle Z_s^*, e_1^*  \rangle}, 
\end{align*}
where in the first equality we used \eqref{Pf_Reg_Eq1}, in the second one we used \eqref{Pf_Reg_Equaa}
and Lemma \ref{LemConverThmQs}, 
and in the last one we applied again Lemma \ref{LemConverThmQs}. 
Since $v \in (\mathbb{R}^d)^*$ is arbitrary, the proof of Lemma \ref{LemIwasaLim} is complete.
\end{proof}

For any $1 \leq k \leq d$,
we briefly recall the notion of exterior algebra $\wedge^k (\mathbb{R}^d)$ of the vector space $\mathbb{R}^d$.
The space $\wedge^k (\mathbb{R}^d)$ is endowed with the dual bracket $\langle \cdot, \cdot \rangle$
and the norm $| \cdot |$; we use the same notation as in $\mathbb{R}^d$
and the distinction should be clear from the context. 
The scalar product in $\wedge^k (\mathbb{R}^d)$ satisfies the following property:
for any $u_i$, $v_j \in \mathbb{R}^d$, $1 \leq i, j \leq d$, 
\begin{align*}
\langle  u_1 \wedge \cdots \wedge u_k,  v_1 \wedge \cdots \wedge v_k  \rangle
= \det( \langle u_i, v_j \rangle )_{1 \leq i, j \leq d}, 
\end{align*} 
where $\det( \langle u_i, v_j \rangle )_{1 \leq i, j \leq d}$ denotes the determinant of the associated matrix.
It is well known that 
$\{ e_{i_1} \wedge e_{i_2} \wedge \cdots \wedge e_{i_k}, 1 \leq i_1 < i_2 < \cdots < i_k \leq d \}$
forms a basis of $\wedge^k (\mathbb{R}^d)$, $1 \leq k \leq d$, 
and that $v_1 \wedge \cdots \wedge v_k$ is nonzero if and only if $v_1, \ldots, v_k$
are linearly independent in $\mathbb{R}^d$. 
For any $g \in GL(d, \mathbb{R})$ and $1 \leq k \leq d$, 
the exterior product $\wedge^k g$ of the matrix $g$
is defined as follows: for any $v_1, \ldots, v_k \in \mathbb{R}^d$, 
\begin{align*}
\wedge^k g ( v_1 \wedge \cdots \wedge v_k ) = g v_1 \wedge \cdots \wedge g v_k.
\end{align*}
Set $\| \wedge^k g \| = \sup \{ | (\wedge^k g) v |: v \in \wedge^k (\mathbb{R}^d), |v| =1 \}$. 
Since $\wedge^k (g g') = (\wedge^k g)  ( \wedge^k g' )$, 
it holds that $\| \wedge^k (g g') \| \leq \| \wedge^k g \| \| \wedge^k g' \|$ 
 for any $g, g' \in GL(d, \mathbb{R})$.  
Besides, if we denote by $a_{11}, \ldots, a_{dd}$ the singular values of the matrix $g$,
then $\| \wedge^k  g \| = a_{11} \ldots a_{kk}$. 
In particular, we have $\| \wedge^k g \| \leq \| g \|^k$.

The following lemma 
was proved in \cite{BL85}.
For any $g \in GL(d,\mathbb{R})$,  by the Iwasawa decomposition we have 
$g = L(g) A(g) K(g)$, where $L(g) \in L$, $A(g) \in A$ and $K(g) \in K$. 
In the sequel, we denote $N(g) = \max \{ \| g \|, \| g^{-1}\| \}$. 

\begin{lemma}\label{LemCruciIne}
For any integers $n, m \geq 0$, we have 
\begin{align*}
\left|  L(G_{n+m}^*) e_1^*  -  L(G_{n}^*) e_1^*  \right|
\leq \sum_{j=n}^{n+m-1} \frac{\|\wedge^2 G_j \|}{ | G_j e_1 |^2 }  e^{2 \log N(g_{j+1}^* )},  
\end{align*}
where we use the convention that $L(G_0) = 0$ 
and $\frac{\|\wedge^2 G_0 \|}{ | G_0 e_1 |^2 } = 0$. 
\end{lemma}
 
The following result shows the simplicity of the dominant Lyapunov exponent 
for $G_n$ under the changed measure $\mathbb{Q}_s^x$. 

\begin{lemma}\label{Lem_Lya_Meas}
Assume condition \ref{Condi-IP}. 
Let $s \in I_{\mu}^{\circ}$. 
Then, uniformly in $x = \bb R v \in \mathbb{P}^{d-1}$, 
\begin{align}\label{LLNMeas01}
 \lim_{n \to \infty} \frac{1}{n} \mathbb{E}_{\mathbb{Q}_s^x} (\sigma(G_n, x)) = \lambda_1(s), 
\end{align}
and  
\begin{align}\label{LLNMeas02}
 \lim_{n \to \infty} \frac{1}{n} \mathbb{E}_{ \mathbb{Q}_s^x }(\log \| \wedge^2 G_n \|) 
  = \lambda_1(s) + \lambda_2(s),
\end{align}
where $\lambda_1(s) > \lambda_2(s)$ are called the first two Lyapunov exponents of $G_n$ under
the measure $\mathbb{Q}_s^x$.  
\end{lemma}
The assertion \eqref{LLNMeas01} is proved in \cite[Theorem 3.10]{GL16}. 
The assertion \eqref{LLNMeas02} follows by combining Theorems 3.10 and 3.17 in \cite{GL16}. 
The fact that $\lambda_1(s) > \lambda_2(s)$ will play an essential role 
in the proof of  the H\"older regularity of the stationary measure $\pi_s$, see Proposition \ref{PropRegularity}. 

Using the simplicity of the Lyapunov exponent (see Lemma \ref{Lem_Lya_Meas})  
we can complement the convergence result  
in Lemma \ref{LemIwasaLim} 
by giving the rate of convergence. 
This result is not used in the proofs, but is of independent interest.

\begin{proposition}\label{Prop_L1Conver}
Assume condition \ref{Condi-IP}. Let $s \in I_{\mu}^{\circ}$. 
Then, there exist constants $\alpha, C > 0$ such that
uniformly in $x \in \mathbb{P}^{d-1}$ and $n \geq 1$, 
\begin{align}\label{Ine_L1Conv}
\mathbb{E}_{ \mathbb{Q}_s^x } 
 \left|  L(G_n^*) e_1^* - \frac{ Z_s^* }{ \langle Z_s^*, e_1^* \rangle}  \right|^{\alpha}  \leq e^{ -Cn }.   
\end{align}
Moreover, the assertion \eqref{Ine_L1Conv} remains valid when the measure $\mathbb{Q}_s^x$ is replaced by $\mathbb{Q}_s$.
\end{proposition}
The proof of Proposition \ref{Prop_L1Conver} is postponed to subsection \ref{Sec_Pf_Reg}.

By Jensen's inequality, the bound \eqref{Ine_L1Conv} implies that
there exists a constant $C>0$ such that uniformly in $x \in \mathbb{P}^{d-1}$,
\begin{align*}
\limsup_{ n \to \infty }  \frac{1}{n}  \mathbb{E}_{ \mathbb{Q}_s^x } 
\log \left|  L(G_n^*) e_1^* - \frac{ Z_s^* }{ \langle Z_s^*, e_1^* \rangle}  \right| \leq -C.
\end{align*}
When $s = 0$, it was proved in \cite{BL85} that $C = \lambda_1(0) - \lambda_2(0)$. 
We conjecture that $C = \lambda_1(s) - \lambda_2(s)$ also for $s>0$, but the proof eluded us.


\subsection{Proof of Propositions \ref{PropRegularity} and \ref{Prop_L1Conver}}\label{Sec_Pf_Reg}
With the results established in subsections \ref{Sec_Cartan} and \ref{Sec_Iwasawa}, 
we are well equipped to prove Propositions \ref{PropRegularity} and \ref{Prop_L1Conver}.


\begin{proof}[\textit{Proof of Proposition \ref{PropRegularity}}]

Since $r_s$ is bounded away from infinity and $0$ uniformly on $\mathbb{P}^{d-1}$,
it suffices to establish \eqref{RegularityIne00} and \eqref{RegularityIne} for the stationary measure $\pi_s$.

Define the function $\rho: GL(d, \mathbb{R}) \times \mathbb{P}^{d-1} \to \mathbb{R}$ as follows: 
for $g \in GL(d, \mathbb{R})$ and $x \in \mathbb{P}^{d-1}$, 
\begin{align*}
\rho(g, x) = \log \| \wedge^2 g \| - 2 \log | gx |. 
\end{align*}
It is clear that 
\begin{align*} 
\mathbb{E}_{ \mathbb{Q}_s^x }  \rho(G_n, x)  
=  \mathbb{E}_{ \mathbb{Q}_s^x } \big( \log \| \wedge^2 G_n \| \big) 
  - 2 \mathbb{E}_{ \mathbb{Q}_s^x }  \big( \log |G_n x| \big). 
\end{align*}
By Lemma \ref{Lem_Lya_Meas}, we see that  
\begin{align*}
\lim_{ n \to \infty } 
\frac{1}{n} \sup_{x \in \mathbb{P}^{d-1} } 
\mathbb{E}_{ \mathbb{Q}_s^x }  \rho(G_n, x )  <0, 
\end{align*}
which clearly implies that, for large enough $n$, 
\begin{align}\label{Pf-RegBoundNeg}
 \sup_{x \in \mathbb{P}^{d-1} }  
 \mathbb{E}_{ \mathbb{Q}_s^x }  \rho(G_n, x )  <0.
\end{align}
We claim that there exists a constant $\alpha > 0$ such that 
\begin{align}\label{Pf-RegIne-b}
\limsup_{n \to \infty} \frac{1}{n} \log 
\sup_{x \in \mathbb{P}^{d-1} } 
\mathbb{E}_{ \mathbb{Q}_s^x } 
  \frac{ \| \wedge^2 G_n \|^{\alpha} }{|G_n x |^{2 \alpha} }
 < 0.
\end{align}
To prove \eqref{Pf-RegIne-b}, we denote
$a_n = \log \big( \sup_{x \in \mathbb{P}^{d-1}} 
\mathbb{E}_{ \mathbb{Q}_s^x } \big( e^{ \alpha \rho(G_n, x) } \big)  \big)$,
for sufficiently small constant $\alpha >0$.
Using the cocycle property \eqref{cocycle01} and the fact that $\rho$ is subadditive, 
we get that for any $n, m \geq 1$, 
\begin{align*}
&  \mathbb{E}_{ \mathbb{Q}_s^x } \big( e^{ \alpha \rho( G_{n+m}, x ) } \big)  \nonumber\\
& \leq  \mathbb{E} \Big(  q_m^s(x, G_m) e^{ \alpha \rho( G_m, x ) }  \Big)  
  \mathbb{E} \Big(  q_n^s(x, g_{m+1} \ldots g_{m+n} ) 
      e^{ \alpha \rho( g_{m+n} \ldots g_{m+1}, x ) }  \Big)   \nonumber\\
& =  \mathbb{E}_{ \mathbb{Q}_s^x } \big( e^{ \alpha \rho( G_m, x ) } \big)
   \mathbb{E}_{ \mathbb{Q}_s^x } \big( e^{ \alpha \rho( G_n, x ) } \big).
\end{align*}
Taking supremum on both sides of the above inequality, 
we see that the sequence $(a_n)_{n\geq 1}$ satisfies the subadditive property: $a_{n+m} \leq a_m + a_n$. 
Hence we get $a = \lim_{n \to \infty} \frac{a_n}{n} = \inf_{n\geq 1} \frac{a_n}{n}.$
To show that $a<0$, it suffices to check that there exists some integer $p \geq 1$ such that 
\begin{align}\label{Pf-RegBound01}
\sup_{x \in \mathbb{P}^{d-1} } 
 \mathbb{E}_{ \mathbb{Q}_s^x } \big( e^{ \alpha \rho( G_{p}, x ) } \big) <1.
\end{align}
We proceed to verify \eqref{Pf-RegBound01}. 
Using the fact 
that $\sup_{x} |\rho(g, x)| \leq 4 \log N(g)$
and the basic inequality $e^y \leq 1 + y + \frac{y^2}{2} e^{|y|}$, $y \in \mathbb{R}$, 
we obtain
\begin{align}\label{Pf_Regu_Inver_aa}
\mathbb{E}_{ \mathbb{Q}_s^x } \big( e^{ \alpha \rho( G_{p}, x ) } \big)
\leq   1 + \alpha \mathbb{E}_{ \mathbb{Q}_s^x } \big(  \rho( G_{p}, x )  \big)  
   + \frac{ {\alpha}^2}{2} \mathbb{E}_{ \mathbb{Q}_s^x } 
     \Big( 16 \log^2 N( G_{p} ) e^{ 4 \alpha \log N( G_{p} ) }  \Big). 
\end{align}
The second term on the right-hand side of \eqref{Pf_Regu_Inver_aa}
is strictly negative by using the bound \eqref{Pf-RegBoundNeg} and taking large enough $p$. 
The third term is finite due to the moment condition \ref{Condi-TwoExp}. 
Consequently, taking $\alpha > 0$ small enough, we obtain the inequality \eqref{Pf-RegBound01} 
and thus the desired assertion \eqref{Pf-RegIne-b} follows.

Since the bound \eqref{Pf-RegIne-b} holds uniformly in $x \in \mathbb{P}^{d-1}$,
taking into account that $\mathbb{Q}_s = \int_{\mathbb{P}^{d-1}} \mathbb{Q}_s^x \pi_s(dx)$, 
it follows that there exist constants $C>0$ and $0< r < 1$ such that 
\begin{align}\label{RegIneq001}
\mathbb{E}_{\mathbb{Q}_s} \frac{ \|\wedge^2 G_n \|^{\alpha} }{|G_n x|^{2 \alpha}  } \leq C r^n. 
\end{align}
Using Lemma \ref{LemIwasaLim}, Fatou's lemma and the fact that $|Z_s^*| = 1$, 
we obtain that for sufficiently small constant $\alpha >0$, 
\begin{align}\label{Pf-Regu-Inv01}
\mathbb{E}_{\mathbb{Q}_s} \frac{ 1 }{ | \langle Z_s^*, e_1^* \rangle |^{\alpha} } 
\leq  \liminf_{n \to \infty} 
\mathbb{E}_{\mathbb{Q}_s}  \big(  | L(G_n^*) e_1^* |^{\alpha} \big).
\end{align}
From Lemma \ref{LemCruciIne} with $n=0$, it follows that
\begin{align*}
| L(G_n^*) e_1^* |^{\alpha} 
\leq \sum_{ j=1 }^{\infty} \frac{ \|\wedge^2 G_j \|^{c} }{ | G_j e_1|^{2 \alpha} }  e^{2 \alpha \log N(g_{j+1}^* )}.  
\end{align*}
Notice that $G_j$ and $g_{j+1}^*$ are not independent under the measure $\mathbb{Q}_s$. 
Using Fubini's theorem, H\"{o}lder's inequality and the bound \eqref{RegIneq001},  we get
\begin{align*}
\mathbb{E}_{\mathbb{Q}_s}  \big(  | L(G_n^*) e_1^* |^{\alpha}  \big) 
&  \leq     \sum_{ j=1 }^{\infty} 
  \left[  \mathbb{E}_{\mathbb{Q}_s} \frac{\|\wedge^2 G_j \|^{2 \alpha} }{ | G_j e_1|^{4 \alpha} }   \right]^{1/2}
  \left[  \mathbb{E}_{\mathbb{Q}_s} e^{4 \alpha \log N(g_{j+1}^* )}  \right]^{1/2}  \nonumber\\ 
&  \leq    C  \mathbb{E}_{\mathbb{Q}_s} ( e^{4 \alpha \log N(g_{1}^* )} )  \sum_{j=1}^{\infty} r^j < + \infty.  
\end{align*} 
Combining this with \eqref{Pf-Regu-Inv01} leads to 
$\mathbb{E}_{\mathbb{Q}_s} \frac{ 1 }{ | \langle Z_s^*, e_1^* \rangle |^{\alpha} } < + \infty.$
Note that for any $y \in (\mathbb{P}^{d-1})^*$, we can choose an orthogonal matrix $k$ such that 
$k e_1^* = y$. If we replace $g_i^*$ by $k^{-1} g_i^* k$, then it is easy to see that 
$G_n^*$ is replaced by $k^{-1} G_n^* k$.
Moreover, in view of Lemma \ref{Lem_DiracMea}, the random variable $Z_s^*$ is replaced by $k^{-1} Z_s^*$. 
Since the bound \eqref{RegIneq001} holds uniformly in $x \in \mathbb{P}^{d-1}$, it follows that 
\begin{align*}
\mathbb{E}_{\mathbb{Q}_s} \frac{ 1 }{ | \langle k^{-1} Z_s^*, e_1^* \rangle |^{\alpha} } 
\leq C  \mathbb{E}_{\mathbb{Q}_s} \big( e^{4 \alpha \log N( k^{-1} g_{1}^* k )} \big) 
    \sum_{j=1}^{\infty} r^j < + \infty. 
\end{align*}
Observe that 
$N( k^{-1} g_{1}^* k ) = N( g_{1}^*)$ 
and $\langle k^{-1} Z_s^*, e_1 \rangle = \langle Z_s^*, y \rangle$.
Therefore, for any $s \in I_{\mu}^{\circ}$, there exists a constant $\alpha > 0$ such that 
\begin{align*}
\sup_{x \in \mathbb P^{d-1}} \int_{(\mathbb P^{d-1})^* } \frac{1}{ \delta(y, x)^{\alpha} } \pi_s^*(dy)
=  \sup_{y \in (\mathbb P^{d-1})^* }  \mathbb{E}_{\mathbb{Q}_s} \frac{ 1 }{ | \langle Z_s^*, y \rangle |^{\alpha} }
< + \infty. 
\end{align*}
This implies that there exists a constant $C>0$ such that for any $0< t < 1$, 
uniformly in $x \in \mathbb P^{d-1}$, 
\begin{align*}
 \pi_s^* \left( \left\{ y \in (\bb P^{d-1})^*:  \delta(y, x) \leq t  \right\}  \right)  
  \leq    t^{\alpha}  \int_{ (\mathbb P^{d-1})^* } \frac{1}{ \delta(y, x)^{\alpha} } \pi_s^*(dy)
\leq C t^{\alpha}. 
\end{align*}
The proof of Proposition \ref{PropRegularity} is complete. 
\end{proof}

\begin{proof}[Proof of Proposition \ref{Prop_L1Conver}]

In view of Lemma \ref{Lem_AbsoConti}, it suffices to prove the assertion of the proposition with $\mathbb{Q}_s$ instead of $\mathbb{Q}_s^x$, 
i.e.\ we show that there exist constants $\alpha, C>0$ such that for all $n \geq 1$, 
\begin{align}\label{Ine_L1Conv02}
\mathbb{E}_{ \mathbb{Q}_s } 
 \left|  L(G_n^*) e_1 - \frac{ Z_s^* }{ \langle Z_s^*, e_1 \rangle}  \right|^{\alpha}  < e^{ -Cn }.   
\end{align}
Using Lemma \ref{LemCruciIne} and H\"{o}lder's inequality, 
for sufficiently small constant $\alpha >0$ and for any $n, m \geq 1$, we get 
\begin{align*}
& \mathbb{E}_{ \mathbb{Q}_s } \big|  L(G_{n+m}^*) e_1^*  -  L(G_n^*) e_1^*  \big|^{\alpha}  \nonumber \\ 
& \leq   \sum_{j=n}^{ n+m-1 } 
  \left[  \mathbb{E}_{\mathbb{Q}_s} \frac{\|\wedge^2 G_j \|^{2 \alpha} }{ | G_j e_1|^{4 \alpha} }   \right]^{1/2}
  \left[  \mathbb{E}_{\mathbb{Q}_s} e^{4 \alpha \log N(g_{j+1}^* )}  \right]^{1/2}   \nonumber\\
& \leq \  C  \sum_{j=n}^{ n+m-1 } 
  \left[  \mathbb{E}_{\mathbb{Q}_s} \frac{\|\wedge^2 G_j \|^{2 \alpha} }{ | G_j e_1|^{4 \alpha} }   \right]^{1/2}, 
\end{align*}
where the last inequality holds due to the moment condition \ref{Condi-TwoExp}. 
By the Fatou lemma, 
taking the limit as $m\to\infty$, we see that
\begin{align*}
\mathbb{E}_{ \mathbb{Q}_s } 
 \left|  L(G_n^*) e_1^* - \frac{ Z_s^* }{ \langle Z_s^*, e_1^* \rangle}  \right|^{\alpha}  
\leq  C \sum_{j=n}^{\infty} 
  \left[  \mathbb{E}_{\mathbb{Q}_s} \frac{\|\wedge^2 G_j \|^{2 \alpha} }
  { | G_j e_1|^{4 \alpha} }   \right]^{1/2}
 \leq C e^{ - Cn},
\end{align*}
where the last inequality holds due to the bound \eqref{RegIneq001}. 
\end{proof}


\subsection{Proofs of Propositions \ref{Prop_Regu_Strong01}, \ref{Prop_Regu_Strong02}, \ref{LLN_CLT_Entry} and Theorem \ref{Thm_Coeff_BRLD_changedMea02}}

We first establish Propositions \ref{Prop_Regu_Strong01} and \ref{Prop_Regu_Strong02}
based on Propositions \ref{PropRegularity} and \ref{PropRegu02}, respectively, 
together with the fact that, under the changed measure $\bb Q_s^x$, 
the Markov chain $(X_n^x)_{n \geq 0}$ converges exponentially fast to the stationary measure $\pi_s$. 

\begin{proof}[Proof of Propositions \ref{Prop_Regu_Strong01} and \ref{Prop_Regu_Strong02}]
For any $1 \leq k \leq n$ and $\ee >0$, 
denote $\chi_k(u) := \mathbbm{1}_{\{u \in (-\infty, -\ee k] \}}$
and $\chi_{k, \ee_1}^+(u) = \sup_{u' \in \mathbb{B}_{\ee_1}(u)} \chi_k(u')$ for $\ee_1 > 0$. 
In the same way as in \eqref{Pf_LD_SmoothIneHolder01}, 
we have the following smoothing inequality: 
\begin{align} \label{Pf_Regu_Smooth_dd}
\chi_k(u) \leq 
(\chi_{k, \ee_1}^+ * \bar{\rho}_{\ee_1})(u)
= : \tilde\chi_k(u),  \quad  u \in \mathbb{R},
\end{align}
where $\bar{\rho}_{\ee_1}$ is the density function given in \eqref{Pf_LD_SmoothIneHolder01}. 
For brevity, we denote 
\begin{align} \label{Pf_regu_varphi_dd}
\varphi_{k,\ee_1}^y(x) =  \tilde\chi_k(\log \delta(y, x)),  \quad  x \in \bb{P}^{d-1}. 
\end{align}  
By \eqref{Pf_Regu_Smooth_dd} and \eqref{Pf_regu_varphi_dd}, it follows that 
\begin{align*} 
\bb Q_s^x \Big( \delta (y, G_n x) \leq e^{- \ee k}  \Big) 
& \leq  \bb E_{\bb Q_s^x} \big[ \varphi_{k,\ee_1}^y (G_n x) \big]  \nonumber\\
& \leq  \left| \bb E_{\bb Q_s^x} \big[ \varphi_{k,\ee_1}^y (G_n x) \big] - \pi_s(\varphi_{k,\ee_1}^y) \right|
   + \pi_s(\varphi_{k,\ee_1}^y). 
\end{align*}
For the first term, note first that 
$\| \varphi_{k,\ee_1}^y \|_{\gamma} \leq \frac{ e^{ \ee \gamma k} }{ ( 1 - e^{-2\ee_1} )^{\gamma} }.$
Using \eqref{equcontin Q s limit} and taking $\gamma >0$ sufficiently small, 
we get that for any $1 \leq k \leq n$, 
\begin{align}\label{Pf_Regu_I1_f}
\left| \bb E_{\bb Q_s^x} \big[ \varphi_{k,\ee_1}^y (G_n x) \big] - \pi_s(\varphi_{k,\ee_1}^y) \right|
 \leq C e^{-cn} \| \varphi_{k,\ee_1}^y \|_{\gamma} \leq  C e^{-cn /2}. 
\end{align}
For the second term, using the fact that
$\tilde\chi_k(u) \leq \chi_{k, 2\ee_1}^+(u) = \mathbbm{1}_{\{u \in (-\infty, -\ee k + 2 \ee_1] \}}$, 
and applying Propositions \ref{PropRegularity} and \ref{PropRegu02} 
(respectively for $s \in I_{\mu}^{\circ}$ and $s \in (-s_0, 0)$), 
we obtain that there exist constants $c, C >0$ such that 
\begin{align}\label{Pf_Regu_I2_f}
\pi_s(\varphi_{k,\ee_1}^y) 
\leq \pi_s \left(  x \in \bb P^{d-1}: \delta(y,x) \in [0, e^{- \ee k + 2 \ee_1}] \right)
\leq C e^{-ck}. 
\end{align}
Putting together \eqref{Pf_Regu_I1_f} and \eqref{Pf_Regu_I2_f}, 
we conclude the proof of Propositions \ref{Prop_Regu_Strong01} and \ref{Prop_Regu_Strong02}. 
\end{proof}

Using Propositions \ref{Prop_Regu_Strong01} and \ref{Prop_Regu_Strong02}, we are now in a position to 
establish Proposition \ref{LLN_CLT_Entry} 
on the SLLN and the CLT for the coefficients $\langle f, G_n v \rangle$
under the measure $\mathbb{Q}_s^x$. 

\begin{proof}[Proof of Proposition \ref{LLN_CLT_Entry}]
(1) We first prove \eqref{SLLN_Entry}. 
By Proposition \ref{Prop_Regu_Strong01} and Borel-Cantelli's lemma, 
we get that for any $\varepsilon > 0$ and $s \in I_{\mu}^{\circ}$, 
uniformly in $f \in (\bb R^d)^*$ and $v \in \bb R^d$ with $|f| = |v| = 1$, 
\begin{align*}
\liminf_{ n \to \infty } \frac{1}{n} \log \frac{ | \langle f, G_n v \rangle | }{ | G_n v| } \geq - \varepsilon,
\quad  \mathbb{Q}_s^x \mbox{-a.s..} 
\end{align*}
Since $\varepsilon > 0$ can be arbitrary small, this together with \eqref{SLLN_Gnx} implies the desired lower bound:
uniformly in $f \in (\bb R^d)^*$ and $v \in \bb R^d$ with $|f| = |v| = 1$, 
\begin{align*}
\liminf_{ n \to \infty } \frac{1}{n} \log | \langle f, G_n v \rangle |  \geq \Lambda'(s), 
\quad  \mathbb{Q}_s^x \mbox{-a.s..}
\end{align*}
The upper bound follows easily from \eqref{SLLN_Gnx} and the fact that 
$\log | \langle f, G_n v \rangle | \leq \log | G_n v|$. 
Hence \eqref{SLLN_Entry} holds. 

We next prove \eqref{CLT_Entry}. 
Using Proposition \ref{Prop_Regu_Strong01} with $k = \sqrt{n}$, 
we get the following convergence in probability: for any $\varepsilon>0$, 
uniformly in $f \in (\bb R^d)^*$ and $v \in \bb R^d$ with $|f| = |v| = 1$, 
\begin{align*}
\lim_{ n \to \infty }  \mathbb{Q}_s^x 
\left( \frac{ \log |G_n v| - \log | \langle f, G_n v \rangle | }{ \sigma_s \sqrt{n} } \geq \varepsilon \right) 
  = 0. 
\end{align*}
This yields \eqref{CLT_Entry} using \eqref{CLT_Cocycle01} together with Slutsky's lemma. 

(2) The proof of part (2) can be carried out in an analogous way using Proposition \ref{Prop_Regu_Strong02},
the SLLN and the CLT for the norm cocycle $\log |G_n v|$
under the changed measure $\mathbb{Q}_s^x$ when $s<0$ established in \cite{XGL19b}. 
\end{proof}

Using Propositions \ref{Prop_Regu_Strong01} and \ref{Prop_Regu_Strong02},
we are able to prove Theorem \ref{Thm_Coeff_BRLD_changedMea02}. 

\begin{proof}[Proof of Theorem \ref{Thm_Coeff_BRLD_changedMea02}]
In a similar way as in the proof of \eqref{Pf_BRP_Changed_ff}, 
one can verify that for any $s < t$ with $s \in (-s_0, 0] \cup I^{\circ}_{\mu}$
and $t \in K_{\mu} \subset (-s_0, s)$,
\begin{align*}
&  \mathbb{E}_{\bb Q_s^x} 
  \Big[ \varphi(G_n x) \psi \big( \log |\langle f, G_n v \rangle| - nq_t \big) \Big] 
 =   \frac{ \kappa^n(t) r_t(x) }{\kappa^n(s) r_s(x)}  e^{ (s-t) n q_t}  \times  \nonumber\\
& \quad     \mathbb{E}_{\bb Q_t^x} 
 \left[  (\varphi r_s r_t^{-1})(G_n x)  e^{(s - t) (\log |G_n v| - n q_t)} 
     \psi \big( \log |\langle f, G_n v \rangle| - nq_t \big)  \right]. 
\end{align*}
Recalling that $\Lambda^*(q_s) = sq_s - \Lambda(s)$, $\Lambda^*(q_t) = tq_t - \Lambda(t)$, 
$\Lambda(s) = \log \kappa(s)$ and $\Lambda(t) = \log \kappa(t)$, 
we have 
\begin{align*}
\frac{ \kappa^n(t) }{\kappa^n(s)}  e^{ (s-t) n q_t}
= \exp  \{  -n(\Lambda^*(q_t) - \Lambda^*(q_s) - s(q_t -q_s)) \}. 
\end{align*}
Hence, to prove Theorem \ref{Thm_Coeff_BRLD_changedMea02}, we are led to handle
\begin{align*} 
J : =  \sigma_t \sqrt{2\pi n}  
  \mathbb{E}_{\bb Q_t^x} 
 \left[  (\varphi r_s r_t^{-1})(G_n x)
     e^{(s - t) (\log |G_n v| - n q_t)} 
     \psi \big( \log |\langle f, G_n v \rangle| - nq_t \big)  \right].  
\end{align*}
For simplicity, denote
\begin{align*}
T_n^v: = \log |G_n v| - n q_t,  \qquad  Y_n^{x,y}: = \log \delta(y, G_n x).  
\end{align*}
For any fixed small constant $0< \eta <1$, set $I_k: = (-\eta k, -\eta(k-1)]$,  $k \geq 1$. 
Take a sufficiently large constant $C_1 > 0$ and let $M_n:= \floor{ C_1 \log n }$. 
Then, 
\begin{align}\label{Pf_LD_Low_decom_f}
J = J_1 + J_2, 
\end{align}
where  
\begin{align*}
& J_1 : =  \sigma_t \sqrt{2\pi n}  \mathbb{E}_{ \bb Q_t^x } 
\left[ (\varphi r_s r_t^{-1})(G_n x)  e^{(s-t) T_n^v}
\psi \big( T_{n}^v + Y_n^{x,y} \big) \mathbbm{1}_{\{Y_n^{x, y} \leq -\eta M_n \}} \right],  
    \nonumber\\
& J_2 : =  \sigma_t \sqrt{2\pi n}  \sum_{k =1}^{M_n}  
\mathbb{E}_{ \bb Q_t^x } 
\left[ (\varphi r_s r_t^{-1})(G_n x)  e^{(s-t) T_n^v}
\psi \big( T_{n}^v + Y_n^{x,y}  \big) \mathbbm{1}_{\{Y_n^{x,y} \in I_k \}} \right]. 
\end{align*}
For $J_1$, 
since the function $u \mapsto e^{-s' u} \psi(u)$ is directly Riemann integrable on $\mathbb{R}$ 
for any $s' \in K_{\mu}^{\epsilon}$,
we see that the function $u \mapsto e^{(s-t) u} \psi(u)$ is bounded on $\mathbb{R}$
and so there exists a constant $C >0$ such that for all $s \in (-s_0, 0] \cup I^{\circ}_{\mu}$, 
\begin{align*}
e^{(s-t) T_n^v} \psi( T_{n}^v + Y_n^{x,y} )  \leq  C e^{ (t-s) Y_n^{x,y} }. 
\end{align*}
Hence, using Propositions \ref{Prop_Regu_Strong01} and \ref{Prop_Regu_Strong02}, 
we get that as $n \to \infty$, 
\begin{align*}
J_1  \leq  C \sqrt{n}  
 \bb Q_t^x \Big( \log \delta(y, G_n x) \leq -\eta \floor{C_1 \log n} \Big)  
 \leq  C \sqrt{n}  \,  e^{-c_{\eta} \floor{C_1 \log n} }  \to 0.  
\end{align*}
For $J_2$, one can follow the proof of Theorem \ref{Thm-Posi-Neg-sBRP} to obtain
that as $n \to \infty$, 
uniformly in $f \in (\bb R^d)^*$ and $v \in \bb R^d$ with $|f| = |v| = 1$,  
\begin{align*}
J_2 = \int_{ \bb P^{d-1} } \varphi(x) \delta(y, x)^t \,  \frac{r_s(x)}{r_t(x)} \pi_t(dx)
       \int_{\mathbb{R}} e^{- (t-s) u} \psi(u) du + o(1). 
\end{align*}
This ends the proof of Theorem \ref{Thm_Coeff_BRLD_changedMea02}. 
\end{proof}


\end{document}